\documentclass[10pt]{article}

\usepackage{etex}
\usepackage[margin={1.2in}]{geometry}
\usepackage[titletoc]{appendix}

\usepackage{amsmath,amssymb, blkarray,booktabs,bigstrut}

\usepackage{enumerate}
\makeatletter
\newcommand{\mylabel}[2]{#2\def\@currentlabel{#2}\label{#1}}
\makeatother

\usepackage{tikz,eso-pic}
\usepackage{tikz-3dplot}

\usepackage{tikz}
\usetikzlibrary{matrix}

\usepackage{eucal}
\usepackage{mathrsfs}
\usepackage{stmaryrd}

\usepackage{hyperref} 

\usepackage{rotating}

\usepackage{longtable}

\usepackage[sc]{mathpazo}
\usepackage{courier}

\usepackage[utf8]{inputenc}
\usepackage[T1]{fontenc}

\usepackage[english]{babel}
\usepackage{blindtext}

\usepackage{amsmath}

\usepackage{microtype}

\usepackage{amsmath}
\usepackage{amssymb}
\usepackage{amsthm}
\usepackage{color}
\usepackage{latexsym,extarrows}

\usepackage[all]{xy}
\usepackage{graphicx}
\usepackage{subcaption}

\usepackage[stable,multiple]{footmisc}

\RequirePackage[font=small,format=plain,labelfont=bf,textfont=it]{caption}
\makeatletter
\newsavebox{\@brx}
\newcommand{\llangle}[1][]{\savebox{\@brx}{\(\m@th{#1\langle}\)}%
	\mathopen{\copy\@brx\kern-0.5\wd\@brx\usebox{\@brx}}}
\newcommand{\rrangle}[1][]{\savebox{\@brx}{\(\m@th{#1\rangle}\)}%
	\mathclose{\copy\@brx\kern-0.5\wd\@brx\usebox{\@brx}}}
\makeatother

\usepackage[utf8]{inputenc}

\newcommand\xqed[1]{%
	\leavevmode\unskip\penalty9999 \hbox{}\nobreak\hfill
	\quad\hbox{#1}}
\newcommand\xxqed{\xqed{$\triangle$}}

\def\e#1\e{\begin{equation}#1\end{equation}}
\def\ea#1\ea{\begin{align}#1\end{align}}

\theoremstyle{plain}
\newtheorem{thm}{Theorem}[section]
\newtheorem{thmx}{Theorem}

\newtheorem{lem}[thm]{Lemma}
\newtheorem{prop}[thm]{Proposition}
\newtheorem{assumption}[thm]{Assumption}
\newtheorem{cor}[thm]{Corollary}
\theoremstyle{definition}
\newtheorem{dfn}[thm]{Definition}
\newtheorem{ex}[thm]{Example}
\newtheorem{rem}[thm]{Remark}

\newtheorem{cons}[thm]{Construction}

\newcommand{\op}{\operatorname}

\newcommand{\im}{\op{im}}

\newcommand{\pii}{2\pi \mathbf{i}}

\def\Xint#1{\mathchoice
	{\XXint\displaystyle\textstyle{#1}}%
	{\XXint\textstyle\scriptstyle{#1}}%
	{\XXint\scriptstyle\scriptscriptstyle{#1}}%
	{\XXint\scriptscriptstyle\scriptscriptstyle{#1}}%
	\!\int}
\def\XXint#1#2#3{{\setbox0=\hbox{$#1{#2#3}{\int}$}
		\vcenter{\hbox{$#2#3$}}\kern-.5\wd0}}

\def\dashint{\Xint-}

\numberwithin{equation}{section}

\makeatletter
\newcommand{\subjclass}[2][2020]{%
	\let\@oldtitle\@title%
	\gdef\@title{\@oldtitle\footnotetext{#1 \emph{Mathematics Subject Classification.} #2}}%
}
\newcommand{\keywords}[1]{%
	\let\@@oldtitle\@title%
	\gdef\@title{\@@oldtitle\footnotetext{\emph{Key words and phrases.} #1.}}%
}
\makeatother

\makeatletter
\let\orig@afterheading\@afterheading
\def\@afterheading{%
	\@afterindenttrue
	\orig@afterheading}
\makeatother

\begin{document}
	\title{\bf 
			Regularized Integrals on Configuration Spaces of Riemann Surfaces and Cohomological Pairings
	}
	\author{Jie Zhou}
	\date{}
	\maketitle
		\begin{abstract}
We extend the notion of regularized integrals introduced by Li-Zhou that aims to assign finite values to divergent integrals on configuration spaces of Riemann surfaces. 
We then give cohomological formulations for the extended notion using the tools of current cohomology and mixed Hodge structures.
We also provide practical ways of constructing representatives of the corresponding cohomology classes in terms of smooth differential forms.
	\end{abstract}

	\setcounter{tocdepth}{2} \tableofcontents

	\section{Introduction}
	
	Regularizing divergent integrals is a key step in the mathematical understanding of quantum field theories via their correlation functions.
	Specializing to 2d chiral conformal field theories (CFTs) with Lagrangian descriptions, it provides an alternative approach to formulate the rich aspects of the physics theories, besides the one provided by vertex algebras. The tools and theories introduced in the procedure are often useful on their own from the mathematical point of view. \\
		
The divergent integrals that we are interested in throughout  this work take the form 
\begin{equation}\label{eqnsettingofdivergentintegrals}
	\int_{X}\omega~ \,,
\end{equation}
 where 
\begin{itemize}
	\item
	$X$ is a smooth compact complex manifold;
	\item
$D$ is a reduced effective divisor on $X$ with  smooth irreducible components;
\item  $\omega$ is a differential form on $X$, smooth everywhere except for possibly holomorphic poles along $D$.
\end{itemize}
What attracts us most  are the following situations.
\begin{enumerate}
	\item[\mylabel{caseC}{\textbf{(R)}}]
	Let $C$ be a closed Riemann surface.
	For $n\geq 2$, let $X=C^{n}$ and $D$ be a subdivisor of the big diagonal $\Delta=\bigcup_{i\neq j}\Delta_{ij}$, where 
	\[
	\Delta_{ij}=\Big\{(p_1,\cdots, p_{n})\in C^{n}|~p_{i}=p_{j}\Big\}\,.
	\]
	Then $\mathrm{Conf}_{n}(C)=X-\Delta$ is the configuration space of $n$ points on $C$.
	For $n=1$, let $X$ be a Riemann surface $C$ and $D$ consist of a finite set of points.
	
		\item[\mylabel{caseHA}{\textbf{(HA)}}] 
	Let $X$ be a nonsingular projective algebraic variety and $D$ be a smooth hypersurface arrangement. Here a smooth hypersurface arrangement in $X$ of dimension $n$ is a union of smooth hypersurfaces  
	$D=\bigcup_{a=1}^{N}D_{a}$ that locally looks like a union of hyperplanes in $\mathbb{C}^{n}$, such that 
	for any ordered tuple $I=(a_1<a_2<\cdots< a_{m})$, $D_{I}:=\bigcap_{a\in I}D_{a}$ is connected and smooth.

	A special case is when
	$D$ is a simple normal crossing divisor which we shall refer to as the $(\textbf{NCD})$ case. An even more special case is when $D$ is a smooth hypersurface which we refer to as the $(\textbf{H})$ case.
	
\end{enumerate}
In each case, let $j: U=X- D\rightarrow X$ be the open embedding.\\

An analytic approach
in regularizing such divergent integrals in case \ref{caseC} is introduced in \cite{Li:2020regularized} and further developed in \cite{Li:2022regularized}.
See the references above and \cite{Zhou:GW} for motivation and applications of this approach of regularization.

Let us  briefly review this analytic approach. We first introduce the following definitions.
\begin{dfn}
\begin{itemize}
	\item Let $\Omega_{X}^{p}$ be the sheaf of holomorphic $p$ forms on $X$ and 
	$H^{0}(X,\Omega_{X}^{p})$ be the corresponding space of global sections over $X$.
	\item Let $A^{p,q}_{X}$ (resp. $A^{p,q}_{X}(\star D)$) be the space of $(p,q)$ forms on $X$ (resp. the space of $(p,q)$ forms that are 
	smooth on $X$ except for holomorphic poles along $D$).
	Let $A^{n}_{X}=\bigoplus_{p+q=n}A^{p,q}_{X},A^{n}_{X}(\star D)=\bigoplus_{p+q=n}A^{p,q}_{X}(\star D)$.
	\item 
	Let
	also $A^{p,q}_{X}(\log  D)$ be the space of $(p,q)$ forms that are 
	smooth on $X$ except for logarithmic poles\footnote{Note that the notion of logarithmic forms we use here 
		is  different from the one introduced by Saito \cite{Saito:1980} 	and coincides with it in the normal crossing divisor case.
	} along $D$; locally
	it is generated by the usual smooth forms and ${ds_a/ s_a}$,
	where the $s_a$'s are the local defining equations of the irreducible components of $D$.
	Let $A^{n}_{X}(\log D)=\bigoplus_{p+q=n}A^{p,q}_{X}(\log D)$.
	\item
		An element 	$\omega\in A^{n,n}_{X}(\star D )$ is called 
	\emph{factorizable} if it lies in the image of the following map:
	\begin{equation}\label{eqndfnfactorizableforms}
	\wedge:\quad 
	(A^{n,0}_{X}(\star D )\cap \mathrm{ker}\,\bar{\partial})
	\times
	(A^{0,n}_{X}\cap \mathrm{ker}\,\partial)
	\rightarrow 
	A^{n,n}_{X}(\star D )\,.			
	\end{equation}
	Denote  the subspace of factorizable forms by $\mathfrak{F}\subseteq A^{n,n}_{X}(\star D )$.	
	
\end{itemize}
\end{dfn}

The regularized  integral $\dashint_{X} \omega$  for case \ref{caseC} is then defined for 
$\omega\in A^{n,n}_{X}(\star D )$ as follows.
By reduction of pole, one can find a (non-unique) decomposition
(see Proposition \ref{prophomotopy} (i))
\begin{equation}\label{eqnlogdecomposition}
	\omega=\alpha+\partial \beta\,,\quad \quad \quad
	\alpha\in A^{n,n}_{X}(\log D )\,,~
	\beta\in A^{n-1,n}_{X}(\star D )\,.
\end{equation}
A local analysis shows that $\alpha$ is absolutely integrable on $X$ (see Lemma \ref{lemintegrabilityoflogpart} for a stronger version).
Then one introduces the following definition.
\begin{dfn}[\cite{Li:2020regularized}]\label{dfnregularizedintegral}
The regularized integral $	\dashint_{X}: A^{n,n}_{X}(\star D )\rightarrow \mathbb{C}$ for case \ref{caseC} is defined as follows
\begin{equation}\label{eqndfnregularizedintegral}
	\dashint_{X}\partial\beta:=0\,,\quad 
	\dashint_{X} \omega:= \int_{X}\alpha\,.
\end{equation}
\end{dfn}
By the absolute integrability of $\alpha$ and Stokes theorem, it is  shown in \cite{Li:2020regularized} that the regularized integral is independent of the choice of the decomposition \eqref{eqnlogdecomposition} above (see Lemma \ref{lemindependenceondecompositionaseX} for a different and more direct proof).	
The operator $\dashint_{X}$ shares many properties with the ordinary integral, such as Stokes theorem, Fubini-type theorem, works
for family versions (e.g., it respects modularity under the action by $\mathrm{SL}_{2}(\mathbb{Z})$ when $C$ has genus one\footnote{
As explained in \cite{Li:2022regularized}, this property, together with 
the relation between regularized integrals and iterated $A$-cycle integrals, suggests that
$\dashint_{X}$ is \emph{the} regularization expected from physics.}) $\cdots$, among many other nice properties. 
For the $n=1$ case it also reduces to the Cauchy principal value.
It provides a quite satisfying regularization scheme and explain many interesting structures \cite{Gui-Li2021, Li:2020regularized, Li:2022regularized} such as holomorphic anomaly.
See Section \ref{secreformulationextension} below for a quick review of some key properties of this operator.\\

The present work is a continuation and extension of the joint works \cite{Li:2020regularized, Li:2022regularized}. 	In this work we 
provide  cohomological formulations for the operator $\dashint_{X}$ in case \ref{caseC} 
that give more conceptual explanations of its properties such as modularity.
We in fact prove the following stronger results for extensions of the operator $\dashint_{X}$ beyond case \ref{caseC}. 

	Before preceeding, we introduced the following notion that extends the one in \eqref{eqndfnregularizedintegral}.
	
	\begin{dfn}\label{dfnpropertyP}
			Let the setting be as in \eqref{eqnsettingofdivergentintegrals}.
		A linear  operator $I:  A^{n,n}_{X}(\log D)+\partial A^{n-1,n}_{X}(\star D)\rightarrow \mathbb{C} $
		is called an admissible regularized integration operator if 
				\begin{equation}\label{eqnreasonableregularizationintro}
		I(\alpha)=\int_{X}\alpha ~ ~\textrm{for}~~ \alpha\in A^{n,n}_{X}(\log D)\,,\quad 
			I(\partial\beta)=0~~ \textrm{for}~~\beta\in A^{n-1,n}_{X}(\star D)\,.
		\end{equation}
		\end{dfn}

	The main results of this work are summarized as follows.
	
	\begin{thmx}
		\label{thmgeneralizedregularizationmapAintro}
		Let the setting be as in \eqref{eqnsettingofdivergentintegrals}.
							Let $\mathfrak{r}: A^{n,n}_{X}(\log D)+\partial A^{n-1,n}_{X}(\star D)\rightarrow \mathbb{C} $
			be  an admissible regularized integration operator.	
								Then there exists a linear map
			\begin{eqnarray}\label{eqnDolmap}
			_{\mathrm{Dol}}:
			 A^{n,n}_{X}(\log D)+\partial A^{n-1,n}_{X}(\star D)& \rightarrow & A^{n,n}_{X}\,,\\
			 \omega &\mapsto & \omega_{\mathrm{Dol}}\,,\nonumber
			\end{eqnarray}
			such that for any $\omega\in A^{n,n}_{X}(\log D)+\partial A^{n-1,n}_{X}(\star D)$ one has
			\begin{equation}
				\mathfrak{r}(\omega)=\int_{X}\omega_{\mathrm{Dol}}=\langle [X], [\omega_{\mathrm{Dol}}]
				\rangle \,.
			\end{equation}
				Here $[X]$ is the fundamental class of $X$, and $[\omega_{\mathrm{Dol}}]$
		is a cohomology class in $H^{n}(A_{X}^{\bullet,n},\partial)\cong H^{2n}(X)$ of the smooth form $\omega_{\mathrm{Dol}}$.		
		Furthermore, the  form $\omega_{\mathrm{Dol}}$ in \eqref{eqnDolmap} can be concretely constructed from $\omega$
		using the conjugate Dolbeault-Cech double complex.			
	\end{thmx}

		\begin{thmx}
		\label{thmgeneralizedregularizationmapBintro}
		Let the setting be as in the $\ref{caseHA}$ case in \eqref{eqnsettingofdivergentintegrals}.
			Let $\mathfrak{s}: \mathfrak{F}+\partial A^{n-1,n}_{X}(\star D)\rightarrow \mathbb{C} $
			be a linear map satisfying \eqref{eqnreasonableregularizationintro}, with  $I$ replaced by $\mathfrak{s}|_{\mathfrak{F}+\partial A^{n-1,n}_{X}(\star D)}\,$.				
			Then there exists a linear map
			\begin{eqnarray}\label{eqnHodgemap}
				[~~]_{\mathrm{Hodge}}:
				A^{n,0}_{X}(\star D )\cap \mathrm{ker}\,\bar{\partial}& \rightarrow & 	H^{0}(X,\Omega_{X}^{n})\,,\\
				\phi &\mapsto &	[\phi]_{\mathrm{Hodge}}\,,\nonumber
			\end{eqnarray}
			such that for any factorizable form $\omega=\phi\wedge \psi\in\mathfrak{F}$ 
			one has 
			\begin{equation}
				\mathfrak{s}(\phi\wedge \psi)=\mathrm{Tr}\left( 	[\phi]_{\mathrm{Hodge}}\wedge \psi\right)\,.
			\end{equation}
			Here $\mathrm{Tr}$ is the trace map defined by
			\[
			\mathrm{Tr}: 
			H^{0}(X,\Omega_{X}^{n})\otimes H^{0}(X,\overline{\Omega}_{X}^{n})
			\xrightarrow{\wedge } H^{0}(X,\mathcal{A}_{X}^{n,n})\xrightarrow{\int_{X}}\mathbb{C}\,.
			\]
						Furthermore, the class $	[\phi]_{\mathrm{Hodge}}$ can be  constructed from $\phi$
			on the level of smooth differential forms from the Deligne splitting of the mixed Hodge structure on $H^{n}(U)$.

	\end{thmx}

	The subcases $\ref{caseC}, (\textbf{NCD})$ of $\ref{caseHA}$ play a special role:
	one has as  vector spaces  (see Proposition \ref{prophomotopy} (i))
	\begin{equation}\label{eqndecompositionofforms}
	A^{n,n}_{X}(\star D)=A^{n,n}_{X}(\log D)+\partial A^{n-1,n}_{X}(\star D)\,.
	\end{equation}
	In particular, for case $\ref{caseC}$, the operator $\mathfrak{r}$ coincides with the regularized integral $\dashint_{X}$ in \eqref{dfnregularizedintegral}.
	Theorem \ref{thmgeneralizedregularizationmapAintro} thus yields the desired cohomological formulation for $\dashint_{X}$.
	 Taking $\mathfrak{s}=\dashint_{X}|_{ \mathfrak{F}+\partial A^{n-1,n}_{X}(\star D)}\,$,
	Theorem \ref{thmgeneralizedregularizationmapBintro}  then  provides another cohomological formulation for the restriction of $\dashint_{X}$ on factorizable forms.\\

	The main idea of the proofs of Theorem \ref{thmgeneralizedregularizationmapAintro}
	and Theorem  \ref{thmgeneralizedregularizationmapBintro} is as follows. Starting from the space $ A^{n,n}_{X}(\log D)+\partial A^{n-1,n}_{X}(\star D)$ 
	in which forms with high order poles along $D$ do not contribute according to \eqref{eqnreasonableregularizationintro},	
	we take one step further by 		
	``getting rid of logarithms" in $ A^{n,n}_{X}(\log D)$.
	This is done
	using the tools of current cohomology and mixed Hodge structures \cite{Deligne:1971, Deligne:1974} respectively, which produces identification between regularized integrations and cohomological operations.

	It is possible that Theorem \ref{thmgeneralizedregularizationmapAintro}
	and Theorem \ref{thmgeneralizedregularizationmapBintro}
	can be further generalized.
	For example, the assumptions on
	projectivity  and smoothness of hypersurface arrangement  in $\ref{caseHA}$ can be weakened (see Proposition \ref{prophomotopy}  for related discussions). We however do not make this effort in this work.

	\begin{rem}
	While both cohomological formulations in Theorem \ref{thmgeneralizedregularizationmapAintro} and Theorem \ref{thmgeneralizedregularizationmapBintro} work with cohomologies constructed from \emph{global} differential forms (see
	Remark \ref{remreductionofpoleholomorphiccomplex} for related discussions), the coboundaries are different.
	To be more precise, let $D_{X}^{'n-1,n}(\star D)$ be the space of currents\footnote{That is, differential forms with coefficients being distributions instead of smooth functions.} of type $(n-1,n)$,
	then one has (see Section 	\ref{secextension}, Section
	\ref{sectracemaponcurrentcohomology})
		\[
	dA^{n-1}_{X}(\star D)\wedge \psi \subsetneq \partial A^{n-1,n}_{X}(\star D) \subsetneq \partial D_{X}^{'n-1,n}\,,\quad\text{for}~ \psi\in 	A^{0,n}_{X}\cap \mathrm{ker}\,\partial\,.
	\]
Theorem \ref{thmgeneralizedregularizationmapAintro}
works with the last set of coboundaries, while Theorem \ref{thmgeneralizedregularizationmapBintro}  works with the first set.
\end{rem}

		\begin{rem}
For case $\ref{caseC}$, a key feature of Theorem \ref{thmgeneralizedregularizationmapAintro} and Theorem \ref{thmgeneralizedregularizationmapBintro} lies in that it
	expresses regularized integrals in terms of  cohomological pairings on the \emph{compact} space $X=C^{n}$.
	This offers a particularly convenient approach to perform regularizations of divergent integrals on the naive compactification $C^{n}$ of $\mathrm{Conf}_{n}(C)$ instead of other subtle ones in the literature such as the Fulton-MacPherson compactification.

	For cases $\ref{caseC}, (\textbf{NCD})$, the map $\mathfrak{r}$
	is closely related (see Corollary \ref{coruniquenessofregularization}) to the principal value current introduced in  \cite{Dolbeault:1970, Herrera:1971}.
	It seems that this connection  make it convenient to
	compare our regularization scheme to
	other ones in the literature (e.g., \cite{Felder:2017, Felder:2018, Brown:2021a, Brown:2021b})
	which are often compared with 
	the aforementioned principal value current. 
\end{rem}

		\subsection*{Structure of the paper}

	In Section \ref{secreformulationextension} we give a brief overview
	on regularized integrals on configuration space integrals of Riemann surfaces \cite{Li:2020regularized, Li:2022regularized}, and slightly reformulate part of results therein.
	
	In Section 	\ref{seccurrentcohomology} we extend the notion of regularized integrals and 	
	provide a cohomological formulation for the extended notion
	using current cohomology.
	We also explain how to obtain smooth differential form representatives of the current cohomology class using the conjugate Dolbeault-Cech double complex.
	
	In Section 	\ref{secfactorizableforms}, we specialize to the subspace of factorizable forms for case $\ref{caseHA}$
	and provide another cohomological formulation of regularized integrals using mixed Hodge structures. We also discuss practical ways of obtaining the resulting cohomology class in terms of smooth differential forms.

	We demonstrate our results through some examples in Section \ref{secexamples}.
	Some discussions on the local absolute integrability of logarithmic forms
	are relegated to Appendix \ref{appendixlocalestimates}.


	\subsection*{Acknowledgement}
	
	 The author would like to thank Si Li for fruitful collaborations on related topics, and Jin Cao, Si Li, Zhixuan Sheng, Xinxing Tang and Dingxin Zhang for helpful conversations.
	J.~Z. is supported by the
	national key research and development
	program of China (No. 2022YFA1007100) and the Young overseas high-level talents introduction plan of China.


	\subsection*{Notation and conventions}
	
	\begin{itemize}
		\item 
		
		The space of sections of a sheaf will be denoted by the same letter
		as the sheaf, but in script instead of calligraphic.
		For example, for a sheaf $\mathcal{A}$ on $X$, we  denote $A=\Gamma(X,\mathcal{A})=\mathcal{A}(X)$ to be the space of global sections over $X$.   For a general open subspace $V$ of $X$, denote $A_{V}=\mathcal{A}(V)$.  
		
		\item For a complex of sheaves of Abelian groups $\mathcal{F}^{\bullet}$ on $X$,  
		denote by $\mathbb{H}^{k}(X,\mathcal{F}^{\bullet})$ or more simply $\mathbb{H}^{k}(\mathcal{F}^{\bullet})$ its $k$th hypercohomology.
		We also write interchangeably
		$H^{k}(A^{\bullet,\ell}_{X},\partial)$, $H^{k}_{\partial}(X, A_{X}^{\bullet,
		\ell})$ for 
		$H^{k}(X, (A_{X}^{\bullet,\ell},\partial))$, depending on the emphasis.
		
		\item 
		Let $\mathcal{D}_{X}^{p,q}$ be the sheaf of smooth $(p,q)$ forms on $X$ with compact support, called test forms of degree $(p,q)$.
		When $X$ is compact, one has $\mathcal{D}_{X}^{p,q}=\mathcal{A}_{X}^{p,q}$
		and we will use both interchangeably.
		Let $\mathcal{D}_{X}^{'n-p,n-q}$ be the linear dual of $\mathcal{D}_{X}^{p,q}$, called the sheaf of degree $(n-p,n-q)$ currents.

		\item Let $\omega$ be a differential form, if it is a locally integrable form (i.e., a $L^{1}$ form), we denote the corresponding integral current by $T_{\omega}$.
		We also use the same notation $T_{\omega}$ to denote the current (which needs to extend the integration operator) that $\omega$ might give rise to.
		
	\end{itemize}

	\section{Review on regularized  integrals on configuration spaces of Riemann surfaces}
	\label{secreformulationextension}

	In this section, we 
	first recall some basic constructions
	that are needed in studying regularized integrals.
	We then
	review the notion of regularized configuration space integrals of Riemann surfaces introduced in \cite{Li:2020regularized} and further developed in \cite{Li:2022regularized}, and reformulate
	some results therein.
	
		\subsection{Preliminaries}
	
We now list some constructions and operations
in case $(\textbf{H})$, which provide the basis of 
more complicated ones.

	\subsubsection{Poincar\'e, Leray, and holomorphic residue maps}\label{secresidueonmeromorphicforms}

	Suppose $D$ is a smooth hypersurface in a smooth compact complex manifold $X$, that is case $(\textbf{H})$.
	Then one can define 
	the Poincar\'e residue map $\mathrm{res}_{D}$ on $(\mathcal{A}^{\bullet}_{X}(\log D),d)$
	and the Leray residue  $\mathrm{res}_{L}$
	on 
	$(\mathcal{A}^{\bullet}_{X}(\star D),d)$.
	Both of these descend to maps on de Rham cohomology which are denoted by
	$\mathrm{Res}_{D},\mathrm{Res}_{L}$ respectively.
	In fact they coincide on the cohomology level due to the quasi-isomorphism (see Proposition \ref{prophomotopy}) between the two complexes of fine sheaves.
	See \cite{Aizenberg:1994} for detailed discussions on these residue maps.
	
	The residue maps $\mathrm{res}_{D}, \mathrm{res}_{L}$  restriction to maps $\mathrm{res}_{\partial}$
	on
	$(\Omega^{p}_{X}(\log D),\partial), (\Omega^{p}_{X}(\star D),\partial)$. 
		Here 		
	$\Omega^{\bullet}_{X}(\log D), \Omega_{X}^{\bullet}(\star D)$ are the holomorphic versions of 
	$\mathcal{A}^{\bullet}_{X}(\log D), \mathcal{A}_{X}^{\bullet}(\star D)$, respectively.
	The map $\mathrm{res}_{\partial}$ can be furthermore defined on the conjugate Dolbeault complexes
	 $(\mathcal{A}^{\bullet,0}_{X}(\log D), \partial)$ and $(\mathcal{A}^{\bullet,0}_{X}(\star D), \partial)$,
	and thus  $\mathrm{Res}_{\partial}$
	on the cohomology spaces
	\[
 H^{\bullet}(A_{X}^{\bullet,0}(\log D), \partial)\,,\quad	H^{\bullet}(A_{X}^{\bullet,0}(\star D), \partial)\,.
	\]
	Concretely, one has
	\begin{equation}\label{eqnholomorphicresidue}
		\mathrm{res}_{\partial}={1\over 2\pi \mathbf{i}}\lim_{\varepsilon\rightarrow 0}\int_{\partial B_{\varepsilon}(D)}\quad \,,
	\end{equation}
	where $B_{\varepsilon}(D)$ is a disk bundle of radius $\varepsilon$ over the divisor $D$. The result is independent \cite{Herrera:1971, Coleff:1978}  (see also \cite{Li:2020regularized}) of 
	the local defining equation for $D$ and the Hermitian metric compatible with the complex structure.
	In particular, from the description of the right hand side of \eqref{eqnholomorphicresidue} one has \cite{Herrera:1971} (see also
	 			\cite[Lemma 2.2]{Li:2020regularized}) 
	\begin{equation}\label{eqnvanishingofholomorphicresiduebytypereasonscaseDol0}
	\mathrm{res}_{\partial}(\omega)=0\,,\quad \forall\, \omega\in A^{p,q}_{X}(\star D)\,,~q\geq 1\,.
	\end{equation}
	Again this vanishing result holds for any choice of local defining equation for $D$ and Hermitian metric.
		Due to \eqref{eqnvanishingofholomorphicresiduebytypereasonscaseDol0}, the map $\mathrm{res}_{\partial}$ is naturally
	extended
	to $(\mathcal{A}_{X}^{\bullet}(\log D),\partial),
	(\mathcal{A}_{X}^{\bullet}(\star D),\partial)$.
	
	The following alternative sheaf-level description is also useful
	\begin{equation}\label{eqnholresidue}
		\mathrm{res}_{\partial}: {
			\mathcal{A}^{m}_{X}(\star D )\over
			\partial  	  \mathcal{A}^{m-1}_{X}(\star D )
			+ \overline{F^{1}}\mathcal{A}^{m}_{X}(\star D )}\rightarrow 
		\mathcal{A}^{m-1}_{D}
		\,,
	\end{equation}
	where the filtration $F^{\bullet}$ is the Hodge filtration given by
	\begin{equation}\label{eqndfnHodgefiltration}
	F^{p}\mathcal{A}^{m}_{X}(\star D )=\bigoplus_{k\geq p}\mathcal{A}^{k,m-k}_{X}(\star D )\,.
	\end{equation}
	The map $\mathrm{res}_{\partial}$  is related to the Leray residue map $\mathrm{res}_{L}$ by
	\[
	\mathrm{res}_{\partial}(\omega)=\pi_{\mathrm{top}}\, \mathrm{res}_{L}(\omega)\,,\quad\forall\,
	\omega\in A^{m}_{X}(\star D)\,.
	\]
	where $\pi_{\mathrm{top}}$ is the projection to the top component in the Hodge decomposition,
	see \cite[Proposition 2.6]{Felder:2017}.

	\subsubsection{Currents}

We again consider case $(\textbf{H})$.
	Recall that
	the action of differential operators on currents  is induced by that on test forms. For example, for
	the integral current $T_{\rho}$ determined from an integrable form $\rho\in L^{1}$  and 	
	 a test form $\psi\in D^{\bullet,\bullet}_{X}$, one has
	\begin{equation}\label{eqndfnpartialcurrent}
		(\bar{\partial}T_{\rho})(\psi):=(-1)^{|\rho|+1}
		T_{\rho}( \bar{\partial} \psi)
		\,.	\end{equation}
	Hereafter
	$|\rho|$ stands for the degree of the differential form $\rho$.
	By Stokes theorem, \eqref{eqnholomorphicresidue} and \eqref{eqnvanishingofholomorphicresiduebytypereasonscaseDol0},
	one then has the following relation
	\begin{equation}\label{eqnbarpartialcurrent}
		(\bar{\partial}T_{\rho})(\psi)=T_{\bar{\partial}\rho}(\psi)+2\pi \mathbf{i}\, \mathrm{res}_{\partial}(\rho \psi)\,,\quad \forall\, \psi\in D^{\bullet,\bullet}_{X}\,.
	\end{equation}
	In particular,
	any $\rho \in L^{1}\cap \mathrm{ker}\,\bar{\partial}$ defines a residue current
	$
	\bar{\partial}T_{\rho}\in D_{X}^{'\bullet,\bullet}
	$
	that satisfies
	\begin{equation}
	(\bar{\partial}T_{\rho})(\psi)=2\pi \mathbf{i}\,
	\mathrm{res}_{\partial}(\rho \psi)\,.
	\end{equation}
	The relation \eqref{eqnvanishingofholomorphicresiduebytypereasonscaseDol0}
	and 
	the same reasoning above proving \eqref{eqnbarpartialcurrent} 
	gives 
	\begin{equation}\label{eqnpartialcurrent}
		{\partial}T_{\rho}=T_{{\partial}\rho}+2\pi \mathbf{i} \,\mathrm{res}_{\partial}(\rho \wedge \cdot)=T_{{\partial}\rho}\,,
	\end{equation}
	if all of the quantities above make sense.
	
	Note that the validity of the above formulas
	does not require $\rho$ to lie in $L^{1}$: $\rho $ can  admit considerable singularities
	as discussed in Section \ref{secextension}
 and Appendix \ref{appendixlocalestimates} below.

	\subsection{Regularized integrals on configuration spaces of Riemann surfaces}

	We next review the notion of regularized configuration space integrals of Riemann surfaces introduce 
	in \cite{Li:2020regularized} and reformulate some results therein.
	
		\subsubsection{$\partial$-cohomology, regularized integrations as currents}

	We first rephrase 
part of the results in \cite{Li:2020regularized} (more precisely, \cite[Lemma 2.33, Lemma 2.34, Lemma 2.35, Theorem  2.36, Definition 2.37]{Li:2020regularized}) for the regularized integrals in case \ref{caseC},
in terms of currents.

	\begin{thm}
		\label{thmformdefiningcurrent}
		Consider  case $\ref{caseC}$.
		\begin{enumerate}[i).]
			\item
			Any $\omega\in A_{X}^{n,n}(\star D)$ defines a  current $T_{\omega}$ 
			that	
			is independent of choices such as the Hermitian metric compatible with the complex structure on $C$, with
			\begin{equation}\label{eqnregularizedintegral=principalvalue}
				T_{\omega}(1)=\dashint_{X} \omega\cdot 1\,.
			\end{equation}	
			
			Furthermore, one has $T_{\partial \gamma}=0$ for $\gamma\in A_{X}^{n-1,n}(\star D)$.

			\item For the $n=1$ case, $T_{\omega}$ coincides with the Cauchy principal value current defined by $\omega$.
			
				\item The regularized integral on $X=C^{n}$ is equal to the
			iterated regularized integrals on $C$, and is independent of the order of iteration.
		\end{enumerate}
	\end{thm}
	Let us illustrate the connection to the principal value current given  in Theorem \ref{thmformdefiningcurrent} (ii), for the $n=1$ case in $\ref{caseC}$.
	Taking $B_{\varepsilon}(D)$ to be a small disk bundle over $D$, then 
	by \eqref{eqndfnregularizedintegral} and \eqref{eqnvanishingofholomorphicresiduebytypereasonscaseDol0} one has for the Cauchy principal value
	\begin{eqnarray*}
	\mathrm{P.V.}\int_{X}\omega&=&
	\lim_{\varepsilon\rightarrow 0}\left(
	\int_{X\setminus B_{\varepsilon}(D)} \alpha
	+\int_{X\setminus B_{\varepsilon}(D)} \partial\beta \right)\\
	&=&		\lim_{\varepsilon\rightarrow 0}
	\int_{X\setminus B_{\varepsilon}(D)} \alpha
	-	\lim_{\varepsilon\rightarrow 0}\int_{\partial  B_{\varepsilon}(D)} \beta 
	=
	\int_{X} \alpha
	-	2\pi \mathbf{i}\, \mathrm{res}_{\partial}\beta =	\int_{X} \alpha\,.
	\end{eqnarray*}
		A similar reasoning and \eqref{eqnvanishingofholomorphicresiduebytypereasonscaseDol0} also explains why \eqref{eqndfnregularizedintegral}  is independent of the decomposition $	\omega=\alpha+\partial \beta$ in \eqref{eqnlogdecomposition} in this case.
		
	For the general case $\ref{caseC}$ studied in \cite{Li:2020regularized}, the independence of the regularized integral on the decomposition \eqref{eqnlogdecomposition} is
		shown by first proving the commutativity of $\partial$ and $\dashint$, then using the product structure of $X=C^{n}$,  
		 Theorem \ref{thmformdefiningcurrent} (iii),
and \eqref{eqnvanishingofholomorphicresiduebytypereasonscaseDol0}.

	\begin{ex}\label{exwpn=1definitioncomputation}
		Let $X=E=\mathbb{C}/(\mathbb{Z}\oplus \mathbb{Z}\tau)$ be an elliptic curve, and $D=0$ be the origin represented by $0\in \mathbb{C}$.
		Consider $\omega=\wp(z)dz\wedge {d\bar{z}\over \bar{\tau}-\tau}$.
	Let $\theta$ be the Jacobi theta function $\vartheta_{({1\over 2},{1\over 2})}$ and
		\begin{equation}\label{eqndfnofZhat}
		\widehat{\theta}(z):=\theta(z)\cdot \exp(-2\pi {(\mathrm {im}~z)^2\over \mathrm{im}~\tau})\,,\quad
		\widehat{Z}=\partial_{z} \ln \widehat{\theta}=\zeta(z)-z\widehat{\eta}_{1}		+\bar{z}\cdot {-\pi\over \mathrm{im}\,\tau}\,.
			\end{equation}
		Here $\zeta(z),\wp(z)$ are the Weierstrass elliptic functions and 		
		\begin{equation}
		\widehat{\eta}_1:=
		\eta_{1}+{-\pi\over \mathrm{im}\,\tau}\,,
		\end{equation}
		 with $\eta_{1}=	 
		 {\pi^2\over 3}E_{2}$ being the second Eisenstein series.
	This gives
		\begin{equation}\label{eqnderivativesofZhat}
		\partial\widehat{Z}=-(\wp(z)+\widehat{\eta}_{1})dz\,,\quad
		\bar{\partial}\widehat{Z}=d(2\pi \mathbf{i}\, {\bar{z}\over \bar{\tau}-\tau})\,.
		\end{equation}
	It follows that
		\[
		\omega=-\widehat{\eta}_1 \mathrm{vol}-\partial(\widehat{Z} {d\bar{z}\over \bar{\tau}-\tau})\,,\quad \mathrm{vol}:={\mathbf{i}\over 2\,\mathrm{im}\,\tau}dz\wedge d\bar{z}\,.
		\]
		By the definitions in \eqref{eqnlogdecomposition}, \eqref{eqndfnregularizedintegral}, this gives 
		$\dashint_{X}\omega=-\widehat{\eta}_1$.			
			
		Consider also
		\[
		\omega=\widehat{Z}dz\wedge {d\bar{z}\over \bar{\tau}-\tau}=\partial\left( \log (\widehat{\theta}\bar{\theta})\wedge {d\bar{z}\over \bar{\tau}-\tau}\right)\,.
		\]
				Since $\ln (\widehat{\theta} \bar{\theta} )\wedge d\bar{z}\notin A_{X}^{0,1}(\star D)$, the above does not give a decomposition \eqref{eqnlogdecomposition}.
		However, according to \eqref{eqnregularizedintegral=principalvalue},  Stokes theorem, and \eqref{eqndfnregularizedintegral}, we still have
		\begin{eqnarray}\label{eqnvanihsingofloglog}
		\dashint_{X}\omega
		=	\lim_{\varepsilon\rightarrow 0}
		\int_{X\setminus B_{\varepsilon}(0)}\partial (\ln (\widehat{\theta} \bar{\theta} )\wedge  {d\bar{z}\over \bar{\tau}-\tau})
		&=&\lim_{\varepsilon\rightarrow 0}
		\int_{X\setminus B_{\varepsilon}(0)}d (\ln (\widehat{\theta} \bar{\theta} )\wedge  {d\bar{z}\over \bar{\tau}-\tau})\nonumber\\
&		=&-\lim_{\varepsilon\rightarrow 0}
		\int_{\partial B_{\varepsilon}(0)}\ln (\widehat{\theta} \bar{\theta} )\wedge  {d\bar{z}\over \bar{\tau}-\tau}=0\,.
		\end{eqnarray}
		\xxqed
	\end{ex}

	\subsubsection{$d$-cohomology, regularized integrations as iterated residues}\label{secdcohomologycaseC}
	
For the $n=1$ case of \ref{caseC}, one has (see Proposition \ref{prophomotopy} (ii) for a more general statement)
		\begin{equation}\label{eqnvanishingoftopcoh}
		H^{2}(A^{\bullet}_{C}(\star D ),d)
		\cong H^{2}(C, Rj_{*}\mathbb{C}_{U})\cong H^{2}(U)=0
		\,.
	\end{equation}
	Now take any $\omega\in A_{C}^{1,1}(\star D)$. 
		It follows from \eqref{eqnvanishingoftopcoh} that there exists $\gamma\in A^{1}_{C}(\star D )$ such that
	\begin{equation}\label{eqndprimitive}
	\omega=d\gamma\,.
	\end{equation}
	Using the relation in Theorem \ref{thmformdefiningcurrent} between regularized integral and Cauchy principal value, we have
	\[
	\dashint_{C}\omega=
	\mathrm{P.V.} \int_{C} \omega=
	\lim_{\varepsilon\rightarrow 0}
	\int_{C\setminus B_{\varepsilon}(D)} d\gamma
	=-	\lim_{\varepsilon\rightarrow 0}
	\int_{\partial  B_{\varepsilon}(D)} \gamma
	=-2\pi \mathbf{i}\, \mathrm{res}_{\partial}\gamma\,.
	\]	
	By Stokes theorem, the above is independent of the choice of the primitive $\gamma$. 	
	For this reason, we shall often write $\gamma=d^{-1}\omega$.

	\begin{ex}\label{exwpn=1residuecomputation}
		Consider as in Example \ref{exwpn=1definitioncomputation} the case with $X=E, D=0$,
		with $\omega=\wp(z)dz\wedge {d\bar{z}\over \bar{\tau}-\tau}$.
		Then from \eqref{eqnderivativesofZhat} one has
		\[
	 \omega=d(-{1\over 2\pi \mathbf{i} } \wp \widehat{Z}dz)\,.
		\]
		Using the Laurent series expansions of the Weierstrass elliptic functions, this gives
		\[
		\dashint_{X}\omega=-2\pi \mathbf{i}\, \mathrm{res}_{\partial}	(	d^{-1}\omega)
		=-\widehat{\eta}_1\,.
		\]
		For the case $\omega=\widehat{Z}dz\wedge {d\bar{z}\over \bar{\tau}-\tau}$, again using \eqref{eqnderivativesofZhat} one has
		\[
		\omega=d(-{1\over \pii}{1\over 2}\widehat{Z}^2 dz)\,.
		\]
		By parity, one has $\dashint_{X}\omega=\mathrm{res}_{\partial}({1\over 2}\widehat{Z}^2 dz)=0$.
		\xxqed
	\end{ex}
	
	Instead of finding the primitive $\gamma$ in \eqref{eqndprimitive} on $C$, one can lift to a cover, say the universal cover of $C$, and then perform the computations in a fundamental domain $F$.
	This is particularly useful in computing  the regularized integral of a factorizable form $\omega=\phi\wedge \psi$ 
	as defined in  \eqref{eqndfnfactorizableforms}.
	In fact, let $u(p)=\int^{p}_{p_{*}}\psi$ be the primitive of $\psi$, where $p_{*}\cap D=\emptyset$.
	Then with the condition $d\phi=\bar{\partial}\phi=0$ one has
	\begin{equation}\label{eqncontacttermn=1}
		\dashint_{C}\omega
		=	\lim_{\varepsilon\rightarrow 0}
		\int_{F\setminus B_{\varepsilon}(D)} 
		\phi\wedge \psi
		=\lim_{\varepsilon\rightarrow 0}
		\int_{F\setminus B_{\varepsilon}(D)} 
		\phi\wedge du
		=2\pi \mathbf{i}\,\mathrm{res}_{\partial}(\phi u)-\int_{\partial F} \phi u\,.
	\end{equation}
	This is an analogue of the Cauchy integral formula.
	
	The product structure of $X=C^{n}$ in  case $\ref{caseC}$ reduces a regularized integral
	to an iterated regularized integral by Theorem \ref{thmformdefiningcurrent} (iii).
	The iterated version of \eqref{eqncontacttermn=1} then  leads to a mathematical formulation \cite{Li:2020regularized, Li:2022regularized}  of the so-called contact term
	singularities in the context of 2d chiral CFTs, which schematically takes the form
	\[
	\text{anomaly free regularized integral}=\text{point-splitting regularization}+\text{contact term singularities}\,.
	\]
	Here by anomaly free we mean that the construction works for the relative version of Riemann surfaces. 
	For instance, for the genus one case, being anomaly free means  being modular when regarded as constructions on 
	the universal covering space.

	\begin{ex}\label{exwzn=1residuecomputation}
		Consider as in Example \ref{exwpn=1definitioncomputation} and Example \ref{exwpn=1residuecomputation} the case with $X=E, D=0$.
		Suppose $\omega=\phi\wedge \psi$ with $\phi\in 	A^{1,0}_{X}(\star D )\cap \mathrm{ker}\,\bar{\partial}$
		and
		$
		\psi={d\bar{z}\over \bar{\tau}-\tau}
		$.
		Then from \eqref{eqnderivativesofZhat} one has
		\[
		\phi\wedge \psi=\phi\wedge {d\bar{z}-dz\over \bar{\tau}-\tau}={1\over 2\pi \mathbf{i}}\phi \wedge d\widehat{Z}\,,
		\]
		where ${d\bar{z}-dz\over \bar{\tau}-\tau}$ is the Poincar\'e dual of the $A$-cycle class in the weak Torelli marking $(A,B)$ for the elliptic curve $E$.
		Here $A$ is the class of the cycle connecting $0,1$
		and $B$ the one connecting $0,\tau$ in the fundamental domain on the universal cover $\mathbb{C}$ of $E$.
		Now in  \eqref{eqncontacttermn=1} one can take
		\[
		u={\bar{z}\over \bar{\tau}-\tau}\,,\quad{ \bar{z}-z\over \bar{\tau}-\tau}\,,\quad  \widehat{Z}\,,\cdots
		\]
		
		Taking the choice $u= {z-\bar{z}\over \bar{\tau}-\tau}$, one then obtains the relation between
		regularized integral, $A$-cycle integral and residues \cite[Proposition 2.18, Proposition 2.20, Proposition 2.22]{Li:2020regularized}. 
		
		Taking $u=\widehat{Z}$, as  in Example \ref{exwpn=1residuecomputation} one expresses the regularized integral in terms of residues of 
		as
		polynomials in $\widehat{Z}$ with coefficients being elliptic functions, known as almost-elliptic functions.		
		Based on some structural results on the latter, the calculation of the iterated regularized integrals in case $\ref{caseC}$
		can be drastically simplified
		in terms of iterated residues of highly structured integrands \cite{Li:2022regularized}.
		
			\xxqed
	\end{ex}

	\section{Regularized integration as trace map on current cohomology}
	\label{seccurrentcohomology}
	
	In this section we extend the notion of regularized integrals defined for case $\ref{caseC}$. We then prove Theorem \ref {thmgeneralizedregularizationmapAintro}, which in particular provides a cohomological formulation of the  regularized integration $\dashint_{X}$ in 
	case $\ref{caseC}$ and $(\textbf{NCD})$.
	We also explain how to obtain smooth differential form representatives of the cohomology class
	$[\omega_{\mathrm{Dol}}]\in H^{n}(A_{X}^{\bullet,n},\partial)$.	
	
	\subsection{Comparison theorems}
	We list a few comparison results that are useful for later purposes.
	
	\begin{prop}\label{prophomotopy}
~
		\begin{enumerate}[i).]
					\item 	The following statement holds for case $\ref{caseC}$ and $(\textbf{NCD})$. The inclusion
$	(\mathcal{A}^{\bullet,n}_{X}(\log D),\partial )\hookrightarrow (\mathcal{A}^{\bullet,n}_{X}(\star D),\partial )
			$
			induces a surjection
					\begin{equation}
						\label{eqnsmoothlogstarquasiisopartialnndegree}
						H^{n}(\mathcal{A}^{\bullet,n}_{X}(\log D),\partial )\rightarrow 
			H^{n}(\mathcal{A}^{\bullet,n}_{X}(\star D),\partial )
				\end{equation}
			whose kernel lies in  
	$\partial\left( H^{n-1}(\mathcal{A}^{\bullet,n}_{X}(\star D)/\mathcal{A}^{\bullet,n}_{X}(\log D),\partial )\right)
	$.
			In particular, for any $\omega\in A_{X}^{n,n}(\star D)$, one has a (non-unique) decomposition
			\begin{equation}\label{eqnlogdecompositioncasepartial}
				\omega=\alpha+\partial \beta\,,\quad\quad
				\alpha\in A^{n,n}_{X}(\log D )\,,\quad
				\beta\in A^{n-1,n}_{X}(\star D )\,.
			\end{equation}
			\item 	The following statement holds for case $\ref{caseHA}$. The inclusion
$(\mathcal{A}^{\bullet}_{X}(\log D),d )\hookrightarrow (\mathcal{A}^{\bullet}_{X}(\star D),d)$
			induces an isomorphism at the $n$th cohomology
			\begin{equation}\label{eqnsmoothlogstarquasiisod}
				H^{n}(\mathcal{A}^{\bullet}_{X}(\log D),d )\cong  H^{n}(\mathcal{A}^{\bullet}_{X}(\star D),d)\,.
			\end{equation}		
			In particular, for any $\omega\in A_{X}^{n}(\star D)\cap \mathrm{ker}\,d$, one has a decomposition
			\begin{equation}\label{eqnlogdecompositiond}
				\omega=\alpha+d \beta\,,\quad\quad
				\alpha\in A^{n}_{X}(\log D )\cap \mathrm{ker}\,d\,,\quad
				\beta\in A^{n-1}_{X}(\star D )\,.
			\end{equation}
		\end{enumerate}	
	\end{prop}
	
	\begin{proof}
		\begin{enumerate}[i).]
			\item
				For case $(\textbf{H})$,
		by using integration by parts  one can  reduce the order of pole and show that the inclusion 
							$(\mathcal{A}^{\bullet,q}_{X}(\log D),\partial )\hookrightarrow (\mathcal{A}^{\bullet,q}_{X}(\star D),\partial )$, with $0\leq q\leq n$,
actually gives a quasi-isomorphism, see e.g., \cite[Proposition 2.1]{Felder:2017} for a proof that incorporates the reduction of pole order in a spectral sequence argument.
The same argument therein can be adapted to prove the desired claim 	\eqref{eqnsmoothlogstarquasiisopartialnndegree}
in case   $(\textbf{NCD})$ and $\ref{caseC}$, with the 
reduction of pole order step proved in  \cite{Li:2020regularized}.
For case $(\textbf{NCD})$ the reduction can be performed as in \cite[Lemma 2.35]{Li:2020regularized},
 by choosing a system of local coordinates that contains the defining equations for the irreducible components of the divisor $D$. For case $\ref{caseC}$, the reduction uses 
the combinatorial structure of the  cohomology of the  configuration space of $\mathbb{C}$ that provides a coherent way of 
performing reduction of holomorphic pole order  \cite[Lemma 2.32, Lemma 2.33, Lemma 2.35, Lemma 2.36]{Li:2020regularized}.
To achieve this, the following Arnold type relations also play a key role (here $z_{ij}:=z_{i}-z_{j}$)
\begin{eqnarray*}
	{1\over z_{ij}}{1\over z_{jk}}
	+{1\over z_{jk}}{1\over z_{ki}}+	{1\over z_{ki}}{1\over z_{ij}}
	&=&0\,,\\
	{dz_{ij}\over z_{ij}}\wedge {dz_{jk}\over z_{jk}}
	+{dz_{jk}\over z_{jk}}\wedge {dz_{ki}\over z_{ki}}+	{dz_{ki}\over z_{ki}}\wedge {dz_{ij}\over z_{ij}}
	&=&0\,,~~ \text{for distinct values of}~i,j,k\,.
\end{eqnarray*}

		Observe that the sheaves involved in $(\mathcal{A}^{\bullet,n}_{X}(\log D),\partial )\hookrightarrow (\mathcal{A}^{\bullet,n}_{X}(\star D),\partial )
		$ are $C_{X}^{\infty}$-modules and thus acyclic, their hypercohomologies coincide with cohomologies of the corresponding complexes of global sections. Taking hypercohomology on $(\mathcal{A}^{\bullet,n}_{X}(\log D),\partial )\hookrightarrow (\mathcal{A}^{\bullet,n}_{X}(\star D),\partial )
		$ and using	\eqref{eqnsmoothlogstarquasiisopartialnndegree}, one proves
			  \eqref{eqnlogdecompositioncasepartial}.			
			\item 
			
			By Atiyah-Hodge/Grothendieck comparison one has
			\begin{equation}\label{eqnAtiyahHodge}
				(\mathcal{A}_{X}^{\bullet}(\star D),d)\xrightarrow{\mathrm{quasi-isomorphism}} (Rj_{*}\mathcal{A}^{\bullet}_{U},d)
				\xrightarrow{\mathrm{quasi-isomorphism}} Rj_{*}\mathbb{C}_{U}\,.
			\end{equation}
			
			On the other hand, based on Brieskorn-Orlik-Solomon theorem for case $\ref{caseHA}$,
			one has the logarithmic comparison asserting 
			\[
			(\mathcal{A}_{X}^{\bullet}(\log D),d)\xrightarrow{\mathrm{quasi-isomorphism}} (Rj_{*}\mathcal{A}^{\bullet}_{U},d)
			\]
			and its holomorphic version.
			See e.g., \cite[Theorem 3.13]{Dupont:2015} for a proof.
			Combining the above two,
one sees that $(\mathcal{A}^{\bullet}_{X}(\log D),d )\hookrightarrow (\mathcal{A}^{\bullet}_{X}(\star D),d)$ is  a quasi-isomorphism, proving   \eqref{eqnsmoothlogstarquasiisod}.
Taking hypercohomology and using the acyclicity of sheaves of $C_{X}^{\infty}$-modules, one obtains 
			\begin{equation*}\label{eqnAHBOS}
				H^{*}	(A^{\bullet}_{X}(\log D ),d )\cong 
				H^{*}	(A^{\bullet}_{X}(\star D ),d )
				\cong
				\mathbb{H}^{*}	(Rj_{*}A^{\bullet}_{U})
				\cong 
				H^{*}	(U)\,.
			\end{equation*}
			This  proves \eqref{eqnlogdecompositiond}.
		\end{enumerate}
	\end{proof}

	\begin{rem}\label{remreductionofpoleholomorphiccomplex}
		Note that in Proposition \ref{prophomotopy} (i), the reduction of pole method actually gives rise to a quasi-isomorphism between complexes of sheaves
		\begin{equation}\label{eqnhollogstarquasiiso}
			(\Omega^{\bullet}_{X}(\log D),\partial )\xrightarrow{\mathrm{quasi-isomorphism}} (\Omega^{\bullet}_{X}(\star D),\partial )\,.
		\end{equation}
		In particular, one has for the hypercohomologies $\mathbb{H}^{n}(\Omega^{\bullet}_{X}(\star D),\partial )\cong 
		\mathbb{H}^{n}(\Omega^{\bullet}_{X}(\log D),\partial )$.
		However, this does not yield the holomorphic version of \eqref{eqnlogdecompositioncasepartial}
		since in general the hypercohomologies are different from
		cohomologies of the complexes of \emph{global} differential forms.
	\end{rem}

	\subsection{Extension of regularized integrals}
	\label{secextension}

	We next extend the notion of regularized integrals defined for case $\ref{caseC}$.
We also relate it to  the principal value current introduced in \cite{Dolbeault:1970, Herrera:1971}.

\subsubsection{Extension}

We first generalize the definition of regularized integrals  given in
\eqref{eqnlogdecomposition} and \eqref{eqndfnregularizedintegral}.
		
		 \begin{dfn}\label{dfnregularizedintegralcaseD}
		 Define a linear map $\mathfrak{r}:A^{n,n}_{X}(\log D)+\partial A^{n-1,n}_{X}(\star D)\rightarrow \mathbb{C}$ as follows.		 		
	 			 		For any element $\omega\in A^{n,n}_{X}(\log D)+\partial A^{n-1,n}_{X}(\star D)$, writing it as a sum
	 			 		\begin{equation}\label{eqndecompositionforrmap}
	 			 		\omega=\alpha+ \partial \beta\,, \quad \quad 
	 			 		\alpha\in A^{n,n}_{X}(\log D)\,,~
	 			 		\beta\in A^{n-1,n}_{X}(\star D)\,.
	 			 		\end{equation}
	 			 		Then set
		 			\begin{equation}\label{eqndfnregularizedintegralcaseD}
		 			\mathfrak{r}( \omega):= \int_{X}\alpha\,.
		 		\end{equation}	 
	 				In particular, the map $\mathfrak{r}$ is defined on the full space $A^{n,n}_{X}(\star D)$ when  Proposition \ref{prophomotopy} (i) holds (which  covers  cases $\ref{caseC}, (\textbf{NCD}$)) since 
	 						\begin{equation}\label{eqnfullspace}
	 							\text{ Proposition \ref{prophomotopy} (i) holds.}~\Rightarrow~
	 				A^{n,n}_{X}(\star D)=A^{n,n}_{X}(\log D)+\partial A^{n-1,n}_{X}(\star D)\,.
	 				\end{equation}
	 		\end{dfn}

	 			\begin{lem}\label{lemindependenceondecompositionaseX}
		 	The map $\mathfrak{r}$ is independent of the choice of the decomposition $\omega=\alpha+\partial\beta$ in \eqref{eqndecompositionforrmap}.
		 	\end{lem}
	 		\begin{proof}
	 			
	 			The absolute convergence 
	 			of $\int_{X}\alpha, \alpha\in A_{X}^{n,n}(\log D)$
	 			 follow from partition of unity and a local analysis using polar coordinates (see Lemma \ref{lemintegrabilityoflogpart} for a stronger version).
	 						
	 			To prove the independence on the decomposition, one needs to show that
	 		 			\begin{equation}\label{eqnvanishingonpartialexacttermforhypersurfacearrangement}
	 				\int_{X}\partial \beta=0\, ~\textrm{for}~\beta\in 
	 				A^{n-1,n}_{X}(\star D)\,,~\partial \beta\in A^{n,n}_{X}(\log D)\,.
	 			\end{equation}
	 		 				For each irreducible component $D_{a},a=1,2,\cdots, N$ of $D$, we take a sufficiently small $\varepsilon_{a}>0$ and let $B_{\varepsilon_{a}}(D_{a})$ be the disk bundle of radius $\varepsilon_{a}$ over the smooth hypersurface $D_{a}$.
	 			It is easy to see that one can construct a set $K_{\varepsilon}\subseteq \bigcup_{a,b}(B_{\varepsilon_{a}}(D_{a})\cap B_{\varepsilon_{b}}(D_{b}))$
	 		 			such that 
	 			$
	 			\bigcup_{a=1}^{N}B_{\varepsilon_{a}}(D_{a})-K_{\varepsilon}
	 			$ consists of a disjoint union of disk bundles
	 			$\bigsqcup_{a=1}^{N}B_{\varepsilon_{a}}^{'}(D_{a})$ with
	 			$B_{\varepsilon_{a}}^{'}(D_{a})\subseteq B_{\varepsilon_{a}}(D_{a})$.
	 		Then by the integrability of $\partial\beta$, the relation $\bar{\partial}\beta=0$ implied by type reasons, and Stokes theorem, one has
	 			\begin{eqnarray*}
	 			\int_{X}\partial\beta&=&
	 			\lim_{(\varepsilon_{1},\cdots,\varepsilon_{N})\rightarrow 0}\int_{X-\bigcup_{a}B_{\varepsilon_{a}}(D_{a})}\partial\beta\\
	 			&=&\lim_{(\varepsilon_{1},\cdots,\varepsilon_{N})\rightarrow 0}\left(\int_{K_{\varepsilon}}\partial\beta+
	 			\int_{X-\bigsqcup_{a}B_{\varepsilon_{a}}^{'}(D_{a})}
	 			\partial\beta\right)\\
	 				&=&\lim_{(\varepsilon_{1},\cdots,\varepsilon_{N})\rightarrow 0}\left(\int_{K_{\varepsilon}}\partial\beta+
	 			\int_{X-\bigsqcup_{a}B_{\varepsilon_{a}}^{'}(D_{a})}
	 		d\beta\right)\\
	 			&=&\lim_{(\varepsilon_{1},\cdots,\varepsilon_{N})\rightarrow 0}\int_{K_{\varepsilon}}\partial\beta-
	 		\sum_{a=1}^{N}\lim_{(\varepsilon_{1},\cdots,\varepsilon_{N})\rightarrow 0}\int_{\partial B_{\varepsilon_{a}}^{'}(D_{a})}
	 		\beta\,.
	 			\end{eqnarray*}
 			By the construction that $\lim_{(\varepsilon_1,\cdots,\varepsilon_{N})\rightarrow 0}\mathrm{vol}(K_{\varepsilon})=0$, the first term is zero.	 			
	 			Applying the same local calculation proving \eqref{eqnvanishingofholomorphicresiduebytypereasonscaseDol0}
	 			to each of the summands  in the second term, one then obtains
	 			the desired vanishing.	 			
	In particular,  this vanishing result is independent of the choice of the metric used to define the disk bundles $B_{\varepsilon_a}(D_a)$.
	 			\end{proof}
	 		
	 				With the independence on the decomposition \eqref{eqndecompositionforrmap} in Lemma \ref{lemindependenceondecompositionaseX} proved, the well-definedness of $\mathfrak{r}$ requires the following relation as  a consistency condition
	 		\begin{equation}\label{eqnconsistencyofdefinitionofregularizedintegral}
	 		\mathfrak{r}(\partial  \beta)=0\,  ~\textrm{for}~\beta\in 
	 		A^{n-1,n}_{X}(\star D)\,,~\partial \beta\in A^{n,n}_{X}(\star D)\,.
	 		\end{equation}

		 Part (i)  of Theorem \ref{thmformdefiningcurrent} also extends to the map $\mathfrak{r}$.
		 	
		 	\begin{prop}\label{propformdefiningcurrentcaseD}		
		 	 	Any $\omega\in A^{n,n}_{X}(\log D)+\partial A^{n-1,n}_{X}(\star D)$ defines a  current $T_{\omega}$:  for any test form $f\in A^{0,0}_{X}$ ,		 	
		 		\begin{equation}\label{eqnregularizedintegral=principalvaluecaseD}
		 			T_{\omega}(f)=\mathfrak{r}( \omega f)\,.
		 		\end{equation}	
		 	Furthermore, one has $T_{\partial \beta}(1)=0$ for $\beta\in A_{X}^{n-1,n}(\star D)$.
		 			\end{prop}
	 		
	 		\begin{proof}
	Note that $A^{n,n}_{X}(\log D)+\partial A^{n-1,n}_{X}(\star D)$ is a $C^{\infty}_{X}$ module. Hence for 
any test function $f\in A^{0,0}_{X}$, one has $\omega f\in A_{X}^{n,n}(\star D)$.
From  Definition \ref{dfnregularizedintegralcaseD} and Lemma \ref{lemindependenceondecompositionaseX}, we see that
$\mathfrak{r}(\omega f)$ exists and is independent of choices.
It is then clear that the assignment $f\mapsto \mathfrak{r}(\omega f)$ gives rise to a linear functional on $A_{X}^{0,0}$
and thus defines a current $T_{\omega}$.		
		That $T_{\partial \beta}(1)=0$ follows  the consistency relation  \eqref{eqnconsistencyofdefinitionofregularizedintegral}
		of Definition \ref{dfnregularizedintegralcaseD}.

	\end{proof}

\subsubsection{Relation to principal value currents}

	We now discuss the relation between the current $T_{\omega}$ determined via Proposition \ref{propformdefiningcurrentcaseD} and the principal value current  defined by \cite{Dolbeault:1970, Herrera:1971}, thus generalizing part (ii)  of Theorem \ref{thmformdefiningcurrent}.
	
	Recall from \cite{Herrera:1971} that $\omega$ defines a principal value current  $\mathrm{P.V.}_{\omega}$ which by a partition of unity argument admits a local presentation 	
		\begin{equation}\label{eqnHLcurrent}
		\mathrm{P.V.}_{\omega}(\psi):=
	\lim_{\varepsilon\rightarrow 0}\int_{|s|\geq \varepsilon} \omega\cdot \psi\,
	\end{equation}
for a test form $\psi$, where $s$ is a local defining equation of $D$. 
It is shown in  \cite{Herrera:1971} that the result is
	independent of the choice of $s$, the metric $|\cdot|$, and the local presentations of $\omega$ that patch globally.
	Furthermore, one has
	\begin{equation}\label{eqnresiduecurrentvanishing}
	\mathrm{P.V.}_{\partial \beta}(1)=0\,,\quad \beta\in A_{X}^{n-1,1}(\star D)\,.
	\end{equation}
	These results are proved
		using Hironaka’s  resolution of singularities, which reduces to the case where $D$ is normal crossing
		and where the vanishing  \eqref{eqnresiduecurrentvanishing} essentially follows by type reasons as in \eqref{eqnvanishingofholomorphicresiduebytypereasonscaseDol0}.
		However, unlike \eqref{eqnvanishingonpartialexacttermforhypersurfacearrangement} in Lemma \ref{lemindependenceondecompositionaseX} where
		one has the additional integrability, here the establishing of  vanishing is more involved.
			\begin{rem}
	As a comparison, other currents such as the Coleff-Herrera  products	\cite{Coleff:1978} of the residue current 
	\cite{Herrera:1971} 
	\[
	(\bar{\partial}{1\over s}):\psi\mapsto \lim_{\varepsilon\rightarrow 0}\int_{|s|= \varepsilon} \psi\,,
	\]
	do depend on choices like the integration chain.  See also \cite{Passare:1988} for details.
	\end{rem}
		
		\begin{cor}\label{coruniquenessofregularization}
			Let $\mathfrak{r}:A^{n,n}_{X}(\log D)+\partial A^{n-1,n}_{X}(\star D)\rightarrow \mathbb{C}$ be the linear map defined as in Definition \ref{dfnregularizedintegralcaseD}.
			Then 
			\begin{equation}\label{eqnregularizedintegralisHLcurrent}
				\mathfrak{r}(\omega)=\mathrm{P.V.}_{\omega}(1)\,,\quad 
				\forall\, \omega\in  A^{n,n}_{X}(\log D)+\partial A^{n-1,n}_{X}(\star D)\,.
			\end{equation}
		In particular, when  Proposition \ref{prophomotopy} (i) holds
		which is the case for cases $\ref{caseC}, (\textbf{NCD})$, one has  	
		\begin{equation}\label{eqnregularizedintegralisHLcurrentfull}
	\mathfrak{r}(\omega)=\mathrm{P.V.}_{\omega}(1)\,,\quad 
	\forall\, \omega\in   A^{n,n}_{X}(\star D)\,.
\end{equation}
				 
		\end{cor}
		\begin{proof}
			
		By the integrability of elements in $A_{X}^{n,n}(\log D)$ and 
		\eqref{eqnHLcurrent} we see that
		$\mathrm{P.V.}_{\omega}(1)$
		agrees with $\mathfrak{r}(\omega)=\int_{X}\omega$ for $\omega\in A_{X}^{n,n}(\log D)$.
		Furthermore, similar to the vanishing of $\mathfrak{r}$ in \eqref{eqnconsistencyofdefinitionofregularizedintegral}, it also vanishes on $ \partial\beta,\beta\in A_{X}^{n-1,n}(\star D)$ by \eqref{eqnresiduecurrentvanishing} (although for a different reason).
		That they agree then follows from the independence of their definitions on the decomposition of $\omega$, as shown in Lemma \ref{lemindependenceondecompositionaseX}.
						The second part of the claim follows from the relation
  \eqref{eqnfullspace} ensured by Proposition \ref{prophomotopy} (i).
		\end{proof}
	
\subsection{Trace map on current cohomology}
\label{sectracemaponcurrentcohomology}	

Proposition \ref{propformdefiningcurrentcaseD} gives a map
\begin{equation}\label{eqnPVintegrablecurrent}
	T:  A^{n,n}_{X}(\log D)+\partial A^{n-1,n}_{X}(\star D) \rightarrow {D}'^{n,n}\,,\quad \omega\mapsto T_{\omega}(\cdot)=\mathfrak{r}(\omega\cdot)\,.
\end{equation}
We now study properties of  this map more closely.
By degree reasons, one has
\[
\partial T_{\omega}=\bar{\partial}T_{\omega}=dT_{\omega}=0\,.
\]
Thus $T_{\omega}$ gives an element $[T_{\omega}]$ in the current cohomology
$ 
H^{n}({D}_{X}^{'\bullet,n},\partial)$.

\begin{dfn}\label{dfnontracemaponcurrentcohomology}
	Define  the following trace map on current cohomology
	\begin{equation}
		\mathrm{Tr}: H^{n}(D_{X}^{'\bullet,n}, \partial)\rightarrow \mathbb{C}\,,\quad 
		T\mapsto T(1)\,.
	\end{equation}
By the definition of $\partial t, t\in C^{2n-1}(D_{X}^{'\bullet,n},\partial)$, one has 
\begin{equation}\label{eqnvanishingconditiononregularizationscheme}
\mathrm{Tr}\,(\partial t)=(\partial t)(1)=(-1)^{2n}(\partial 1)=0\,.
\end{equation} 
Thus the map $\mathrm{Tr}$ is well-defined.
\end{dfn}

\begin{thm}\label{thmresidueformulaforcurrents}
	Let $\mathfrak{r}:A^{n,n}_{X}(\log D)+\partial A^{n-1,n}_{X}(\star D)\rightarrow \mathbb{C}$ be the linear map defined as in Definition \ref{dfnregularizedintegralcaseD} and $T$ be the map determined by $\mathfrak{r}$ via Proposition \ref{propformdefiningcurrentcaseD}.
	
Then 
	for any $\omega\in A^{n,n}_{X}(\log D)+\partial A^{n-1,n}_{X}(\star D)$, there exists a 
	smooth form $\omega_{\mathrm{Dol}}\in  A_{X}^{n,n}$ such that 
	\begin{equation}\label{eqnresidueintegralforcurrents}
		\mathfrak{r}(\omega)=\int_{X}\omega_{\mathrm{Dol}}=\langle [X], [\omega_{\mathrm{Dol}}]
		\rangle \,.
	\end{equation}
	Here $[X]$ is the fundamental class of $X$, and $[\omega_{\mathrm{Dol}}]\in H^{n}(D_{X}^{\bullet,n},\partial)\cong H^{2n}(X)$ is the cohomology class of the smooth form $\omega_{\mathrm{Dol}}$ that only depends on $\omega$.
	
			In particular, when  Proposition \ref{prophomotopy} (i) holds
	which is the case for cases $\ref{caseC}, (\textbf{NCD})$, the relation \eqref{eqnresidueintegralforcurrents} holds for any $\omega\in  A^{n,n}_{X}(\star D)$.

\end{thm}
\begin{proof}
	By Definition
	\ref{dfnontracemaponcurrentcohomology}, one has 
	\begin{equation}\label{eqnregularizedintegral=trace}
	\mathfrak{r}(\omega)=T_{\omega}(1)=\mathrm{Tr}\,([T_{\omega}])\,.
		\end{equation}
	From the conjugate Dolbeault resolution of $\overline{ \Omega}^{n}_{X}$ by the fine sheaves of smooth forms and the fine sheaves of currents, one has 
	\[
	\overline{ \Omega}^{n}_{X}
	\xrightarrow{\mathrm{quasi-isomorphism}} (\mathcal{D}_{X}^{\bullet,n}, \partial)\xrightarrow{\mathrm{quasi-isomorphism}}  (\mathcal{D}_{X}^{'\bullet,n}, \partial)\,.
	\]
	It follows that
\begin{equation}\label{eqnHnisomorphismofcurrentandsmoothforms}
	H^{n}(D_{X}^{'\bullet,n}, \partial)\cong \mathbb{H}^{n}(\mathcal{D}_{X}^{'\bullet,n}, \partial)
	\cong  \mathbb{H}^{n}(\mathcal{D}_{X}^{\bullet,n}, \partial)
	\cong 	H^{n}(D_{X}^{\bullet,n}, \partial)
	\,.
	\end{equation}
Pick any smooth form $\omega_{\mathrm{Dol}}\in C^{n}(D_{X}^{\bullet,n}, \partial)$ representing 
	the class $[T_{\omega}]\in H^{n}(D_{X}^{'\bullet,n}, \partial)$ under the above identification. The resulting integral current $T_{\omega_{\mathrm{Dol}}}$ then satisfies 
	\[
	\mathrm{Tr}\,([T_{\omega}])
		=\mathrm{Tr}\,([T_{\omega_{\mathrm{Dol}}}])=T_{\omega_{\mathrm{Dol}}}(1)
		=
		\int_{X}\omega_{\mathrm{Dol}}
		\,.
	\]
		By type reasons one has $d\omega_{\mathrm{Dol}}=0$. By Stokes theorem, the integral above descends to the claimed pairing between homology and cohomology.
\end{proof}

 \subsection{Representatives of current cohomology classes in terms of smooth forms}	\label{secrepresentativeofcurrentcohomologysintermsofsmooth}

 A  practical way of obtaining a representative $\omega_{\mathrm{Dol}}$  of the 
 cohomology class $[\omega_{\mathrm{Dol}}]\in H^{n}(D_{X}^{\bullet,n}, \partial)$, 
 without using partition of unity,
 can be obtained as follows.

\begin{prop}\label{propconstuctingclassusingDolCech}
Let the notation be as above.
	The class  $[\omega_{\mathrm{Dol}}]\in H^{n}(D_{X}^{\bullet,n}, \partial)$ in Theorem \ref{thmresidueformulaforcurrents}
	can be concretely represented by a smooth form   $
	\omega_{\mathrm{Dol}}\in  A_{X}^{n,n}$ using the conjugate Dolbeault-Cech double complex.
	
\end{prop}
\begin{proof}
 
 We take a good cover $\mathcal{U}$ of $X$ consisting of Stein open sets.
 Using the current conjugate Dolbeault-Cech double complex, 
 starting from $T_{\omega}\in D_{X}^{'n,n}= C^{0}(\mathcal{D}_{X}^{'n,n})$
 one obtains a representative in $\check{C}^{n}(X, \mathcal{D}_{X}^{'0,n})$ of the class 
  $[T_{\omega}]\in \check{H}^{n}(X, \mathcal{D}_{X}^{'0,n})$ that is the image from
 $\check{C}^{n}(X,\overline{ \Omega}^{n}_{X})$.
 Replacing the above current conjugate Dolbeault-Cech double complex 
 by conjugate Dolbeault-Cech double complex and run the above diagram chasing backwards, one then obtains a
 desired element in $\check{C}^{0}(\mathcal{D}_{X}^{n,n})=A_{X}^{n,n}$ representing the corresponding cohomology class
 in $H^{n}(D_{X}^{\bullet,n},\partial)$.
 
 \end{proof}

We now explain how this works for the $n=1$ case  in case $\ref{caseC}$.
The study on other cases  is only notationally different.
For a good cover $\mathcal{U}$ of $X=C$,
 we can take $\mathcal{U}=\{U,V\}$, where $U=X-D$ is Stein and 
$V$ is a small analytic neighborhood of $D$.
We can also work with the universal cover or other cover of $X$, whichever is more convenient.
The corresponding double complex is displayed in 
	Figure 	\ref{figure:figconjDolCechdoublecomplexcurrent} below.
	
	\begin{figure}[h]
		\centering
		\[
		\xymatrix{
	&	\check{C}^{1}(\overline{\Omega}_{X}^{1})\ar@{.>}[r]^{\hookrightarrow}	& \check{C}^{1}(\mathcal{D}_{X}^{'0,1})\ar[rr]^{\partial}&&  \check{C}^{1}(\mathcal{D}_{X}^{'1,1})&&\\
			&&  && &&\\
			&& \check{C}^{0}(\mathcal{D}_{X}^{'0,1})\ar[rr]^{\partial}\ar[uu]^{\delta}&& 
			\check{C}^{0}(\mathcal{D}_{X}^{'1,1})\ar[uu]^{\delta} &&\\
								&&  &&D_{X}^{'1,1}\ar@{.>}[u]_{\hookrightarrow}  &&			
		}
		\]	
		\caption{Double complex computing $\check{H}^{1}(X,\overline{\Omega}^{1}_{X})$ when $X=C$.}
		\label{figure:figconjDolCechdoublecomplexcurrent}
	\end{figure}	
For the current $T_{\omega}$ determined from $\omega\in A_{X}^{1,1}(\star D)$  via 
Theorem \ref{thmformdefiningcurrent}, it is injected into the element $(U_{i}, T_{i})\in \check{C}^{0}(\mathcal{D}_{X}^{'1,1})$.
On each Stein open $U_{i}$ in $\mathcal{U}$ one takes the $\partial$-primitive $(U_{i},S_{i})$ (this is possible by the fact that 
 $(\mathcal{D}_{X}^{'\bullet,n}, \partial)$ is a resolution).
 Then the element $\delta (U_{i},S_{i})=(U_{ij}, -S_{i}+S_{j})\in \check{C}^{1}(\mathcal{D}_{X}^{'0,1})$ is the image of 
 a cocycle in $\check{C}^{1}(\overline{\Omega}^{1}_{X})$ whose  class belongs to
 $ \check{H}^{1}(X,\overline{\Omega}^{1}_{X})$.
As explained in Proposition \ref{propconstuctingclassusingDolCech},
running the diagram chasing in the conjugate Dolbeault-Cech double complex, one obtains a cocycle
in $\check{C}^{0}(\mathcal{D}_{X}^{1,1})$ that is injected from $\Gamma(X, \mathcal{D}_{X}^{1,1})=D_{X}^{1,1}=A_{X}^{1,1}$.

\begin{rem}\label{improvedalgorithm}
The above procedure can sometimes be improved as follows. Taking a singular differential form as the $\partial$-primitive, say $\rho_{i}$, of $\omega$ in $U_{i}$.
The local analysis  in  Lemma \ref{lemeliminatinglog}  shows that $T_{\rho_i}$ exists as an element in $\mathcal{D}_{X}^{'0,1}(U_{i})$.
Note that in general it is not possible to arrange such that $\rho_i\in A^{n-1,n}_{X}(\star D)|_{U_{i}}$ as already illustrated in Example \ref{exwpn=1definitioncomputation}.
From \eqref{eqnpartialcurrent} we 
have $T_{\omega}|_{U_{i}}=T_{\partial \rho_{i}}=\partial T_{\rho_{i}}$
and thus 
 the  class 
 $[T_{\omega_{\mathrm{Dol}}}]$
 is represented by $(U_{ij}, -T_{\rho_{i}}+T_{\rho_{j}})\in \check{C}^{1}(\mathcal{D}_{X}^{'0,1})$.
 In nicest cases (e.g., when working with a universal cover of $X$), one has 
 \[
 -\rho_{i}+\rho_{j}=-\sigma_{i}+\sigma_{j}\quad\textrm{for some}~ (U_{i},\sigma_{i})\in \check{C}^{0}(\mathcal{D}_{X}^{0,1})\,,
 \quad\textrm{such that}~(U_{i},\partial\sigma_{i})\in \check{C}^{0}(\mathcal{D}_{X}^{0,1})\cap \mathrm{ker}\,\delta\,.
 \]
Since 
 $(U_{ij}, -\sigma_{i}+\sigma_{j})\in \check{C}^{1}(\mathcal{D}_{X}^{0,1})$ represents  the  class 
 $[T_{\omega_{\mathrm{Dol}}}]$, a diagram chasing 
 then gives a smooth form
$
 (U_{i},\partial \sigma_{i})\in \check{C}^{0}(\mathcal{D}_{X}^{1,1})
\cong \Gamma(X, \mathcal{D}_{X}^{1,1})=A_{X}^{1,1}$
 representing the class $[\omega_{\mathrm{Dol}}]\in  H^{1}(D_{X}^{\bullet,1}, \partial)$.
\end{rem}

\subsection{Proof of Theorem \ref{thmgeneralizedregularizationmapAintro}}

The proof of our first main result Theorem \ref{thmgeneralizedregularizationmapAintro} 
is given as follows.
\begin{proof}[Proof of Theorem \ref{thmgeneralizedregularizationmapAintro}]
	This follows by combining Theorem \ref{thmresidueformulaforcurrents} and 
	Proposition \ref{propconstuctingclassusingDolCech}.
\end{proof}

\begin{ex}\label{exfindingsmoothformrepresentativeofcurrentcohomology}
	Consider as in Example \ref{exwpn=1definitioncomputation}, Example \ref{exwpn=1residuecomputation} the case with $X=E, D=0$.
	
	Let $\omega=\widehat{Z}dz\wedge \psi$,  $\psi={d\bar{z}\over \bar{\tau}-\tau}$.
		To find a desired form $\omega_{\mathrm{Dol}}$ we follow the discussions in 
	Remark \ref{improvedalgorithm}.	
	Take $\mathcal{U}$ to be a cover by parallelograms $U_{i}$ on the universal cover $\mathbb{C}$,
	with the transition function being translations by certain lattice points. 
	Then one has for any $U_{i}$
	\[
	T_{\omega}=\partial T_{\rho}=T_{\partial \rho}\,,\quad 
	\rho=\ln (\bar{\theta}\widehat{\theta})\cdot  \psi\,,
	\]
	and thus $ (U_{ij}, -\rho_{i}+\rho_{j})=\delta (U_{i},\rho_{i})=0\in  \check{H}^{1}(X,\overline{\Omega}^{1}_{X})$. 
	
	Consider also $\omega=\wp(z)dz\wedge \psi$.
	Then again following the discussions in 
	Remark \ref{improvedalgorithm}, one has
	\[
	\rho_{i}=-\partial\widehat{Z}-\widehat{\eta}_{1} z_{i}\wedge {d\bar{z}_{i}\over \bar{\tau}-\tau}\,.
	\]
	The corresponding class in $\check{H}^{1}(X,\mathcal{D}_{X}^{'0,1})$ is given by
	$(U_{ij}, (z_{i}-z_{j})\widehat{\eta}_{1} \wedge {d\bar{z}_{i}\over \bar{\tau}-\tau})$.
	The resulting class  
	in  $H^{1}(D_{X}^{\bullet,1},\partial)$ is then represented by 
	\[
	(U_{i}, -\widehat{\eta}_{1} dz_{i}\wedge {d\bar{z}_{i}\over \bar{\tau}-\tau})\in  \check{C}^{0}(\mathcal{D}_{X}^{1,1})\,.
	\]
	This is the injection of the  smooth form
	$ -\widehat{\eta}_{1} dz\wedge {d\bar{z}\over \bar{\tau}-\tau}\in 
	\Gamma(X, \mathcal{D}_{X}^{1,1})= A_{X}^{1,1}$.
	\xxqed
\end{ex}

		\section{Regularized integrals of factorizable differential forms}
		\label{secfactorizableforms}

		In this section we prove  Theorem \ref {thmgeneralizedregularizationmapBintro}  by providing a mixed-Hodge theoretic formulation of the 
		regularized integration maps.
	
	Throughout this section, we work with the following map $\mathfrak{s}$
	that extends $\dashint_{X}$.

	\begin{dfn}\label{dfnsmap}
		Let $\mathfrak{F}$ be the space of factorizable forms defined as in \eqref{eqndfnfactorizableforms}.
Let $\mathfrak{s}: \mathfrak{F}+\partial A^{n-1,n}_{X}(\star D)\rightarrow \mathbb{C}$ be a linear map satisfying
	\begin{equation}\label{eqnassumptionsons}
		\mathfrak{s}(\alpha)=\int_{X}\alpha ~ ~\textrm{for}~~ \alpha\in A^{n,n}_{X}(\log D)\cap (\mathfrak{F}+\partial A^{n-1,n}_{X}(\star D))\,,\quad 
		\mathfrak{s}(\partial\beta)=0~~ \textrm{for}~~ \beta\in A^{n-1,n}_{X}(\star D)\,.
	\end{equation}
	\end{dfn}
			We first
			make the following simple observation.

			\begin{lem}\label{lemregularizedintegralforfactorizableformsincaseHA}
Let the notation be as above.
Then
	for any factorizable form $\omega=\phi\wedge\psi\in \mathfrak{F}$, $\mathfrak{s}(\omega)$
	 depends on $\phi$ through its de Rham cohomology class $[\phi]_{\mathrm{dR}}\in  H^{n}(A_{X}^{\bullet}(\star D),d)$.
		\end{lem}
\begin{proof}
	By \eqref{eqnvanishingonpartialexacttermforhypersurfacearrangement}, the relations
	\eqref{eqnassumptionsons} themselves are compatible.
Now take $\omega=\phi\wedge \psi\in\mathfrak{F}$, with
 $\phi\in 
A^{n,0}_{X}(\star D )\cap \mathrm{ker}\,\bar{\partial}, \psi\in
			A^{0,n}_{X}\cap \mathrm{ker}\,\partial$.
			By degree reasons, $\phi$ determines a cohomology class $[\phi]_{\mathrm{dR}}$ in $H^{n}(A_{X}^{\bullet}(\star D),d)$.	
			Consider  the action $\phi\mapsto \phi+d \gamma,\gamma\in A_{X}^{n-1}(\star D)$.
			By type reasons one has $\bar{\partial}\gamma\wedge \psi=0$.
			Thus the above action gives the change 
			\[
			\phi\wedge \psi\mapsto \phi\wedge \psi+\partial(\gamma\wedge \psi)\,.
			\]
		From
	$		dA^{n-1}_{X}(\star D)\wedge \psi \subseteq \partial A^{2n-1}_{X}(\star D)
		$
		and	\eqref{eqnassumptionsons}, we see that   $\mathfrak{s}$ depends on $\phi$ through its  cohomology class $[\phi]_{\mathrm{dR}}\in H^{n}(A_{X}^{\bullet}(\star D),d)$.
		\end{proof}

		Suppose Proposition \ref{prophomotopy} (ii) holds, then one can find a representative $\Phi\in A^{n}_{X}(\log D)$ of the corresponding cohomology class $[\phi]_{\mathrm{dR}}\in H^{n}(A_{X}^{\bullet}(\star D),d)$.
						Since $\Phi\wedge \psi\in  A^{n,n}_{X}(\log D) $,  by Lemma \ref{lemregularizedintegralforfactorizableformsincaseHA} 
	this tells
		\begin{equation}\label{eqnreductioofstointwithlogmerocomparison}
			\text{Proposition \ref{prophomotopy} (ii) holds.}~\Rightarrow~
		\mathfrak{s}(\phi\wedge \psi)=\mathfrak{s}(\Phi\wedge \psi)=\int_{X}\Phi\wedge \psi\,.
		\end{equation}
	Combining \eqref{eqnassumptionsons},  we see that
the map $\mathfrak{s}$ provides a substitute of the regularized integration operator $\dashint_{X}$ on factorizable forms, for cases beyond $\ref{caseC},(\textbf{NCD})$.
	For this reason, in what follows we also call  $\mathfrak{s}$ the regularized integration map and also write  $\mathfrak{s}(\phi\wedge \psi)=	\mathfrak{s}([\phi]_{\mathrm{dR}}\wedge \psi)$. \\
	

	We shall  next study more closely the dependence of 
$\mathfrak{s}(\phi\wedge \psi)$ on the corresponding cohomology class $[\phi]_{\mathrm{dR}}\in H^{n}(A_{X}^{\bullet}(\star D),d)$, for the case $\ref{caseHA}$ in which Proposition \ref{prophomotopy} (ii) holds and in which $H^{n}(A_{X}^{\bullet}(\star D),d)\cong H^{n}(A_{X}^{\bullet}(\log D),d)\ \cong H^{n}(U)$.

		For clarity, we only give the details for the special case of
$\ref{caseHA}$ where $D$  is a simple normal crossing divisor,
that is, case (\textbf{NCD}).
The main tool is Deligne's spectral sequence \cite{Deligne:1971} in computing the mixed Hodge structure for the
complement $U$.  
The results apply to the Orlik-Solomon spectral sequence (see \cite{Dupont:2015} for details) for the general situation in case $\ref{caseHA}$.
The nonsingular projectivity condition on $X$ then ensures the degeneration of the spectral sequence at the $E_{2}$ page.
Alternatively, one could apply the desingularization argument to reduce case $\ref{caseHA}$  to  case (\textbf{NCD}).

		\subsection{Mixed Hodge-theoretic formulation of regularized integrals}

		\subsubsection{Deligne's mixed Hodge structure}

		We recall Deligne's spectral sequence \cite{Deligne:1971} in computing the mixed Hodge structures for the
		complement $U$ of a simple normal crossing divisor $D$ in a nonsingular projective algebraic variety $X$.

		Put a linear order on the components of $D$ as $D_{a},a=1,2,\cdots, N$. For any \emph{ordered} tuple $I\subseteq [N]:=\{1,2,\cdots,N\}$ consisting of an increasing sequence,
		define
		\begin{equation}\label{eqnstratumDI}
		D_{I}=\bigcap_{a\in I} D_{a}\,,\quad
		D_{[j]}=\bigsqcup_{I:\,|I|=j}D_{I}\,.
		\end{equation}
		
				Consider the bifiltered complex of sheaves $(\mathcal{A}^{\bullet}_{X}(\log D),d)$, equipped with the usual 
				decreasing Hodge filtration $F^{\bullet}$ induced by the 
	 $dz,d\bar{z}$ types of differential forms (see \eqref{eqndfnHodgefiltration}),
	 and the increasing 
	 filtration $W_{\bullet}$ that keeps track of the pole order along $D$.
	 Explicitly, one has 
	 \begin{equation}\label{eqndfnoffiltrations}
	 F^{p}\mathcal{A}_{X}^{\bullet}(\log D)
	 :=\bigoplus_{k\geq p}\mathcal{A}^{k,\bullet-k}_{X}(\log  D )\,,\quad 
	W_{m}\mathcal{A}^{\bullet}_{X}(\log D)
	:=\mathcal{A}_{X}^{\bullet-m}\wedge \mathcal{A}^{m}_{X}(\log D)\,.
	 \end{equation}
	 	 As mentioned in the proof of Proposition \ref{prophomotopy} (ii),
one has $\mathbb{H}^{k}(X, \mathcal{A}^{\bullet}_{X}(\log D))
	 \cong H^{k}(X,A_{X}^{\bullet}(\log D))\cong H^{k}(U)$.
	 The above data on $\mathcal{A}^{\bullet}_{X}(\log D)$ induces 
 filtrations on $\mathbb{H}^{k}(X,\mathcal{A}^{\bullet}_{X}(\log D))\cong H^{k}(U)$
	 \begin{eqnarray}\label{eqnfiltrationsoncohomology}
	 	F^{p}H^{k}(X,A_{X}^{\bullet}(\log D))&=&\im \left(H^{k}(X,F^{p}A_{X}^{\bullet}(\log D))
	 	\rightarrow H^{k}(X,A_{X}^{\bullet}(\log D))\right)
	 	\,,\nonumber\\
	 	W_{j}H^{k}(X,A_{X}^{\bullet}(\log D))&=&\im \left(H^{k}(X,W_{j}A_{X}^{\bullet}(\log D))
	 	\rightarrow H^{k}(X,A_{X}^{\bullet}(\log D))\right)
	 	\,.
	 \end{eqnarray}
	 This puts a mixed Hodge structure on $H^{k}(U)$, with the weight filtration given by $\widetilde{W}_{\bullet}=W_{\bullet}[-k]$. 
	
 The spectral sequences for the corresponding filtered complexes compute the 
 the associated graded on the cohomology 
 from the cohomology of the associated graded.  
	 The spectral sequence associated to the Hodge filtration $F^{\bullet}$ degenerates at the 1st page, and one has
	 \begin{equation}\label{eqnHodgefiltrationdegeneration}
	 \mathrm{gr}^{p}_{F}H^{p+q}(X,A_{X}^{\bullet}(\log D))\cong 
	  H^{p+q}(X,\mathrm{gr}^{p}_{F}A_{X}^{\bullet}(\log D))\,.
	 \end{equation}
	 The spectral sequence associated to the weight filtration  $W_{\bullet}$ degenerates at its 2nd page.
	 In what follows we focus on this one. 	 
	 	For $r\geq 0$, we denote by $d_{r}$  the differential on the $E_{r}$-page  of the spectral sequence associated to the filtered complex $(A_{X}^{\bullet}(\log D), W_{\bullet})$. We also denote $[x]_{E_{r}}$ to be the image of $x$ on the $E_{r}$ page, if it exists. 
	 
	 				For this spectral sequence, one has for $j,q\geq 0$ (with $k=q-j$)
		\begin{eqnarray}\label{eqnE0E1pageentries}
		E_{0}^{-j,q}&=&\Gamma(X, \mathrm{gr}^{W}_{j}\mathcal{A}^{k}_{X}(\log D))\,,\nonumber
\\
				E_{1}^{-j,q}&=&H^{k}(X, \mathrm{gr}^{W}_{j}A^{\bullet}_{X}(\log D))\,,\nonumber\\
					E_{2}^{-j,q}&=&\mathrm{gr}^{W}_{j}H^{k}(X, A^{\bullet}_{X}(\log D))\,.
		\end{eqnarray}
	
				Recall the quotient sheaf sequence 
		\begin{equation}\label{eqnquotientsheafsequence}
			0\rightarrow W_{j-1}\mathcal{A}^{k}_{X}(\log D)\rightarrow W_{j}\mathcal{A}^{k}_{X}(\log D)\rightarrow \mathrm{gr}^{W}_{j}\mathcal{A}^{k}_{X}(\log D)\rightarrow 0\,.
		\end{equation}
		Applying the global section functor and using
		the fact that the sheaves
		$W_{j}\mathcal{A}^{k}_{X}(\log D),W_{j-1}\mathcal{A}^{k}_{X}(\log D)$ are fine and thus $\Gamma$-acyclic, we have an induced sequence of global sections
		\begin{equation}\label{eqnquotientsheafsequence}
			0\rightarrow \Gamma(W_{j-1}\mathcal{A}^{k}_{X}(\log D))\rightarrow  \Gamma(W_{j}\mathcal{A}^{k}_{X}(\log D))\rightarrow^{p_{j}} \Gamma(\mathrm{gr}^{W}_{j}\mathcal{A}^{k}_{X}(\log D))\rightarrow 0\,.
		\end{equation}
		This gives 
		\begin{equation}\label{eqnE0entryexplicit}
			E_{0}^{-j,k+j}=\Gamma(X, \mathrm{gr}^{W}_{j}\mathcal{A}^{k}_{X}(\log D))=\mathrm{gr}^{W}_{j}\Gamma(X, \mathcal{A}^{k}_{X}(\log D))\,. 
		\end{equation}		
		
The iterated Poincar\'e residue map gives a quasi-isomorphism of complexes of sheaves
\begin{equation}\label{eqnE0pageiso}
\mathrm{res}_{j}=\bigoplus_{I:\,|I|=j}\mathrm{res}_{I}:
 \mathrm{gr}^{W}_{j}\mathcal{A}^{\bullet}_{X}(\log D)
 \cong (a_{j})_{*}\mathcal{A}^{\bullet}_{D_{[j]}}[-j]\,,
\end{equation}
where $a_{j}: D_{[j]}\rightarrow X$ is the closed immersion.
Passing to cohomology, it gives an isomorphism of pure Hodge structures of weight $q$
\begin{equation}\label{eqnE1pageiso}
\mathrm{Res}_{j}=\bigoplus_{I:\,|I|=j}\mathrm{Res}_{I}:
E_{1}^{-j,q}\xrightarrow{\cong}
H^{k}(D_{[j]}, \mathcal{A}^{\bullet}_{D_{[j]}}[-j])\otimes \mathbb{Q}(-j)\,.
\end{equation}
The differential $d_{1}$ on the $E_{1}$ page is
a morphism of pure Hodge structures and is strict with respect to the Hodge filtration $F^{\bullet}$.
It is induced by the operator $d$,  such that $\mathrm{res}_{j}\circ d_{1}\circ \mathrm{res}_{j-1}$ is an alternating sum of Gysin maps.
For the special case with $j=0$, one has 
\begin{eqnarray}\label{eqnE0E1E2entryexplicitj=0}
	E_{0}^{0,k}&=&A^{k}_{X}\,,\nonumber\\
	E_{1}^{0,k}&=&{(E_{0}^{0,k}\cap \ker\,d_{0})/ d_{0}E_{0}^{0,k-1}}={A^{k}_{X}/ dA^{k-1}_{X}}\,,\nonumber\\
	E_{2}^{0,k}&=&{(E_{1}^{1,k}\cap \ker\,d_{1})/ d_{1}E_{1}^{0,k-1}}=H^{k}(X, A_{X}^{\bullet})/ dH^{k-1}(X, \mathrm{gr}^{W}_{1}A^{\bullet}_{X}(\log D))\,.
\end{eqnarray}

Deligne's splitting on $(H^{\bullet}(U),F^{\bullet}, W_{\bullet})$ is given by
\begin{equation}\label{eqndfnDelignsplitting}
	H^{\bullet}(U)=\bigoplus_{p,q}I_{p,q}\,,\quad 
	I^{p,q}:=F^{p}\cap\widetilde{W}_{p+q}\cap \left(
	\overline{F^{q}}\cap \widetilde{W}_{p+q}+\sum_{i\geq 2}\overline{F^{q-i+1}}\cap \widetilde{W}_{p+q-i}
	\right)\,.
\end{equation}
The subspaces $I_{p,q}$ satisfy
\begin{equation}
\widetilde{W}_{k}\otimes \mathbb{C}=\bigoplus_{p+q\leq k}I^{p,q}\,,\quad
F^{p}=\bigoplus_{i\geq p} \bigoplus_{q}I^{i,q}\,.	
\end{equation}
Furthermore, under the natural projection  $\mathrm{pr}_{j}: \widetilde{W}_{p+q}\rightarrow \mathrm{gr}^{\widetilde{W}}_{p+q}$ with $j=p+q-k$, one has the isomorphism 
\begin{equation}\label{eqnprofDeligne}
	\mathrm{pr}_{j}:
I^{p,q}\cong F^{p}\mathrm{gr}^{\widetilde{W}}_{p+q}\cap 	\overline{F^{q}}\mathrm{gr}^{\widetilde{W}}_{p+q}\,.
\end{equation}

\subsubsection{Regularized integrals of factorizable forms}

	We now 
	use mixed Hodge structure on $H^{n}(U)$
 to study 
 the regularized integration operator
$\mathfrak{s}$.

\begin{dfn}\label{dfnprojectiontoIn0}
	Consider case  $(\textbf{NCD})$.
	Let $\omega=\phi\wedge \psi$ be a factorizable form and the notation be as above.
	Let $\pi_{I^{p,q}}:H^{n}(U)\rightarrow I^{p,q}$ be the projection
	using Deligne's splitting.
	Under the Deligne splitting of $H^{n}(U)$, one has the Deligne splitting of $[\phi]_{\mathrm{dR}}$ given by
	\begin{equation}\label{DelignesplittingofphidR}
		[\phi]_{\mathrm{dR}}=
		\sum_{p,q:\, p+q\geq n}\pi_{I^{p,q}}([\phi]_{\mathrm{dR}})
		\in \bigoplus_{p,q:\, p+q\geq n}I^{p,q}\,.
	\end{equation}	
\end{dfn}

\begin{lem}\label{lemprojectiontoIn0byMHS}	
	Consider case  $(\textbf{NCD})$.
	Let the notation be as above. Then one has
	\[
	\mathfrak{s}(\phi\wedge \psi)=	\mathfrak{s}\left( \pi_{I^{n,0}}([\phi]_{\mathrm{dR}})\wedge \psi\right)\,.
	\]
\end{lem}
\begin{proof}
Lemma \ref{lemregularizedintegralforfactorizableformsincaseHA}
tells that the regularized integral only depends on the class $[\phi]_{\mathrm{dR}}\in H^{n}(U)$.
For any $q\geq 1$, by \eqref{eqndfnDelignsplitting} one has 
\[
\pi_{I^{p,q}}([\phi]_{\mathrm{dR}})\in \overline{F^{q}}H^{n}(X,A_{X}^{\bullet}(\log D))
=\im \left(H^{n}(X,\overline{F^{q}}A_{X}^{\bullet}(\log D))\rightarrow H^{n}(X,A_{X}^{\bullet}(\log D))\right)\,.
\]
That is, this cohomology class in $H^{n}(U)$ can be represented by a d-closed form in $A_{X}^{\bullet,q}(\log D)$ with $q\geq 1$.
Applying Lemma \ref{lemregularizedintegralforfactorizableformsincaseHA} and  \eqref{eqnreductioofstointwithlogmerocomparison} to this cohomology class, we see that by type reasons
\[
\mathfrak{s}(\pi_{I^{p,q}}([\phi]_{\mathrm{dR}})\wedge \psi)=0\,,\quad q\geq 1\,.
\]
It follows that 
\[
	\mathfrak{s}([\phi]_{\mathrm{dR}}\wedge \psi)=	\sum_{p,q:\, p+q\geq n}\mathfrak{s}(\pi_{I^{p,q}}([\phi]_{\mathrm{dR}})\wedge \psi)=
	\mathfrak{s}\left(\pi_{I^{n,0}}([\phi]_{\mathrm{dR}})\wedge \psi\right)\,.
\]
	\end{proof}

	\subsubsection{Regularized integral as a cohomological pairing on compact space}
	\label{seclifttoE1E0page}
	
	In order to reduce the regularized integral in Lemma \ref{lemprojectiontoIn0byMHS}	 from the noncompact space $U$ to the compact space $X$, we use the relation between $\mathrm{gr}^{W}_{0}H^{n}(U)$ and $H^{n}(X)$.
	
	\begin{lem}\label{lemInospacedescription}	Consider case  $(\textbf{NCD})$.
Then  
		$I^{n,0}=F^{n}H^{n}(U)\cap W_{0}H^{n}(U)$ is canonically identified with  $H^{0}(X,\Omega_{X}^{n})$.
	\end{lem}
\begin{proof}
By the definition of $W_{\bullet}$ on $H^{n}(U)$, the fact that 
$W_{0}=\mathrm{gr}^{W}_{0}$ and \eqref{eqnE0E1E2entryexplicitj=0}, one has 
\begin{eqnarray*}
I^{n,0}
&=&F^{n}H^{n}(X,A_{X}^{\bullet}(\log D))\cap  \im \left(H^{n}(X, \mathrm{gr}^{W}_{0}A_{X}^{n}(\log D))\rightarrow H^{n}(X,A_{X}^{\bullet}(\log D))\right)\\
&=&F^{n}H^{n}(X,A_{X}^{\bullet}(\log D))\cap 
{H^{n}(X, \mathrm{gr}^{W}_{0}A_{X}^{n}(\log D))\over   dH^{n-1}(X, \mathrm{gr}^{W}_{1}A^{\bullet}_{X}(\log D))}
\,.
\end{eqnarray*}
By the strictness  \eqref{eqnHodgefiltrationdegeneration} of $d_{1}$ with $F^{\bullet}$ and that $H^{k}(X, \mathrm{gr}^{W}_{j}A^{\bullet}_{X}(\log D))$ carries a pure Hodge structure of weight $k$, we then have as desired
	\begin{eqnarray}\label{eqnIn0asquotient}
I^{n,0}
&\cong &{F^{n}H^{n}(X, \mathrm{gr}^{W}_{0}A_{X}^{\bullet}(\log D))/  F^{n}dH^{n-1}(X, \mathrm{gr}^{W}_{1}A^{\bullet}_{X}(\log D))}\nonumber\\
&= &{F^{n}H^{n}(X, A_{X}^{\bullet})/  dF^{n}H^{n-1}(X, \mathrm{gr}^{W}_{1}A^{\bullet}_{X}(\log D))}\nonumber\\
&= &{H^{0}(X, \Omega_{X}^{n}})\,.
\end{eqnarray}
\end{proof}

Realizations of $\pi_{I^{n,0}}([\phi]_{\mathrm{dR}})$
on the compact space $X$ can now be constructed by lifting $\mathrm{pr}_{0} \pi_{I^{n,0}}([\phi]_{\mathrm{dR}})$ to the $E_{1}$ and $E_{0}$ pages of the spectral sequence as follows.
Denote the natural projection
\begin{equation}\label{eqnprj}
\mathrm{pr}_{j}: W_{j}H^{n}(U)\rightarrow E_{2}^{-j,n+j}=\mathrm{gr}^{W}_{j}H^{n}(U)\,.
\end{equation}
One has from \eqref{eqnprofDeligne} that 
\begin{equation}
\mathrm{pr}_{0}: I^{n,0}\cong F^{n}\mathrm{gr}^{W}_{0}H^{n}(U)=F^{n}E_{2}^{0,n}\,.
\end{equation}
Since $W_{-1}H^{n}(U)=0$, the projection $\mathrm{pr}_{0}$ is actually the identity map. 

\begin{dfn}\label{dfnliftofIn0toE1E0pages}
		Let $[\mathrm{pr}_{0}\pi_{I^{n,0}}([\phi]_{\mathrm{dR}})]_{E_{r}}$ be a lift of
	$\mathrm{pr}_{0}\pi_{I^{n,0}}([\phi]_{\mathrm{dR}})$
	along the map $F^{n}E_{r}^{0,n}\rightarrow F^{n}E_{2}^{0,n}$, for $r=0,1$.
	Set
	$[\phi]_{\mathrm{Hodge}}
	=[\mathrm{pr}_{0}\pi_{I^{n,0}}([\phi]_{\mathrm{dR}})]_{E_{1}}$.
	See 	Figure \ref{figure:figliftalongspectralsequence} below for an illustration.
	\begin{figure}[h]
		\centering
		\[
		\xymatrix{
			&&&&		I^{n,0}\ar[dd]^{\mathrm{pr}_{0}} \\
			&&&&\\
			E_{0}^{0,n}\ar[rr]&	&E_{1}^{0,n}\ar[rr] &&E_{2}^{0,n} 
		}
		\]	
		\caption{Lift $\mathrm{pr}_{0}(I^{n,0})$ to the $E_{1}$ and $E_{0}$ pages of the spectral sequence.}
		\label{figure:figliftalongspectralsequence}
	\end{figure}

	\end{dfn}

\begin{dfn}\label{dfntracemaponholantiholforms}

Define the trace map to be the pairing
	\[
	\mathrm{Tr}: 
	H^{0}(X,\Omega_{X}^{n})\otimes H^{0}(X,\overline{\Omega}_{X}^{n})
	\xrightarrow{\wedge } H^{0}(X,\mathcal{A}_{X}^{n,n})\xrightarrow{\int_{X}}\mathbb{C}\,.
	\]
	This map  is a special case of the one in Definition \ref{dfnontracemaponcurrentcohomology}.  By Hodge decomposition, it extends to a map on
	$H^{n}(X)\otimes H^{0}(X,\overline{\Omega}_{X}^{n})$ which we denote by the same symbol.
\end{dfn}

\begin{prop}\label{propliftofstoE1E0}	
Consider case  $(\textbf{NCD})$.
 Then one has
	\[
	\mathfrak{s}(\phi\wedge \psi)=\mathrm{Tr}\left( 	[\phi]_{\mathrm{Hodge}}\wedge \psi\right)=
	\int_{X}[\mathrm{pr}_{0}\pi_{I^{n,0}}([\phi]_{\mathrm{dR}})]_{E_{0}}\wedge \psi\,.
	\]
\end{prop}
\begin{proof}
	From 
	Lemma \ref{lemprojectiontoIn0byMHS}	 and $\mathrm{pr}_{0}=\mathrm{id}$, one has	
			\[
	\mathfrak{s}(\phi\wedge \psi)=	\mathfrak{s}\left(\mathrm{pr}_{0} \pi_{I^{n,0}}([\phi]_{\mathrm{dR}})\wedge \psi\right)\,.
	\]	
	By 
	\eqref{eqnIn0asquotient} in Lemma \ref{lemInospacedescription}, the element
	$	[\phi]_{\mathrm{Hodge}}$ is the unique lift of  $\mathrm{pr}_{0}\pi_{I^{n,0}}([\phi]_{\mathrm{dR}})$ along the inclusion $F^{n}E_{1}^{0,n}\hookrightarrow F^{n}E_{2}^{0,n}$. Since both 
	$[\phi]_{\mathrm{Hodge}}$ and $\psi$ are elements in
	 $E_{1}^{0,n}=H^{n}(X)$, one has
	\[
	\mathfrak{s}\left(\mathrm{pr}_{0} \pi_{I^{n,0}}([\phi]_{\mathrm{dR}})\wedge \psi\right)=\mathrm{Tr}\left( 	[\phi]_{\mathrm{Hodge}}\wedge \psi\right)\,.
	\]	
	By its construction and $\eqref{eqnE0E1E2entryexplicitj=0}$, the form $ [\mathrm{pr}_{0}\pi_{I^{n,0}}([\phi]_{\mathrm{dR}})]_{E_{0}}\in A_{X}^{n}$ is
	a representative of $	[\phi]_{\mathrm{Hodge}}\in H^{n}(X)$. Therefore, 
	\[
	\mathrm{Tr}\left( 	[\phi]_{\mathrm{Hodge}}\wedge \psi\right)=\int_{X}[\mathrm{pr}_{0}\pi_{I^{n,0}}([\phi]_{\mathrm{dR}})]_{E_{0}}\wedge \psi\,.
	\]
Combining the above three relations, one finishes the proof.

\end{proof}

A lift $[\mathrm{pr}_{0}\pi_{I^{n,0}}([\phi]_{\mathrm{dR}})]_{E_{0}}$
can also be provided using the descriptions of the $I^{p,q}$'s as subspaces of $H^{n}(U)$.
Recall
the Deligne splitting of $[\phi]_{\mathrm{dR}}\in H^{n}(U)$ given in \eqref{DelignesplittingofphidR}
\begin{equation}\label{DelignesplittingofphidRrecalled}
	[\phi]_{\mathrm{dR}}=
	\sum_{p,q:\, p+q\geq n}\pi_{I^{p,q}}([\phi]_{\mathrm{dR}})
	\in \bigoplus_{p,q:\, p+q\geq n}I^{p,q}\,.
\end{equation}

\begin{dfn}
\label{dfnlifttoE-1page}
		Consider case  $(\textbf{NCD})$.
Set  $j=p+q-n$ as before. 
	For any $j$ with $0\leq j\leq n$, define as in \eqref{eqnquotientsheafsequence} and \eqref{eqnprj} the natural projections 
\begin{equation}
p_{j}: W_{j}A^{n}_{X}(\log D)\rightarrow \mathrm{gr}^{W}_{j}A^{n}_{X}(\log D)  \,,\quad 
\mathrm{pr}_{j}: W_{j}H^{n}(U)\rightarrow \mathrm{gr}^{W}_{j}H^{n}(U)\,.
\end{equation}
Let $J^{p,q}\subseteq  H^{n}(W_{j}A_{X}^{\bullet}(\log D))$
be the preimage of $I^{p,q}$ along the map
$ H^{n}(W_{j}A_{X}^{\bullet}(\log D))\rightarrow I^{p,q}$,
 and $\Gamma^{p,q}\subseteq  \Gamma(W_{j}A_{X}^{n}(\log D))\cap \ker\,d$ be the preimage of $I^{p,q}$
along the map
\begin{equation}\label{eqnprojectiontoIpq}
 	\Gamma(X,W_{j}A_{X}^{n}(\log D))\rightarrow  H^{n}(W_{j}A_{X}^{\bullet}(\log D))\rightarrow I^{p,q}\,.
\end{equation}
See the commutative diagram in Figure \ref{figure:figliftpiIpqalongspectralsequence} below for an illustration.		
\begin{figure}[h]
	\centering
	\[
	\xymatrix{
		&&	\bigoplus_{j} 	\bigoplus_{p+q=n+j}\Gamma^{p,q}\ar@{.>}[dd]^{\oplus_{j}p_{j}}\ar[r]	&	\bigoplus_{j} 	\bigoplus_{p+q=n+j}J^{p,q}\ar[r] &	\bigoplus_{j}	\bigoplus_{p+q=n+j}I^{p,q}\ar[dd]_{\cong}^{\mathrm{pr}=\oplus_{j}\mathrm{pr}_{j}}	& \\
	&	&&&&\\
	&	\bigoplus_{j}W_{j}A_{X}^{n}(\log D)\ar@{.>}[r]^{\oplus_{j}p_{j}}&	\bigoplus_{j}E_{0}^{-j,n+j}\ar[r]	&	\bigoplus_{j}E_{1}^{-j,n+j}\ar[r] &	\bigoplus_{j}E_{2}^{-j,n+j} 
	}
	\]	
	\caption{Lift to the $E_{r}$ page of the spectral sequence.}
	\label{figure:figliftpiIpqalongspectralsequence}
\end{figure}
For $r=0,1,2$, denote by  
\begin{equation}\label{eqnlifttoE-1page}
		[[\phi]_{\mathrm{dR}}]_{E_{r}}
	=\sum_{j=0}^{n}
	\sum_{p,q:\, p+q= n+j}\phi_{r}^{p,q}
\end{equation}
 an $E_{r}$ lift
of 	$[\phi]_{\mathrm{dR}}$ in \eqref{DelignesplittingofphidRrecalled} along the maps in the top row.
Similarly, for $r=-1,0,1,2$, denote by $[\mathrm{pr}[\phi]_{\mathrm{dR}}]_{E_{r}}$ an $E_{r}$ lift
of 	$\mathrm{pr}[\phi]_{\mathrm{dR}}$  along the maps in the bottom row.
In each case, we denote 
\begin{equation}\label{eqndfndelta0}
	\phi_{r}^{j}:=
	\sum_{p,q:\, p+q=n+j}\phi_{r}^{p,q}\,.
\end{equation}
\end{dfn}

	Since $W_{-1}=0$ and thus $\mathrm{pr}_{0}=\mathrm{id}$, 
	 there is no difference between $[[\phi]_{\mathrm{dR}}]_{E_{r}}$ and $[\mathrm{pr}[\phi]_{\mathrm{dR}}]_{E_{r}}$ as far as only the $I^{p,q}, p+q=n$ components are concerned.
	As $W_{0}H^{n}(U)$ has a pure Hodge structure,
	for $p+q=n$, the decomposition \eqref{eqndfndelta0} is nothing but the Hodge decomposition when $r=1$
	and can be chosen to respect the Hodge type when $r=0$.
	
	\begin{rem}
	Take an element $x_{j}$ from $W_{j}A_{X}^{n}(\log D)$ (or more generally from $W_{j}A_{X}^{k}(\log D)$ with $1\leq j\leq k$).
	In order for $p_{j}(x_{j})\in E_{0}^{-j,n+j}$ to descend to the $E_{2}$ page, it needs to satisfy the conditions
	\begin{equation}\label{eqnd0d1closedness}
		d_{0}p_{j}(x_{j})=0\in E_{0}^{-j,n+1+j}\,,\quad 
		d_{1}[p_{j}(x_{j})]_{E_{1}}=0\in E_{1}^{-j+1,n+j}\,.
	\end{equation}
	These conditions can be made more direct as conditions on $x_{j}$, by applying
	standard facts about the spectral sequence associated to the filtered complex $(A_{X}^{\bullet}(\log D), W_{\bullet})$ which give \begin{eqnarray}\label{eqnd0d1closednessmoredirect}
		\ker\,d_{r}&=&W_{j}A_{X}^{n}(\log D)\cap d^{-1}W_{j-r-1}A_{X}^{n}(\log D)+
		W_{j-1}A^{n}_{X}(\log D)
		\,,\nonumber\\
		\im\,d_{r}&=& W_{j}A_{X}^{n}(\log D)\cap dW_{j+r} A_{X}^{n-1}(\log D)+
		W_{j-1}A^{n}_{X}(\log D)
		\,.
	\end{eqnarray}
Thanks to the $E_{2}$-degeneration of the spectral sequence, the condition $p_{j}(x_{j})\in \ker\,d_{0}\cap \ker\,d_{1}$ in \eqref{eqnd0d1closedness} becomes
	$p_{j}(x_{j})\in \cap_{r=0}^{\infty}\ker\,d_{r}$. By \eqref{eqnd0d1closednessmoredirect}, the latter is equivalent to 
	\begin{equation}\label{eqndrclosednessmoredirect}
		x_{j}\in W_{j}A_{X}^{n}(\log D)\cap \ker\,d+
		W_{j-1}A^{n}_{X}(\log D)\,.
	\end{equation}
	This condition is indeed satisfied by a lift $	[[\phi]_{\mathrm{dR}}]_{E_{0}}$ by the description of $\Gamma^{p,q}$ in
Definition \ref{dfnlifttoE-1page}.
\end{rem}


\begin{cor}\label{cordelta0integral}
	Consider case  $(\textbf{NCD})$.
	\begin{enumerate}[i).]
		\item 

Suppose $[\mathrm{pr}[\phi]_{\mathrm{dR}}]_{E_{-1}}=\sum_{j=0}^{n}
\sum_{p,q:\, p+q= n+j}\phi_{-1}^{p,q}$ is an $E_{-1}$ lift of $\mathrm{pr}[\phi]_{\mathrm{dR}}$.
Then 
	\[
\mathfrak{s}(\phi\wedge \psi)=
\int_{X}	\phi_{-1}^{n,0}\wedge \psi=
\int_{X}	\phi_{-1}^{0}\wedge \psi\,.
\]
\item 
Suppose $[[\phi]_{\mathrm{dR}}]_{E_{0}}=\sum_{j=0}^{n}
\sum_{p,q:\, p+q= n+j}\phi_{0}^{p,q}$ is an $E_{0}$ lift of $[\phi]_{\mathrm{dR}}$. Then one has furthermore
	\[
		\mathfrak{s}(\phi_{0}^{j}\wedge \psi)=\int_{X}\phi_{0}^{j}\wedge \psi=0\,\quad \text{if}\quad  j\geq 1\,.
	\]
		\end{enumerate}
\end{cor}
\begin{proof}
		\begin{enumerate}[i).]
		\item

	Since $W_{-1}A_{X}^{n}(\log D)=0$ and thus $p_{0}=\mathrm{id}$, we have 
	$
	\phi_{-1}^{n,0}=		\phi_{0}^{n,0}\in \Gamma (X,\Omega_{X}^{n})
	$ and $\phi_{-1}^{0}\in A_{X}^{n}\cap \ker\,d$.
Applying  Lemma \ref{lemregularizedintegralforfactorizableformsincaseHA} and  \eqref{eqnreductioofstointwithlogmerocomparison}, one has $\mathfrak{s}(\phi_{-1}^{n,0}\wedge \psi)=\int_{X}	\phi_{-1}^{n,0}\wedge \psi$ and $\mathfrak{s}(\phi_{-1}^{0}\wedge \psi)=\int_{X}	\phi_{-1}^{0}\wedge \psi$.
By	Proposition	\ref{propliftofstoE1E0} and type reasons, one then has the desired claim.
	\item 
		First observe that by construction, one has $	d\phi_{0}^{p,q}=0$, thus 
	$\mathfrak{s}(\phi_{0}^{p,q}\wedge \psi)$ makes sense.	
 From the same reasoning in 
	Lemma \ref{lemprojectiontoIn0byMHS}, one has the vanishing of $\mathfrak{s}(\phi_{0}^{p,q}\wedge \psi)$ for $q\geq 1$.
	This yields the desired claim.
			\end{enumerate}
\end{proof}

\subsection{Proof of Theorem \ref{thmgeneralizedregularizationmapBintro}}

The proof of our second main result Theorem \ref{thmgeneralizedregularizationmapBintro} 
is given as follows.
\begin{proof}[Proof of Theorem \ref{thmgeneralizedregularizationmapBintro}]

	For simplicity, we only consider the special case  (\textbf{NCD}) of
	$\ref{caseHA}$. The proof for the general case $\ref{caseHA}$ is only notationally more difficult. 
	The proof for the (\textbf{NCD}) case now follows from Proposition \ref{propliftofstoE1E0} and Corollary \ref{cordelta0integral}.
	
\end{proof}

\subsection{Splittings of factorizable forms}
\label{secsplittingoffactorizableforms}

Proposition \ref{propliftofstoE1E0} and Corollary \ref{cordelta0integral}	 tell that 
the regularized integral only depends on the data of $\phi$
on the $E_{2}=E_{\infty}$ page and can be lifted to the $E_{1}, E_{0}$ pages
in terms of a cohomology pairing and ordinary integration on $X$.
However, these results do not yield a constructive way to compute
the quantities $\phi_{r}^{n,0}$  from $\phi$.
For the practical purpose of evaluating the regularized integral, it is desirable to make them more concrete, preferably on the form level.

Even for a pure Hodge structure, finding the $I^{p,q}$ components of a cohomology class, that is, Hodge decomposition, 
is already nontrivial and requires more delicate machinery such as harmonic forms.
It is thus natural to work with the lift of $I^{0}:=\bigoplus_{p+q=n}I^{p,q}$ as a whole.
Due to Corollary \ref{cordelta0integral}, this is enough for the purpose of evaluating the regularized integral.

It turns out that finding the $E_{r}$ lifts $\phi_{r}^{0}$ requires understanding the lifts of all components $\mathrm{pr}_{j}[\phi]_{\mathrm{dR}}$ instead of the component $\mathrm{pr}_{0}[\phi]_{\mathrm{dR}}$  only. Towards finding an $E_{0}$ lift of $[\phi]_{\mathrm{dR}}$, what we are after is in fact a section of the projection 
\begin{equation}\label{eqnsplittingofgrmap}
\bigoplus_{j=0}^{n}\mathrm{pr}_{j}: H^{n}(U)=\bigoplus_{j=0}^{n} I^{j}\rightarrow \bigoplus_{j=0}^{n} \mathrm{gr}^{W}_{j}H^{n}(U)\,,\quad I^{j}:=\bigoplus_{p,q:\,p+q=n+j}I^{p,q}\,.
\end{equation}
Conceptually, a construction using the residue map goes as follows.
Recall that the residue map $\mathrm{res}_{j}=\oplus_{I:\,|I|=j}\mathrm{res}_{I}$ is defined on the sheaf level
\begin{equation}\label{eqnresiduejonsheaves}
	\mathrm{res}_{j}:
	(W_{j}\mathcal{A}_{X}^{p,\bullet}(\log D),\bar{\partial})\rightarrow ((a_{j})_{*}\mathcal{A}_{D_{[j]}}^{p-j,\bullet},\bar{\partial}) \,,\quad 
	(W_{j}\mathcal{A}_{X}^{\bullet}(\log D),d)\rightarrow ((a_{j})_{*}\mathcal{A}_{D_{[j]}}^{\bullet-j},d) \,.
\end{equation}
Let $\mathcal{K}_{j}^{p,\bullet}\subseteq (W_{j}\mathcal{A}_{X}^{p,\bullet}(\log D),\bar{\partial})$ and $\mathcal{K}_{j}^{\bullet}\subseteq (W_{j}\mathcal{A}_{X}^{\bullet}(\log D),d)$
be the kernel of the corresponding  maps
which are complexes of fine sheaves.
By the $\bar{\partial}$-Poincar\'e lemma, one has the Dolbeault resolutions
\[
0\rightarrow 
W_{j}\Omega_{X}^{p}(\log D)\rightarrow 
(W_{j}\mathcal{A}_{X}^{p,\bullet}(\log D),\bar{\partial})\,,\quad 
0\rightarrow 
\Omega_{D_{j}}^{p}\rightarrow 
(\mathcal{A}_{D_{[j]}}^{p,\bullet},\bar{\partial})\,.
\]
This implies that the middle and right vertical arrows in the diagram in Figure \ref{figure:figkernelofresiduemap} below are quasi-isomorphisms.
By the five lemma the left vertical arrow is also a quasi-isomorphism.
\begin{figure}[h]
	\centering
	\[
	\xymatrix{
		0\ar[r]	&W_{j-1}\Omega_{X}^{\bullet}(\log D)\ar[r]\ar@{^{(}->}[dd]   & W_{j}\Omega_{X}^{\bullet}(\log D)\ar[rr]^{\mathrm{res}_{j}}\ar@{^{(}->}[dd]&&
		(a_{j})_{*}\Omega_{D_{[j]}}^{\bullet-j}\ar[r]\ar@{^{(}->}[dd]& 0\\
		&&&&\\
		0\ar[r]	& 
		\mathcal{K}_{j}^{\bullet}\ar[r]
		&W_{j}\mathcal{A}_{X}^{\bullet}(\log D)\ar@/^/[rr]^{\mathrm{res}_{j}}&& (a_{j})_{*}\mathcal{A}_{D_{[j]}}^{\bullet-j}\ar@/^/[ll]^{s_{j}}\ar[r] &0
	}
	\]	
	\caption{Kernel of residue map.}
	\label{figure:figkernelofresiduemap}
\end{figure}
In particular, the following inclusions induce quasi-isomorphisms
\begin{equation}\label{eqnquasiisoofcomplexK}
W_{j-1}\Omega_{X}^{\bullet}(\log D)\hookrightarrow W_{j-1}\mathcal{A}_{X}^{\bullet}(\log D)
\hookrightarrow \mathcal{K}_{j}^{\bullet}\,.
\end{equation}
Taking hypercohomology, one thus has 
\begin{equation}\label{eqnhypercohomologyofcomplexK}
	H^{k}(	\mathcal{K}_{j}^{\bullet})\cong  \mathbb{H}^{k}(W_{j-1}\Omega_{X}^{\bullet}(\log D))\cong H^{k}(W_{j-1}A_{X}^{\bullet}(\log D))\,.
\end{equation} 
From the bottom row of the diagram in Figure \ref{figure:figkernelofresiduemap} and \eqref{eqnhypercohomologyofcomplexK}, one obtains a long exact sequence 
\[
\cdots \rightarrow 
H^{k-j-1}(A_{D_{[j]}}^{\bullet})\rightarrow  H^{k}(W_{j-1}A_{X}^{\bullet}(\log D))
\rightarrow H^{k}(W_{j}A_{X}^{\bullet}(\log D))\rightarrow^{\mathrm{Res}_{j}} H^{k-j}(A_{D_{[j]}}^{\bullet})\rightarrow \cdots
\]
For a cohomology class $\varphi_{j}\in H^{k}(W_{j}A_{X}^{\bullet}(\log D))$, by constructing a lift $s_{j}$ of $\mathrm{Res}_{j}$, one obtains 
\begin{equation}\label{eqnsubstracliftofresidueclass}
\varphi_{j}-s_{j}\mathrm{Res}_{j}\varphi_{j}\in \ker\mathrm{Res}_{j}=
\im \left(H^{k}(W_{j-1}A_{X}^{\bullet}(\log D))
\rightarrow H^{k}(W_{j}A_{X}^{\bullet}(\log D))\right)\,.
\end{equation}

Now given any cohomology class $\varphi_{n}\in W_{n}H^{k}(U)=H^{k}(W_{n}A_{X}^{\bullet}(\log D))$, applying the above construction, one has 
\[
\varphi_{n-1}:=\varphi_{n}-s_{n}\mathrm{Res}_{n}\varphi_{n}\in 
\ker\mathrm{Res}_{n}=W_{n-1}H^{k}(A_{X}^{\bullet}(\log D))\,.
\]
Continue this reasoning, by successively subtracting residue classes as in \eqref{eqnsubstracliftofresidueclass},
one eventually reaches 
\[
\varphi_{0}\in W_{0}H^{n}(A_{X}^{\bullet}(\log D))=\im\left( H^{n}(A_{X}^{\bullet})\rightarrow H^{n}(A_{X}^{\bullet}(\log D))\right)\,.
\]
There are two issues that need to be addressed.
\begin{enumerate}[a).]
	\item\label{issuea} The lifts $s_{j},1\leq j\leq n$ need to be constructed carefully so that
	when applying the above construction to the class $[\phi]_{\mathrm{dR}}\in  W_{n}H^{n}(A_{X}^{\bullet}(\log D))$, one obtains the desired class
	$\sum_{p+q=n}\pi_{I^{p,q}} [\phi]_{\mathrm{dR}}\in W_{0}H^{n}(A_{X}^{\bullet}(\log D))$.
	\item\label{issueb} To obtain the desired $E_{0}$ lift $\phi_{0}^{0}$ of $\pi_{I^{0}}[\phi]_{\mathrm{dR}}$ in Definition \ref{dfnlifttoE-1page}, 
	the construction should be carried out on the form level. In particular, this requires constructing $s_{j}$ such that $s_{j}\mathrm{res}_{j}\alpha_{j}$ is a $d$-closed form when $\alpha_{j}$ is so.
\end{enumerate}

In general, explicitly constructing the desired lift on form level can be difficult.
In what follows, we shall only discuss the construction of the $E_{0}$ lift $\phi_{0}^{0}$ in case $\ref{caseC}$. 
Our main reference for this part is \cite{Griffiths:1973}.

\subsubsection{Lifts along the residue map}
\label{secliftsofresiduemap}

As a preparatory step, 
we first recall some basic facts  about residue maps mainly from \cite{Griffiths:1973}, again focusing on  case  $(\textbf{NCD})$. 

We start with a description of the kernel of the residue map.
By its definition, one has $W_{j-1}\mathcal{A}_{X}^{p,q}\subsetneq \mathcal{K}_{j}^{p,q}$ in general.
The sheaf $\mathcal{K}_{j}^{p,q}$ can be described as follows.

Locally near any point in $D_{I}$ with $I=(a_{1}<a_{2}<\cdots <a_{j})$, one can take a polydisk neighborhood $V_{I}$ equipped with coordinates $(s,w)=(s_{a_1},\cdots, s_{a_j},w_{j+1},\cdots, w_{n})$ such that locally $D_{a_t}$ is given by $s_{a_t}=0$ for $t=1,2,\cdots, j$.

\begin{dfn}\label{dfnmIk}  Consider case  $(\textbf{NCD})$.	
	For $I=(a_{1}<a_{2}<\cdots<a_{j})$, define the sheaves
	\begin{equation*}\label{eqnmI}
		\mathfrak{m}_{I}^{p,q}:=\sum_{t=1}^{j} \mathcal{O}_{X}(-D_{a_{t}})\otimes \mathcal{A}_{X}^{p,q}\,,\quad 
		\overline{\mathfrak{m}}_{I}^{p,q}:=\sum_{t=1}^{j} \overline{\mathcal{O}_{X}(-D_{a_{t}})}\otimes \mathcal{A}_{X}^{p,q}\,.
	\end{equation*}	
	Set also 	$\mathfrak{m}_{j}^{p,q}=\sum_{I:\,|I|=j}		\mathfrak{m}_{I}^{p,q},
	\overline{\mathfrak{m}}_{j}^{p,q}=\sum_{I:\,|I|=j}		\overline{\mathfrak{m}}_{I}^{p,q}$ and $\mathfrak{m}_{I}^{k}=\oplus_{p+q=k}		\mathfrak{m}_{I}^{p,q},
	\overline{\mathfrak{m}}_{I}^{k}=\oplus_{p+q=k}		\overline{\mathfrak{m}}_{I}^{p,q}$.
\end{dfn}
A local analysis within an aforementioned neighborhood $V_{I}$ gives
\begin{equation}\label{eqnkerresiduesheaf}
	\mathcal{K}_{j}^{p,q}=\sum_{I}\bigwedge_{a_{t}\in I}{ds_{a_t}\over s_{a_t}}\sum_{a_{t}\in I}\mathcal{A}_{X}^{p-j,q-1}{d\bar{s}_{a_t}}
	+\sum_{I} \bigwedge_{a_{t}\in I}{ds_{a_t}\over s_{a_t}}	\overline{\mathfrak{m}}_{I}^{p-j,q}
	+W_{j-1}\mathcal{A}_{X}^{p,q}(\log D)\,.
\end{equation}

The space $W_{j}A_{X}^{p,q}$ can be made more concrete on form level as follows.
\begin{dfn}
\label{dfnangularform} 	Consider case  $(\textbf{NCD})$.
For each divisor $D_{a}$, fix once and for all a holomorphic section $\sigma_{a}$ of the line bundle $\mathcal{O}_{X}([D_{a}])$ whose divisor is $D_{a}$ and a 
Hermitian metric $|\cdot|^2$  on this line bundle.
Denote 
\begin{equation}\label{eqnangularform}
	\eta_a:={1\over 2\pi \mathbf{i}\,}\partial \ln |\sigma_a|^2\in A_{X}^{1,0}(\log D)\,.
\end{equation}
\end{dfn}
\begin{rem}

By construction,
 the resulting form $\bar{\partial}\eta_a$ is a smooth form on $X$ representing the first Chern class $-c_{1}(\mathcal{O}_{X}([D_{a}]))=-\mathrm{PD}([D_{a}])$. 
In view of \eqref{eqnbarpartialcurrent}, the current decomposition
\begin{equation*}\label{eqndecompositionofintegralcurrentonD}
	2\pi \mathbf{i}\,
	\mathrm{res}_{\partial}(\eta_a\wedge \cdot)=
	\bar{\partial}T_{\eta_a}-T_{\bar{\partial}\eta_a}
\end{equation*}
expresses the integration current over $D_{a}$ as the difference between
the $\bar{\partial}$-exact current $\bar{\partial}T_{\eta_a}$ and the integral current $T_{\bar{\partial}\eta_a}$ represented by the smooth form
$\bar{\partial}\eta_a$.
\end{rem}

Elements in $W_{j}A_{X}^{k}(\log D)$ can now be described  using the forms $\eta_{a}$ above.

\begin{lem}\label{lemcharacterizationofkerresj}
	Consider case  $(\textbf{NCD})$. Suppose $1\leq j\leq k$. Then one has 
			\begin{equation}\label{eqnWjspaceglobalexpression}
		W_{j}A_{X}^{k}(\log D)=\sum_{I:|I|=j}\eta_{I}\wedge A_{X}^{k-j}+W_{j-1}A_{X}^{k}(\log D)\,.
		\end{equation}
		That is, 
		any element  $\alpha\in W_{j}A_{X}^{k}(\log D)$  takes the form 
		\begin{equation}\label{eqnWjelementglobalexpression}
			\alpha=\sum_{I:\,|I|=j}\eta_{I}\wedge \alpha_{I} +\beta\,,\quad  \quad \alpha_{I} \in A_{X}^{k-j}\,,\quad \beta\in W_{j-1}A_{X}^{k}(\log D)\,.
		\end{equation}
		In particular, one has the isomorphism between exterior algebras $A_{X}^{\bullet}(\log D)=A_{X}^{\bullet}\langle \eta_{1},\cdots,\eta_{N}\rangle$.

\end{lem}
\begin{proof}

	 For a point in $D_{I}$,
		choose a small neighborhood $V_{I}$ as before.
		By the construction in \eqref{eqnangularform}, locally in $V_{I}$ one has
		\begin{equation*}\label{eqnxiIbehavior}
			\eta_{I}=\bigwedge_{a_{t}\in I}({1\over 2\pi \mathbf{i}}{ds_{a_{t}}\over s_{a_{t}}}+ x_{a_t})
			\,,\quad x_{a_t}\in \mathcal{A}_{X}^{1,0}(V_{I})\,.
		\end{equation*}		
		By the definition of $\alpha\in W_{j}A_{X}^{k}(\log D)$ and a partition of unity argument,		
		one arrives at \eqref{eqnWjelementglobalexpression} immediately.
		The second part of the claim follows by iterating the decomposition \eqref{eqnWjspaceglobalexpression}.

\end{proof}

Applying the global section functor to \eqref{eqnresiduejonsheaves}, one has 
a corresponding residue map on forms
\begin{equation}\label{eqnresiduejonAk}
	\mathrm{res}_{j}:
	W_{j}A_{X}^{p,q}(\log D)\rightarrow A_{D_{[j]}}^{p-j,q} \,,\quad 1\leq j\leq p\,.
\end{equation}
With the forms in Definition \ref{dfnangularform}
 introduced, one can construct lifts of residue forms along the surjection 	$\mathrm{res}_{j}$ in \eqref{eqnresiduejonAk}.

\begin{dfn}
	\label{dfnliftalongresiduemap}
Consider case  $(\textbf{NCD})$.
A lift of the iterated residue \eqref{eqnresiduejonAk} on the form level  is provided by $\rho_{j}=\{\rho_{I}\}_{I:\,|I|=j}$ with
\begin{equation}\label{eqnlifting}
	\rho_{I}( 
	\gamma_{I}) = \eta_{I}\wedge  \widetilde{\gamma_{I}}\,,
\end{equation}
where $\widetilde{\gamma_{I}}$ is a $C^{\infty}$-extension of $\gamma_{I}$ from $D_{I}$ to $X$.
For $\alpha\in W_{j}A_{X}^{k}(\log D)$, denote 
\begin{equation}\label{eqnpureresiduepart}
	R_{j}(\alpha)=\sum_{I:\,|I|=j}\rho_{I}(\mathrm{res}_{I}(\alpha))\in W_{j}A_{X}^{k}(\log D)\,.
\end{equation}

\end{dfn}

Using a partition of unity argument, the extension $\widetilde{\gamma_{I}}$ can be chosen to 
respect Hodge type.
When 
$\gamma_{I}$ is $d$-closed, one can define a solid tubular neighborhood $U_{I}$ of $D_{I}$ carefully such that 
$d\widetilde{\gamma_{I}}=0$ in $U_{I}$. See \cite{Clemens:1977} for details.
In this paper, we always work with such an extension.\\

In general, the map  $\mathrm{res}_{j}$ is not injective, with
$
W_{j-1}A_{X}^{p,q}(\log D)\subseteq 	\ker\,\mathrm{res}_{j}=
K_{j}^{p,q}
$ according to \eqref{eqnkerresiduesheaf}.	Typical elements in the kernel include those of the form $\alpha_{j}- R_{j}\alpha_{j}$, where  $\alpha_{j}\in  W_{j}A_{X}^{p,q}(\log D)$.

\begin{lem}\label{lemdRjalphainWj-1}	
	 Consider case  $(\textbf{NCD})$.
Suppose $\alpha_{j}\in W_{j}A_{X}^{k}(\log D)\cap \ker\,d$. 
\begin{enumerate}[i).]
	\item One has  $dR_{j}\alpha_{j}\in W_{j-1}A_{X}^{k}(\log D)$.
	\item One has $\alpha_{j}-R_{j}\alpha_{j}\in dW_{j}A_{X}^{k-1}(\log D)+W_{j-1}A_{X}^{k}(\log D)$. In particular, $dR_{j}\alpha_{j}\in d W_{j-1}A_{X}^{k}(\log D)$.
\end{enumerate}
\end{lem}
\begin{proof}
	\begin{enumerate}[i).]
		\item 
		Since the residue operator commutes with $d$, 
one has 
$d\mathrm{res}_{I}\alpha_{j}=\mathrm{res}_{I}d\alpha_{j}=0$ for any $I$ with $|I|=j$. 
By the construction of the extension above, one sees that 
$d\widetilde{\mathrm{res}_{I}\alpha_{j}}=0$ near $D_{I}$.
Since $d\eta_{a}=\bar{\partial}\eta_{a}$ is smooth for each $a\in I$, one has 
$d\eta_{I}\in W_{j-1}A_{X}^{n}(\log D)$.
It follows that 
\begin{equation}\label{eqndRjalphaj}
	dR_{j}\alpha_{j}
	=\sum_{I:\,|I|=j}d\eta_{I}\wedge \widetilde{\mathrm{res}_{I}\alpha_{j}}
	+(-1)^{j}
	\sum_{I:\,|I|=j}\eta_{I}\wedge d\widetilde{\mathrm{res}_{I}\alpha_{j}}
	\in W_{j-1}A_{X}^{n}(\log D)\,.
\end{equation} 
\item 
Consider again the spectral sequence associated to the filtered complex $(A_{X}^{\bullet}(\log D),W_{\bullet})$.
By (i) and the assumption $d\alpha_{j}=0$, one has $d(\alpha_{j}-R_{j}\alpha_{j})\in W_{j-1}A_{X}^{k}(\log D)$. This tells that $p_{j}(\alpha_{j}-R_{j}\alpha_{j})\in E_{0}^{-j,k+j}$ 
lies in $\ker\,d_{0}$ and thus 
gives an element $[p_{j}(\alpha_{j}-R_{j}\alpha_{j})]_{E_{1}}\in E_{1}^{-j,k+j}$.
By construction, $\mathrm{res}_{j}p_{j}(\alpha_{j}-R_{j}\alpha_{j})=\mathrm{res}_{j}(\alpha_{j}-R_{j}\alpha_{j})=0$. Applying the isomorphism \eqref{eqnE1pageiso} one has $[p_{j}(\alpha_{j}-R_{j}\alpha_{j})]_{E_{1}}=0$.
That is, 
\begin{equation}\label{eqnalphaj-Rjalphajinimd0}
p_{j}(\alpha_{j}-R_{j}\alpha_{j})\in \im\,d_{0}\,.
\end{equation} 
 Using the description of $\im\,d_{r}$ recalled in \eqref{eqnd0d1closednessmoredirect}, this  gives the desired claim.
\end{enumerate}
\end{proof}

We are mostly concerned with the following subspace of $K_{j}^{p,q}$ whose elements arise from lifts.
\begin{dfn}\label{dfnWj0space}
	Consider case  $(\textbf{NCD})$.
	Define
	\begin{equation}\label{eqnWj0space}
		W_{j}^{\circ}A_{X}^{p,q}(\log D)=
		\{
		\alpha=\sum_{I:\,|I|=j}\eta_{I}\wedge \alpha_{I}\mid 	\alpha_{I}\in \Gamma(X,\overline{\mathfrak{m}}_{I}^{p-j,q})\,,~
		\alpha_{I}\notin
		\Gamma(X,\mathfrak{m}_{I}^{p-j,q})\setminus\{0\}\}\,.
	\end{equation}
	Set also $	W_{j}^{\circ}A_{X}^{k}(\log D)=\bigoplus_{p,q:\,p+q=k}	W_{j}^{\circ}A_{X}^{p,q}(\log D)$.
\end{dfn}

\begin{ex}\label{exkernelofresiduemaps}
	One example of the strict inclusion $W_{j-1}A_{X}^{p,q}(\log D)\subsetneq 		K_{j}^{p,q}$ is (invoke Example \ref{exwpn=1definitioncomputation})
	\[
	\varepsilon=
	(\widehat{Z}^{2}(z)-\wp(z))dz\in	W_{1}^{\circ}A_{E}^{1,0}(\log 0)-W_{0}A_{E}^{1,0}(\log 0)
	\,,\]
	which locally behaves as ${\bar{z}\over z}dz$ near the singularity $z=0$.
	For a general Riemann surface $C$, Laurent coefficients of the Szeg\"o kernel naturally produce elements in $W_{1}^{\circ}A_{C}^{1,0}(\log D)$. See \cite[Section 2.2]{
		Zhou:GW} for related discussions.
	\xxqed
\end{ex}

Case $\ref{caseC}$
is particularly nice for the purpose of constructing the forms \eqref{eqnangularform}
in Definition \ref{dfnangularform} and lifts of residues in Definition	\ref{dfnliftalongresiduemap}.
Indeed, using the product structure $X=C^{n}$, 
to construct the form $\eta_{a}$	
for a component $D_{a}=\Delta_{ij}\subseteq C^{n}$ it
suffices to consider the $n=2$ case and then pull back along the projection $X\rightarrow  C^{2}$.
Now for the $n=2$ case, a section $\sigma$ of $\mathcal{O}_{C\times C}([\Delta])$ can be explicitly constructed by the prime form on the Riemann surface $C$.
For example, for the genus one case it is given by the Jacobi theta function $\theta$, with a particular Hermitian  metric given by
\begin{equation} \label{eqnthetahat}
	|\theta|^2=\theta\bar{\theta} \exp(-2\pi {(\mathrm {im}~z)^2\over \mathrm{im}~\tau})\,.
\end{equation}
One then has $\eta=\partial \log |\theta|^{2}=\widehat{Z}dz$ as in Example \ref{exwpn=1definitioncomputation}.

The product structure $X=C^{n}$ can also be used to construct a coherent system of projections $X\rightarrow D_{I}$ that naturally split the inclusions $D_{I}\rightarrow X$. Then the desired extension $\widetilde{\gamma_{I}}$ in \eqref{eqnlifting} can be constructed explicitly
as the pull back of $\gamma_{I}$ along the projection, \emph{avoiding the use of partition of unity}.
This particular extension satisfies the following additional nice properties.

\begin{lem}\label{lemextensionincaseR}

Consider case $\ref{caseC}$, where the extension  $\widetilde{\gamma_{I}}$ 
is
 constructed
as the pull back of $\gamma_{I}$ along the projection $X\rightarrow D_{I}$.
\begin{enumerate}[i).]
	\item 
If $\gamma=\{\gamma_{I}\}$ is closed, then the resulting extension $\widetilde{\gamma}$ is $d$-closed  and respects the Hodge type of $\gamma$  on $X$.
\item 
If $\alpha\in W_{j}A_{X}^{k}(\log D)$ is $d$-closed, then one has $d\widetilde{\mathrm{res}_{I}(\alpha)}=0$, for any $I$ with $|I|=j$. 

\end{enumerate}
\end{lem}
\begin{proof}
	\begin{enumerate}[i).]
		\item This follows from the construction of 
		$\widetilde{\gamma_{I}}
		$ as the pull back along the projection $X\rightarrow D_{I}$.
		\item 
		One has $d\mathrm{res}_{I}\alpha=\mathrm{res}_{I}d\alpha=0$.
		This gives 
		$
		d\widetilde{\mathrm{res}_{I}\alpha}=\widetilde{(d\mathrm{res}_{I}\alpha)}
		=\widetilde{(\mathrm{res}_{I}d\alpha)}=0
		$.
	\end{enumerate}
\end{proof}

\subsubsection{Construction of $E_{0}$ lift in  case $\ref{caseC}$}

Let us get back to the discussion on constructing lifts of residues on form level, mentioned at the beginning of Section \ref{secsplittingoffactorizableforms}.
 Given $\alpha_{j}\in W_{j}A_{X}^{n}(\log D)\cap \ker \,d$,
one wants to construct a lift $s_{j}\mathrm{res}_{j}\alpha_{j}\in W_{j}A_{X}^{n}(\log D)\cap \ker \,d$ desired  in
\eqref{issuea} and \eqref{issueb}.
To take advantage of the 
explicit presentation of $W_{j}A_{X}^{n}(\log D)$ given by Lemma \ref{lemcharacterizationofkerresj}, a natural attempt is to take
$
s_{j}\mathrm{res}_{j}\alpha_{j}=R_{j}\alpha_{j}
$.
Lemma \ref{lemdRjalphainWj-1} (ii) tells that up to a $d$-exact term,  the form $\alpha_{j}-R_{j}\alpha_{j}$ lies in $W_{j-1}A_{X}^{k}(\log D)$.
However, a computation following \eqref{eqndRjalphaj} tells that  it is not $d$-closed in general.

We now propose the following construction for the $E_{0}$ lift of $[\phi]_{\mathrm{dR}}$.

\begin{cons}\label{consconstructionofE-1lift}
	Consider case  $(\textbf{NCD})$.
One takes an arbitrary representative $\Phi\in W_{n}A_{X}^{n}(\log D)\cap \ker\, d$
of $[\phi]_{\mathrm{dR}}\in H^{n}(U)$.
Starting from $\Phi_{n}:=\Phi$,  one constructs successively a sequence of forms $\{\Phi_{j}\}_{0\leq j\leq n}$ via 
	\begin{equation}\label{eqncorrectionterm}
\Phi_{j-1}:=\Phi_{j}-R_{j}(\Phi_{j})-\varepsilon_{j}-d\gamma_{j}\,,
	\end{equation}
satisfying (recall Definition \ref{dfnWj0space})
\begin{enumerate}
	\item [a).]
$
		\varepsilon_{j}\in \bigoplus_{1\leq k\leq j}	W_{k}^{\circ}A_{X}^{n}(\log D)\,,~
		\gamma_{j}\in  W_{j}A_{X}^{n}(\log D)\,,~
		\Phi_{j-1}\in  W_{j-1}A_{X}^{n}(\log D)
$.
\item [b).] $d(R_{j}(\Phi_{j})+\varepsilon_{j})=0$.
\end{enumerate}

\end{cons}
Set 
\begin{equation}\label{eqndclosedpureresidue}
	\delta_{j}=R_{j}(\Phi_{j})+\varepsilon_{j}\,,~ j\geq 1\,,\quad \quad \delta_{0}:=\Phi_{0}\,.
\end{equation}
Set also $\Phi_{+}=\sum_{j\geq 1}\delta_{j}$.
Construction \ref{consconstructionofE-1lift} then gives
\begin{equation}\label{eqndecomposePhiintolifts}
\Phi=\sum_{j= 1}^{n}d\gamma_{j}+\sum_{j=0}^{n}\delta_{j}=\sum_{j= 1}^{n}d\gamma_{j}+\Phi_{0}+\Phi_{+}\,.
\end{equation}

At this stage, we are able to achieve Construction \ref{consconstructionofE-1lift}  only 
for case $\ref{caseC}$, based on a case-by-case analysis.
Besides the nice properties of case $\ref{caseC}$ given in Lemma \ref{lemextensionincaseR}, what plays a crucial role in the course is the supply of the
rich theory of Riemann surfaces which allows to construct differential forms explicitly (cf. Example \ref{exkernelofresiduemaps}). 
Together with the product structure $X=C^{n}$, they make it possible to construct
the desired decomposition \eqref{eqncorrectionterm} in Construction \ref{consconstructionofE-1lift}. 
These aspects make  $\ref{caseC}\cap (\textbf{NCD})$ distinguished from the more general case $ (\textbf{NCD})$.
See Section \ref{secexamples} for a collection of examples. 
We hope to get back to Construction \ref{consconstructionofE-1lift}  for case $(\textbf{NCD})$ in greater generality in a future investigation, borrowing some techniques from \cite{Zucker:79}.

\begin{ex}\label{exdemonstratingexample}
	Consider $X=E\times E\times E$, with $E=\mathbb{C}/(\mathbb{Z}\oplus \mathbb{Z}\tau)$ as in Example \ref{exwpn=1definitioncomputation}.
	For $1\leq i\neq j\leq 3$,
	let $s_{ij}=z_{i}-z_{j}+c_{ij}$, where $c_{ij}$ are mutually distinct numbers modulo $\mathbb{Z}\oplus \mathbb{Z}\tau$.
	Set	
	\[
	\Phi=\left(\widehat{Z}(s_{12})-\widehat{Z}(s_{13})-\widehat{Z}(s_{32})\right)\left(\widehat{Z}(s_{21})-\widehat{Z}(s_{23})-\widehat{Z}(s_{31})\right)dz_{1}\wedge dz_{2}\wedge dz_{3}\,,
	\]
	where $\widehat{Z}(z)$ is given as in \eqref{eqndfnofZhat}. 
	Let $D=\sum_{i\neq j}(s_{ij}=0)$, then $\Phi\in W_{2}A_{X}^{3}(\log D)$.
		Using \eqref{eqnderivativesofZhat}, 
	it is direct to check that $d\Phi=0$.
	Denote 
	\begin{eqnarray*}
		\Psi&=&\left(\widehat{Z}(s_{12})\widehat{Z}(s_{21})+\widehat{Z}(s_{23})\widehat{Z}(s_{32})+
		\widehat{Z}(s_{31})\widehat{Z}(s_{13})\right)dz_{1}\wedge dz_{2}\wedge dz_{3}\,,\\
		r_{2}&=&(-\widehat{Z}(s_{12})\widehat{Z}(s_{23})-\widehat{Z}(s_{12})\widehat{Z}(s_{31})
		-\widehat{Z}(s_{13})\widehat{Z}(s_{21})	+\widehat{Z}(s_{13})\widehat{Z}(s_{23})
		-\widehat{Z}(s_{32})\widehat{Z}(s_{21})\\
		&&	+\widehat{Z}(s_{32})\widehat{Z}(s_{31}))dz_{1}\wedge dz_{2}\wedge dz_{3}
		\,,\\
		\varepsilon_{2}	&=&
		-{1\over 2}\sum_{i\neq j}(\widehat{Z}^{2}(s_{ij})-\wp(s_{ij}))dz_{1}\wedge dz_{2}\wedge dz_{3}\,,\\
		r_{1}&=&\sum_{i<j}(\widehat{Z}(s_{ij})+\widehat{Z}(s_{ji}))\widehat{Z}(s_{ij}+s_{ji})dz_{1}\wedge dz_{2}\wedge dz_{3}	\,,	\\
		\delta_{0}	&=&
		-{1\over 2}\sum_{i< j}(\widehat{Z}^{2}(s_{ij}+s_{ji})-\wp(s_{ij}+s_{ji}))dz_{1}\wedge dz_{2}\wedge dz_{3}\,.
	\end{eqnarray*}
One then has $\Phi=\Psi+r_{2}$.
	A direct computation using the addition formula of Weierstrass elliptic functions also gives the relation
	$
	\Psi=\varepsilon_{2}+r_{1}+\delta_{0}
	$.		Set now 
	\[
	\delta_{2}=r_{2}+\varepsilon_{2}\,,\quad 
	\delta_{1}=r_{1}\,.
	\]
	Then using 
	\eqref{eqnderivativesofZhat} and the fact  $\varepsilon_{2}\in W^{\circ}_{1}A_{X}^{3}(\log D)$, one has 
	\[
	d\delta_{2}=d\delta_{1}=0\,,\quad 	
	\Phi-\delta_{2}\in W_{1}A_{X}^{3}(\log D)\cap\ker\,d\,,\quad 
	\Phi-\delta_{2}-\delta_{1}\in W_{0}A_{X}^{3}(\log D)\cap\ker\,d\,.
	\]
	This leads to an $E_{0}$ lift of $[\Phi]_{\mathrm{dR}}$ given by 
	$
	\Phi=\delta_{2}+\delta_{1}+\delta_{0}$.
	
	\xxqed
\end{ex}

\begin{rem}\label{remresidueonmeromorphicforms}
	Given a smooth hypersurface $D$ in a smooth compact complex manifold $X$,
	for $\phi\in A^{n,0}_{X}(\star D)$ one has $\mathrm{res}_{\partial}(\phi)=\mathrm{res}_{L}(\phi)$ by \eqref{eqnholresidue}.
	In Deligne's spectral sequence for normal crossing divisors and the Orlik-Solomon spectral sequence for 
	hypersurface arrangement, the residue operations $\mathrm{res}_{I}$
	are iterated residues along smooth hypersurfaces.
	See \cite{Aizenberg:1994} for detailed discussion on iterated residues.
	Therefore, for
	$\Phi=\phi+d \gamma\in A^{n}_{X}(\log D)$,  one has
	\[
	\mathrm{res}_{I}(\Phi)=\mathrm{res}_{I}(\phi+d \gamma)=\mathrm{res}_{I}(\phi)
	+d\mathrm{res}_{I}(\gamma)=\mathrm{res}_{\partial}(\phi)
	+d\mathrm{res}_{I}(\gamma)
	\,.
	\] 
	Therefore, the calculation on the residue forms in Construction \ref{consconstructionofE-1lift} can be directly performed on $\phi$ instead of the logarithmic form $\Phi$ using 	the iterated version of $\mathrm{res}_{\partial}$, up to addition by $d$-exact terms.
	These extra terms do not affect the cohomology classes of the $\delta_{j}$'s in Construction \ref{consconstructionofE-1lift}.
\end{rem}

\begin{rem}
	
	Case $\ref{caseC}$ falls into the case  $\ref{caseHA}$.
	As mentioned earlier,  
	similar considerations for the 
	$(\textbf{NCD})$ case apply to  case  $\ref{caseHA}$: one needs to 
	replace the Deligne spectral sequence by the Orlik-Solomon spectral sequence (see \cite{Dupont:2015} for details).
	However, for the case where $D$ is not simple normal crossing,
	Construction \ref{consconstructionofE-1lift}
	gets  more subtle.
	Firstly, the subtraction of iterated residues in Construction \ref{consconstructionofE-1lift}
	needs to be done according to the stratum $D_{I}$ instead of   the ordered tuple $I$. Secondly, in general an iterated residue of a differential in $A^{\bullet}_{X}(\log D )$ does not lie in $A^{\bullet}_{X}(\log D )$ but only in $A^{\bullet}_{X}(\star D )$. For example, one has
	\[
	\mathrm{res}_{(z=0)}{dz\wedge dw\over zw(z-w)}=-{dw\over w^2}\,.
	\]
\end{rem}

In case $\ref{caseC}$ 
the $E_{0}$ lift in Construction \ref{consconstructionofE-1lift}
produces a  smooth form $\Phi_{0}$ which allows to compute the regularized integral via ordinary integral according to Corollary	\ref{cordelta0integral}.
Such a form  is not unique. 
What we are using here relies on a choice of section $\sigma_{a}$ and  a Hermitian metric $|\cdot|$ as in \eqref{eqnangularform}.
Difference choices lead to different forms $\Phi_{0}$ but the same underlying cohomology classes on the $E_{2}$ page, by construction.
We now  analyze the (in)dependences on these choices from another perspective, which also provides a  vanishing result that is similar to Corollary	\ref{cordelta0integral} (ii) obtained from type reasons:
\begin{equation}\label{eqnvanishingofhigherweightforms}
	\dashint_{X}R_{I}(\Phi)\wedge \psi=	\int_{X}R_{I}(\Phi)\wedge \psi=0\,,\quad \text{if}\,\quad |I|\geq 1\,.
\end{equation}
Here $\dashint_{X}$ is the regularized integral given in Definition \ref{dfnregularizedintegral} for case $\ref{caseC}$.

\begin{prop}\label{propPhiIvanishing}
	Consider case $\ref{caseC}\cap (\textbf{NCD})$.	Suppose $\Phi\in A_{X}^{n}(\log D)\cap \ker\,d$.
	\begin{enumerate}[i).]
		\item 
		The regularized integral $\dashint_{X} R_{I}(\Phi)\wedge \psi$ is independent of the choice of the sections $\sigma_a$ of $\mathcal{O}_{X}([D_{a}])$ and the Hermitian metric $|\cdot|$ used in constructing $R_{I}(\Phi)$.
		\item More strongly, one has
		$\dashint_{X}R_{I}(\Phi)\wedge \psi=0$ if $|I|\geq 1$.
	\end{enumerate}
	
\end{prop}
\begin{proof}
	\begin{enumerate}[i).]
		\item

		A different choice of the section $\sigma_a$ and the Hermitian metric $|\cdot|$ on the line bundle $\mathcal{O}_{X}([D_{a}])$ changes
		$\eta_{a}=(2\pi \mathbf{i})^{-1}\partial \log |\sigma_{a}|^2$ by a $\partial$-exact form.
		Since
		$\partial \eta_{b}=0$ by construction, 
		the quantity $\eta_{I}$ changes  by a $\partial$-exact form, say, $\partial \gamma$.
		By our construction of the extension in Definition 	\ref{dfnliftalongresiduemap} for case $\ref{caseC}\cap (\textbf{NCD})$, we have  
	$		d\widetilde{\mathrm{res}_{I}\Phi}
	=0	$
	according to Lemma \ref{lemextensionincaseR} (ii).		
		By  type reasons, it follows that $R_{I}(\Phi)\wedge \psi= \eta_{I}\wedge \widetilde{\mathrm{res}_{I}(\Phi)}\wedge \psi$ changes by
		\[
		\partial \gamma\wedge \widetilde{\mathrm{res}_{I}\Phi}\wedge \psi
		=d \gamma\wedge \widetilde{\mathrm{res}_{I}\Phi}\wedge \psi
		=	d (\gamma\wedge \widetilde{\mathrm{res}_{I}\Phi})\wedge \psi
		=	\partial (\gamma\wedge \widetilde{\mathrm{res}_{I}\Phi}\wedge \psi)\,.
		\]
		The desired claim then follows since $\dashint_{X}$ vanishes on $\partial$-exact forms.
		\item
		
		Assume for definiteness $I=(a_{1}<a_{2}<\cdots< a_{j})$.
		By Lemma \ref{lemextensionincaseR} (ii) and type reasons, one has
		\begin{eqnarray*}
			R_{I}(\Phi)\wedge \psi
			&=&\partial (\ln |\sigma_{a_{1}}|^2\wedge\eta_{I-a_{1}}\wedge \widetilde{\mathrm{res}_{I}\Phi} \wedge \psi)
			+	(-1)^{j-1}
			\ln |\sigma_{a_{1}}|^2\wedge\eta_{I\setminus \{a_{1}\}}\wedge \partial\widetilde{\mathrm{res}_{I}\Phi} \wedge \psi\\
			&=&\partial (\ln |\sigma_{a_{1}}|^2\wedge\eta_{I-a_{1}}\wedge \widetilde{\mathrm{res}_{I}\Phi} \wedge \psi)
			+	(-1)^{j-1}
			\ln |\sigma_{a_{1}}|^2\wedge\eta_{I\setminus \{a_{1}\}}\wedge d\widetilde{\mathrm{res}_{I}\Phi} \wedge \psi  \\\
			&=&\partial (\ln |\sigma_{a_{1}}|^2\wedge\eta_{I-a_{1}}\wedge \widetilde{\mathrm{res}_{I}\Phi} \wedge \psi)\,.
		\end{eqnarray*}
		Observe that $\ln |\sigma_{a_{1}}|^2\wedge\eta_{I-a_{1}}\wedge \widetilde{\mathrm{res}_{I}\Phi} \wedge \psi$ is in fact
		integrable, as can be seen by passing to the polar coordinates.
		The desired vanishing\footnote{See \eqref{eqnvanishingpartialgamma1}	 in Proposition \ref{propcurrentsofloglogforms} for a more general statement.} then follows
		from the same argument proving 
		\eqref{eqnvanishingonpartialexacttermforhypersurfacearrangement} in
		Lemma \ref{lemindependenceondecompositionaseX}.
	\end{enumerate}
\end{proof}

	\subsection{Trace map on Dolbeault-Cech cohomology}
	
	When restricted to factorizable forms, the regularized integral admits another cohomological interpretation in terms of a trace map defined on the  Dolbeault-Cech double complex whose hypercohomology computes $H^{n}(X,\Omega^{n}_{X})$, besides the interpretation  provided in Definition \ref{dfnontracemaponcurrentcohomology}.\\
	
To illustrate the idea, we restrict ourselves to the simplest case, with $X=C$ and $D$ an effective reduced divisor.
	The discussions below apply directly to the case $(\textbf{H})$, namely 
	when $D$ is a smooth hypersurface in a smooth compact complex manifold $X$.	They should also extend to the more general situations in $\ref{caseC}$, which however require a more careful treatment of the construction of the integration chains in our discussions.

	Take $U=C-D$ and $V$ to be a small analytic neighborhood of $D$.
		Consider the corresponding Dolbeault-Cech resolution of $\Omega_{C}^{1}$.
	This gives a resolution by fine sheaves.
	 The cohomology $H^{\bullet}(X,\Omega^{1}_{C})$ is then computed by the cohomology of 
	 the total complex $K^{\bullet}$ of smooth forms,
	 with the total differential  $d=\bar{\partial}+(-1)^{p}\delta$.
		The $E_{0}$ page of the corresponding double complex $K^{\bullet,\bullet}$ is displayed in 
	Figure 	\ref{figure:figDolCechdoublecomplex} below.	
	
	\begin{figure}[h]
		\centering
		\[
		\xymatrix{
				& \check{C}^{1}(\Omega_{C}^{1})\ar@{.>}[r]^{\hookrightarrow} & A^{1,0}_{U\cap V}\ar[rr]^{\bar{\partial}}&& A^{1,1}_{U\cap V}&&\\
			&&  && &&\\
			&& A^{1,0}_{U}\oplus A^{1,0}_{V}\ar[rr]^{\bar{\partial}}\ar[uu]^{\delta}&& A^{1,1}_{U}\oplus A^{1,1}_{V}\ar[uu]^{\delta} &&\\
			&&  && A^{1,1}_{X}\ar@{.>}[u]_{\hookrightarrow }&&
		}
		\]	
		\caption{Double complex computing $H^{1}(C,\Omega^{1}_{C})$.}
		\label{figure:figDolCechdoublecomplex}
	\end{figure}

	An element in $H^{1}(K^{\bullet})$ is represented by a triple 
	\[
		(\omega_{U},\omega_{V},\eta_{UV})\in C^{1}(K^{\bullet})=A^{1,0}_{U}\oplus A^{1,0}_{V}\oplus A^{1,0}_{U\cap V}
	\]
	such that
	\[
	d		(\omega_{U},\omega_{V},\eta_{UV})=\omega_{U}-\omega_{V}+\bar{\partial}\eta_{UV}=0\,.
	\]

One has the following trace map on $C^{1}(K^{\bullet})$  \cite{Polishchuk:2005} that enjoys nice properties	
\begin{eqnarray}\label{eqnImap}
	\mathrm{Tr}^{\bar{\partial}}: C^{1}(K^{\bullet})&\rightarrow& \mathbb{C}\,,\nonumber\\
	(\omega_{U},\omega_{V},\eta_{UV})
	&\mapsto & \int_{C-B_{\varepsilon}(D)}\omega_{U}+
	\int_{B_{\varepsilon}(D)}\omega_{V}-\int_{\partial B_{\varepsilon}(D)}\eta_{UV}\,.
\end{eqnarray}
Here $B_{\varepsilon}(D)$ is a disk bundle of radius $\varepsilon$ over $D$ such that $\partial B_{\varepsilon}(D)\subseteq U\cap V$, with respect to some Hermitian metric on $C$.
By Stokes theorem, the result is independent of the choice of $\varepsilon$ and the  metric.
Furthermore, it vanishes on $d$-exact terms.
Therefore, it descends to a map on the cohomology
$H^{1}(K^{\bullet})\cong H^{1}(C,\Omega_{C}^{1})$.
It specializes to the trace map $\mathrm{Tr}$ on cocycles of the form $(\omega_{U},\omega_{V},0)$
and the residue map $\mathrm{res}_{\partial}$ on cocycles of the form $(0,0,\eta_{UV})$.

Now consider a factorizable form $\omega=\phi\wedge\psi$.
Lemma \ref{lemregularizedintegralforfactorizableformsincaseHA} and 
Proposition \ref{lemprojectiontoIn0byMHS}	tell that
\begin{equation}\label{eqnvanishingofphi+}
\mathfrak{s}
\left(
([\phi]_{\mathrm{dR}}- \pi_{I^{n,0}}([\phi]_{\mathrm{dR}}))\wedge \psi\right)=0\,.
\end{equation}
The cohomology class $[\phi]_{\mathrm{dR}}- \pi_{I^{n,0}}([\phi]_{\mathrm{dR}})\in H^{n}(U)$
is the same as the class $[\Phi]_{\mathrm{dR}}- \pi_{I^{n,0}}([\Phi]_{\mathrm{dR}})\in H^{n}(A_{X}^{\bullet}(\log D),d)$ for any form $\Phi\in A_{X}^{n}(\log D)\cap \mathrm{ker}\,d$ representing $[\phi]_{\mathrm{dR}}$.
It
can be lifted from the $E_{2}$ page to the $E_{1},E_{0}$ pages.
The vanishing relation \eqref{eqnvanishingofphi+} above always hold by type reasons as in Proposition \ref{lemprojectiontoIn0byMHS}.
Take
any $d$-closed form $\Phi_{+}\in A_{X}^{n}(\log D)$ representing this class.
By Remark
\ref{remresidueonmeromorphicforms}, locally near $D$ one has 
\[
\Phi_{+}= \eta\wedge \widetilde{\mathrm{res}_{\partial}(\phi)}+\text{form in}~\Gamma(X, W_{0}\mathcal{A}^{n}_{X}(\log D ))\,.
\]

When restricted on a connected component of the Stein open $V$, the form $\psi$ is $\bar{\partial}$-exact.
For each such component, we choose a particular primitive $\bar{\partial}^{-1}\psi$ that satisfies (recall that $\partial \psi=0$)
\begin{equation}\label{eqnrequirementonpartialbarinversepsi}
	\partial \bar{\partial}^{-1}\psi=0\,,\quad ( \bar{\partial}^{-1}\psi)|_{D}=0\,.
\end{equation}
The trivial current cohomology class $[\Phi_{+}\wedge \psi]\in 
H^{n}_{\partial}(C,\mathcal{D}_{C}^{'\bullet,n})	 \cong H^{n}_{\partial}(C,\overline{\Omega}_{C}^{n})\cong H^{n}_{\bar{\partial}}(C,\Omega_{C}^{n})$
is then represented as the following particular coboundary  in  the Dolbeault-Cech complex
	\begin{equation}\label{eqnliftoflogpart}
		\varpi_{+}:=(\Phi_{+}\wedge \psi,\bar{\partial}\Phi_{+}\wedge \bar{\partial}^{-1}\psi,\Phi_{+}\wedge\bar{\partial}^{-1}\psi)\,.	
		\end{equation}
	This is a collection of smooth forms supported on different Cech opens.
		The resulting cocycle for $\omega=\phi\wedge \psi$ 
in $C^{1}(K^{\bullet})$ is then 	
\begin{equation}\label{eqnconstructionofDolCechcyclen=1}
	\varpi(\phi,\psi)
	:=(\Phi_0\wedge \psi,\Phi_0\wedge \psi,0 )
	+\varpi_{+}\in C^{1}(K^{\bullet})\,.
\end{equation}
	Note that $\Phi_0\wedge \psi\in  A^{n,0}_{X}(\overline{\Omega}_{X}^{n})$ gives a Dolbeault form $\omega_{\mathrm{Dol}}$ associated to $T_{\Phi\wedge \psi}$ via  Theorem \ref{thmresidueformulaforcurrents}.
	
Taking the limit $\varepsilon\rightarrow 0$ in \eqref{eqnImap},
	it follows from the $\varepsilon$-independence of $\mathrm{Tr}^{\bar{\partial}}$ in \eqref{eqnImap}, \eqref{eqnholomorphicresidue}, Theorem \ref{thmformdefiningcurrent}, and Corollary \ref{cordelta0integral} that
\[
	\mathrm{Tr}^{\bar{\partial}}(\varpi_{+})
=\dashint_{X} \Phi_{+}\wedge \psi- 2\pi \mathbf{i}\,\mathrm{res}_{\partial}(\Phi_{+}\wedge \bar{\partial}^{-1}\psi)
	=- 2\pi \mathbf{i}\,\mathrm{res}_{\partial}(\Phi_{+}\wedge \bar{\partial}^{-1}\psi)\,.
\]
A local calculation using  \eqref{eqnrequirementonpartialbarinversepsi}
tells that the above result is zero.
Using \eqref{eqnconstructionofDolCechcyclen=1} and Corollary \ref{cordelta0integral}, one obtains
	\begin{equation}\label{eqnconstructionoftracen=1}
	\mathrm{Tr}^{\bar{\partial}}(\varpi(\phi,\psi))=
	\mathrm{Tr}^{\bar{\partial}}((\Phi_0\wedge \psi,\Phi_0\wedge \psi,0 ))+\mathrm{Tr}^{\bar{\partial}}(\varpi_{+})
	=\int_{X}\Phi_0\wedge \psi+0=\dashint_{X}\phi\wedge \psi\,.
	\end{equation}

	\begin{ex}\label{exn=1splittingexamplewpwz}
		This is a continuation of Example \ref{exwpn=1definitioncomputation} and 
		Example \ref{exwpn=1residuecomputation}.
		Consider the case 
		$\omega=\phi\wedge \psi$, with $\phi=\wp(z)dz,\psi={d\bar{z}\over \bar{\tau}-\tau}$. Using \eqref{eqnderivativesofZhat},
			we can take
		\[
		\Phi=\phi+d\widehat{Z}=-\widehat{\eta}_{1}dz+\pii \psi\,,\quad 
		\Phi_{0}=-\widehat{\eta}_{1}dz+\pii \psi\,,\quad
		\Phi_{+}=0
		\,.\]
		The corresponding components in $I^{1,0}=F^{1}\widetilde{W}_{1},I^{0,1}=\overline{F^{1}}\cap \widetilde{W}_{1},
		I^{1,1}=F^{1}\cap \overline{F^{1}}$
		are
		\[
		\pi_{I^{1,0}}([\Phi_{0}]_{\mathrm{dR}})=[-\widehat{\eta}_1 dz] \,,\quad 
				\pi_{I^{0,1}}([\Phi_{0}]_{\mathrm{dR}})=[\pii\psi]\,,\quad
									\pi_{I^{1,1}}([\Phi_{+}]_{\mathrm{dR}})=0
				\,.
		\]
		The cocycle $\varpi(\phi,\psi)$ in \eqref{eqnconstructionofDolCechcyclen=1}
		is then given by
			\[
		\varpi(\phi,\psi)=(-\widehat{\eta}_{1}dz\wedge \psi,-\widehat{\eta}_{1}dz\wedge \psi,0)+0\,.
		\]
		Thus
		\[
			\mathrm{Tr}^{\bar{\partial}}(\varpi(\phi,\psi))
		=-\int_{E}\widehat{\eta}_{1}\mathrm{vol}+0=-\widehat{\eta}_{1}\,.
		\]		

				Consider also $\phi=\wp(z-z_1)\wp(z-z_2)dz$ with $z_{1}-z_{2}\notin \mathbb{Z}\oplus \mathbb{Z}\tau$ and the same $\psi$ as above.
		By Remark \ref{remresidueonmeromorphicforms},  in this case (invoking \eqref{eqndfnofZhat}, \eqref{eqnangularform}, \eqref{eqnthetahat})  one takes
				\[	
		\Phi_{+}=\wp'(z_1-z_2)\widehat{Z}(z-z_1)dz+\wp'(z_2-z_1)\widehat{Z}(z-z_2)dz\,.
	\]
	See \eqref{eqnPhidecompositionforexp1p2} in Example \ref{exresiduep1p2} below for an expression for $\Phi_{0}$.
	Note that indeed one has $[\Phi_{+}]_{\mathrm{dR}}\in I^{1,1} $ as can be proved using  \eqref{eqnthetahat}. 
	The cocycle $\varpi_{+}$ is given by, up to addition by $d$-exact terms,
		\begin{eqnarray*}
\varpi_{+}&=&(\wp'(z_1-z_2)\widehat{Z}(z-z_1)\wedge \psi,
	\wp'(z_1-z_2)
	{\bar{z}-\bar{z}_1\over \bar{\tau}-\tau}\mathrm{vol},\wp'(z_1-z_2)\widehat{Z}(z-z_1)\wedge {\bar{z}-\bar{z}_1\over \bar{\tau}-\tau})\\
	 &+&(\wp'(z_2-z_1)\widehat{Z}(z-z_2)\wedge \psi,
	 \wp'(z_2-z_1)
	 {\bar{z}-\bar{z}_2\over \bar{\tau}-\tau}\mathrm{vol},\wp'(z_2-z_1)\widehat{Z}(z-z_2)\wedge {\bar{z}-\bar{z}_2\over \bar{\tau}-\tau})\,.
	\end{eqnarray*}
	Thus
			\[
			\mathrm{Tr}^{\bar{\partial}}(\varpi(\phi,\psi))
		=\langle [E],[\Phi_{0}]_{E_{1}}\wedge \psi\rangle+\int_{z_{1}}\wp'(z_{1}-z_{2})
		\mathbf{1}_{z_{1}}+\int_{z_{2}}\wp'(z_{2}-z_{1})\mathbf{1}_{z_{2}}=\langle [E],[\Phi_{0}]_{E_{1}}\wedge \psi\rangle\,,
		\]
		where $\mathbf{1}_{z_{i}}$ stands for the constant function that is supported at $z_{i}$ and takes the value $1$.

			\xxqed
	\end{ex}

\section{Examples}
\label{secexamples}

In this part we provide several examples from case \ref{caseC}, ranging from complements of  
simple normal crossing divisors to configuration spaces.
The integrals under investigation are difficult to make sense of without 
regularizations, and 
 serve as non-trivial checks of our cohomological formulations given in
Section \ref{seccurrentcohomology} and 
	Section \ref{secfactorizableforms}.

\begin{ex}\label{exresidueppowers}
		This is a continuation of Example \ref{exwpn=1definitioncomputation}, Example \ref{exwpn=1residuecomputation}.
	Consider the case $X=E,D=0$, with $\omega=\phi\wedge \psi$.
	Let $\phi=\wp(z)^{m}dz,\psi={d\bar{z}\over \bar{\tau}-\tau}$, $m=1,2,3$.
	
	For the $m=1$ case, as in Example \ref{exn=1splittingexamplewpwz} we have from  \eqref{eqnderivativesofZhat}  that
	\begin{equation}\label{eqnwpdzdR}
	\wp dz=-\widehat{\eta}_1 dz-\partial\widehat{Z}=-\widehat{\eta}_1 dz+\pii \psi-d\widehat{Z}\,.
	\end{equation}	
	Recall the Weierstrass relations
	\begin{equation}\label{eqnWeierstrassrelations}
	(\wp')^{2}=4\wp^{3}-g_{2}\wp-g_{3}\,,\quad \wp''=6\wp^2-{1\over 2}g_{2}\,.	
	\end{equation}
This gives
\[
\wp^{2}={1\over 6}\wp''+{1\over 12}g_{2}\,,\quad 
\wp^{3}={3\over 20}g_{2}\wp+{1\over 10}g_{3}+{1\over 10}(\wp \wp')^{'}\,.
\]
It follows from  \eqref{eqnwpdzdR} that 
\[
\wp^{2}dz={1\over 12}g_{2}dz+{1\over 6}d\wp'\,,\quad 
\wp^{3}dz=-{3\over 20}g_{2}\widehat{\eta}_1 dz
+{3\over 20}\pii  g_{2}\psi-{3\over 20} g_{2}d\widehat{Z}+{1\over 10}g_{3}dz
+{1\over 10}d(\wp \wp')\,.
\]
These give
\[
\dashint_{E}\wp^{2}dz\wedge \psi={1\over 12}g_{2}\,,\quad 
\dashint_{E}\wp^{3}dz\wedge \psi=-{3\over 20}\widehat{\eta}_{1}g_{2}+
{1\over 10}g_{3}\,.
\]		 	
	\end{ex}
	
	\begin{ex}\label{exresiduep1p2}
		Consider the case $X=E,\omega=\phi\wedge \psi$, with 
		\[
		\phi=\wp(z-z_1)\wp(z-z_2)dz\,,~z_1\neq z_2~\mathrm{mod}~ \mathbb{Z}\oplus \mathbb{Z}{\tau}\,,\quad 
		\psi={d\bar{z}\over \bar{\tau}-\tau}\,.
		\]			
		
			Recall (see e.g., \cite{Silverman:2009}) the series expansion for the function  $\widehat{Z}$ given in \eqref{eqndfnofZhat}
		\begin{equation}\label{eqnwidehatZexpansion}
			\widehat{Z}(z)={1\over z}-\widehat{\eta}_1 z-\sum_{k\geq 2} {2G_{2k}}{z^{2k-1}}+{-\pi\over \mathrm{im}\,\tau}\bar{z}\,,
			\quad
			G_{2k}={1\over 2}\sum_{(m,n)\in \mathbb{Z}^2-\{(0,0)\}} {1\over (m\tau+n)^{2k}}\,.
		\end{equation}		
				 Using the iterated residue formulas in \cite{Li:2022regularized}, one has 
		 \begin{eqnarray}\label{eqnexresiduep1p2}
		 \dashint_{E}\phi\wedge \psi&=&\mathrm{res}(\wp(z-z_1)\wp(z-z_2)\widehat{Z}(z-z_2))\\
		 &=&
		 \wp'(z_1-z_2)\widehat{Z}(z_1-z_2)
		 +
		  \wp(z_1-z_2)\widehat{Z}'(z_1-z_2)
		 +
		{1\over 2}\wp''(z_1-z_2)+\wp(z_{2}-z_{1})(-\widehat{\eta}_{1})\nonumber\,.
		 \end{eqnarray}
	 
	 Let us also consider the splitting \eqref{eqndecomposePhiintolifts}  in Construction \ref{consconstructionofE-1lift}  which in particular gives a decomposition \eqref{eqnlogdecomposition}.
	 	Hereafter, we often omit the notation for the differentials $dz,d\bar{z}$ part if they are clear from the surrounding texts. We also use
	 $f(z_{ij})$ to denote the function
	 $f(z_{i}-z_{j})$.
		Using the addition formula for the Weierstrass $\wp$- and $\zeta$-function, one has
	 \begin{eqnarray*}
	 &&	\phi-\wp'(z_1-z_2)(\zeta(z-z_1)-\zeta(z-z_2))\\
	 	&=&	 -\wp(z-z_2)^2-\wp(z_{21})\wp(z-z_2)-\wp'(z_{12})\zeta(z_{21})\\
	 	&	&+{1\over 4}( {\wp'(z-z_2)-\wp'(z_{21})\over\wp(z-z_2)-\wp(z_{21}) })^2\wp(z-z_2)
	 	-{1\over 2}\wp'(z_{12}){\wp'(z-z_2)-\wp'(z_{21})\over\wp(z-z_2)-\wp(z_{21}) }\,.
	 	\end{eqnarray*}	
 	Straightforward computations using the Weierstrass relation \eqref{eqnWeierstrassrelations} and the addition formulas for $\wp,\zeta$ show that the last line above gives
 	  \begin{eqnarray*}
 	&&{1\over 4} {\wp'(z-z_2)^2-\wp'(z_{21})^2\over (\wp(z-z_2)-\wp(z_{21})) }
 	+	{1\over 4}\wp(z_{21}) {(\wp'(z-z_2)+\wp'(z_{21}))^2\over (\wp(z-z_2)-\wp(z_{21}))^2 }
 		-	{1\over 4}\wp(z_{21}) {4\wp'(z-z_2)\wp'(z_{21})\over (\wp(z-z_2)-\wp(z_{21}))^2 }\\
 		&=&
 		(\wp(z-z_2)^2+\wp(z_{21})\wp(z-z_2)+\wp(z_{21})^2-{1\over 4}g_{2})\\
 		&&+\wp(z_{21}) (\wp(z-z_2)+\wp(z_{12})+\wp(z+z_1-2z_2))+\partial_{z}
 		\left(		 {\wp(z_{21})\wp'(z_{21})\over \wp(z-z_2)-\wp(z_{21}) }
 		\right)
 		\,.
 	 	\end{eqnarray*}	
  	It follows that
  		 \begin{eqnarray*}
  		&&	\phi-\wp'(z_1-z_2)(\zeta(z-z_1)-\zeta(z-z_2))\\
  		&=&	 -\wp'(z_{12})\zeta(z_{21})+\wp(z_{21})^2-{1\over 4}g_{2}\\\
  		&	&+\wp(z_{21}) \wp(z-z_2)+\wp(z_{12})^2+\wp(z_{21})\wp(z+z_1-2z_2)+\partial_{z}
  		\left(		 {\wp(z_{21})\wp'(z_{21})\over \wp(z-z_2)-\wp(z_{21}) }
  		\right)\,.
  	\end{eqnarray*}	
Recall \eqref{eqndfnofZhat} and \eqref{eqnderivativesofZhat}, we obtain 
    \begin{eqnarray*}
  	\phi&=&-\wp'(z_{12})\widehat{Z}(z_{21})+2\wp(z_{21})^2-{1\over 4}g_{2}
  	+\wp(z_{21}) \wp(z-z_2)+\wp(z_{21})\wp(z+z_1-2z_2)\,\nonumber\\
  	&&+\wp'(z_{12})(\widehat{Z}(z-z_1)-\widehat{Z}(z-z_2))+\partial_{z}
  	\left(		 {\wp(z_{21})\wp'(z_{21})\over \wp(z-z_2)-\wp(z_{21}) }
  	\right)\\
  	&=&-\wp'(z_{12})\widehat{Z}(z_{21})+2\wp(z_{21})^2-{1\over 4}g_{2}
  	-2\widehat{\eta}_{1}\wp(z_{21})\,\nonumber\\
  	&&+\wp'(z_{12})(\widehat{Z}(z-z_1)-\widehat{Z}(z-z_2))+\partial_{z}
  	\left(		 {\wp(z_{21})\wp'(z_{21})\over \wp(z-z_2)-\wp(z_{21}) }
  	\right)\,.
  \end{eqnarray*}
  This gives a desired  splitting \eqref{eqndecomposePhiintolifts} $\phi=\Phi_{0}+\Phi_{+}-d\gamma$  in Construction \ref{consconstructionofE-1lift}, with 
  \begin{eqnarray}\label{eqnPhidecompositionforexp1p2}
  -d\gamma&=&d\left(		 {\wp(z_{21})\wp'(z_{21})\over \wp(z-z_2)-\wp(z_{21}) }
  	\right)\,,\nonumber\\
  \Phi_{0}&=&\left(-\wp'(z_{12})\widehat{Z}(z_{21})+2\wp(z_{21})^2-{1\over 4}g_{2}
 	-2\widehat{\eta}_{1}\wp(z_{21})\right)dz\,,\nonumber\\
  \Phi_{+}&=&\left(\wp'(z_1-z_2)(\widehat{Z}(z-z_1)-\widehat{Z}(z-z_2))\right)dz\,.
    \end{eqnarray}
From the Weierstrass relations \eqref{eqnWeierstrassrelations} and  \eqref{eqnderivativesofZhat}, one sees that the regularized integral of $\Phi_0\wedge \psi$ matches the result in  \eqref{eqnexresiduep1p2} above.

	\xxqed
	\end{ex}

	\begin{ex}\label{exn=3splittingNCDexample}
		Consider $X=E\times E\times E,D=\Delta_{12}+\Delta_{23}$, 
		$\omega= \phi\wedge \psi$, with
		\[
		\phi=\wp(z_1-z_2)\wp(z_2-z_3)dz_1\wedge dz_2\wedge dz_3\,,\quad
		\psi=\bigwedge_{i=1}^3{d\bar{z}_{i}\over \bar{\tau}-\tau}\,.
		\]
		
		Using \eqref{eqnwpdzdR}, the iterated regularized integral gives
		\begin{equation}\label{eqnconfigurationspaceexamplemethod1}
			\dashint_{E_3}\dashint_{E_2} \dashint_{E_{1}} \phi\wedge \psi =
		\dashint_{E_3}\dashint_{E_2} -\widehat{\eta}_1 \wp(z_2-z_3)=\widehat{\eta}_1^2\,.
				\end{equation}
			One can also perform the integration $\dashint_{E_{2}}$ first by applying \eqref{eqnPhidecompositionforexp1p2} and \eqref{eqnWeierstrassrelations}. This gives the same result as above:
					\begin{eqnarray}\label{eqnconfigurationspaceexamplemethod2}
			&&	\dashint_{E_3}\dashint_{E_2} \dashint_{E_{1}} \phi\wedge \psi \nonumber\\&=&
				\dashint_{E_3}\dashint_{E_1} 
				(\wp'(z_{13})\widehat{Z}(z_{13})+2\wp(z_{13})^2-{1\over 4}g_{2}-2\widehat{\eta}_1 \wp(z_{13}) ) \nonumber\\
			&	=&				
				\dashint_{E_3}\dashint_{E_1} 
			((\wp(z_{13})\widehat{Z}(z_{13}))'+\wp(z_{13})(\wp(z_{13})+\widehat{\eta}_1)+2\wp(z_{13})^2-{1\over 4}g_{2}-2\widehat{\eta}_1 \wp(z_{13}) ) \nonumber\\
							&	=&				
			\dashint_{E_3}\dashint_{E_1} 
			(3\wp(z_{13})^2-{1\over 4}g_{2}
				-\widehat{\eta}_1 \wp(z_{13})) \nonumber\\
					&	=&				
				\dashint_{E_3}\dashint_{E_1} 
				({1\over 2}\wp(z_{13})''
				-\widehat{\eta}_1 \wp(z_{13})) \nonumber\\
				&=&
				\widehat{\eta}_1^2\,.
			\end{eqnarray}
			
										For comparison, we also calculate the result using the iterated residue formulas in \cite{Li:2022regularized}. 
				Using the Laurent expansions \eqref{eqnwidehatZexpansion} of $\wp,\widehat{Z}$, the regularized integral gives
		\begin{eqnarray}&&
		\dashint_{E_3}\dashint_{E_2} \dashint_{E_{1}} \phi\wedge \psi \nonumber\\&=&	\mathrm{res}_{z_2=z_3}
		\mathrm{res}_{z_1=z_3} \wp(z_1-z_2)\wp(z_2-z_3)\widehat{Z}(z_1-z_3)\widehat{Z}(z_2-z_3)\nonumber\\
&	&+
			\mathrm{res}_{z_2=z_3}
		\mathrm{res}_{z_1=z_2}\wp(z_1-z_2)\wp(z_2-z_3)
		 \left(
		 \widehat{Z}(z_{1}-z_3)\widehat{Z}(z_2-z_3)+{1\over 2}\widehat{Z}(z_2-z_3)^2
		 \right)\nonumber\\
		 &=&
		 	\mathrm{res}_{z_2=z_3}
		\wp(z_2-z_3)^2\widehat{Z}(z_2-z_3)\nonumber\\
		 &	&+
		 \mathrm{res}_{z_2=z_3}\wp(z_2-z_3)
		 		 \widehat{Z}'(z_{2}-z_3)\widehat{Z}(z_2-z_3)\nonumber \\
	&=&	 		  -\widehat{\eta}_1\mathrm{res}_{z_2=z_3}
		 		 \wp(z_2-z_3)\widehat{Z}(z_2-z_3)\nonumber\\
		 		 &=&\widehat{\eta}_1^2\,.
		\end{eqnarray}
	
	Let us also evaluate the integral using the splitting \eqref{eqndecomposePhiintolifts} in Construction \ref{consconstructionofE-1lift} above.
Denote $\psi_{ij}={d\bar{z}_{i}-d\bar{z}_{j}\over \bar{\tau}-\tau}$.
Straightforward computations using \eqref{eqnwpdzdR}
give 
$
\phi=\Phi_{0}+\Phi_{+}-d\gamma
$
with
\begin{eqnarray}\label{eqnPhidecompositionforp12p23}
\Phi_{+}&=&0\,,\nonumber\\
\Phi_0&=&
\widehat{\eta}_{1}^{2}dz_{1}\wedge dz_{2}\wedge dz_{3}-\pii \widehat{\eta}_{1}\psi_{23}\wedge dz_{3}\wedge dz_{1}
-\pii \widehat{\eta}_{1}\psi_{12}\wedge dz_{2}\wedge dz_{3}+(\pii)^2 
\psi_{12}\wedge \psi_{23}\wedge dz_{3}\,,\nonumber\\
-d\gamma&=&-\widehat{\eta}_{1} d(\widehat{Z}(z_{23})dz_{3}\wedge dz_{1})
-d(\widehat{Z}(z_{12})\wp(z_{23})dz_{2}\wedge dz_{3})
+\pii d(\widehat{Z}(z_{23})\psi_{12}\wedge dz_{3}) \,.
\end{eqnarray}
Therefore the regularized integral of $\omega$ is given by $\dashint_{X}\Phi_0\wedge \psi=\widehat{\eta}_1^2$.	
That $\Phi_{+}=0$ is consistent with the vanishing of the cohomology classes of the residues $\mathrm{res}_{I}\phi$ which can be easily computed using Remark \ref{remresidueonmeromorphicforms}.
			\xxqed
	\end{ex}

\begin{ex}\label{exn=3splittingconfigurationspaceexample}
	Consider $X=E\times E\times E,D=\Delta_{12}+\Delta_{23}+\Delta_{31}$, $\omega=\phi\wedge \psi$,	
	with 
	\[
	\phi=\wp(z_1-z_2)\wp(z_2-z_3)\wp(z_3-z_1)dz_1\wedge dz_2\wedge dz_3
	\,,\quad
	\psi=\bigwedge_{i=1}^3{d\bar{z}_{i}\over \bar{\tau}-\tau}\,.
	\]
	
The iterated regularized integral has essentially been studied in \cite[Example 3.24, Appendix C]{Li:2020regularized}  that utilizes \eqref{eqncontacttermn=1}.
Let us also evaluate it via the iterated residue formulas in \cite{Li:2022regularized}.
	The details are as follows:
	\begin{eqnarray*}
	&&
	\dashint_{E_{3}}\dashint_{E_{2}}(R^{1}_{2}+R^{1}_{3})\wp(z_{12})\wp(z_{23})\wp(z_{31})\widehat{Z}(z_{13})\\
	&=&
	\dashint_{E_{3}}\dashint_{E_{2}}
	\wp(z_{23})\wp'(z_{23})\widehat{Z}(z_{23})+	\wp(z_{23})\wp(z_{23})\widehat{Z}'(z_{23})+\wp(z_{23})\wp(z_{32})(-\widehat{\eta}_1)+{1\over 2}\wp''(z_{23})\wp(z_{23})
	\\
		&=&
	\dashint_{E_{3}}R^{2}_{3}
	\left(
	\wp(z_{23})\wp'(z_{23}){1\over 2}\widehat{Z}^2(z_{23})+(	\wp(z_{23})^2\widehat{Z}'(z_{23})+\wp(z_{23})^2(-\widehat{\eta}_1)+{1\over 2}\wp''(z_{23})\wp(z_{23}))\widehat{Z}(z_{23})\right).
			\end{eqnarray*}		
	Hereafter we use the short-hand notation $R^{i}_{j}$ to denote $\mathrm{res}_{z_{i}=z_{j}}$.
	Observe that
	\[
	\wp(z_{23})\wp'(z_{23}){1\over 2}\widehat{Z}^2(z_{23})=
	\partial_{z_2}( {1\over 2}\wp(z_{23})^2  {1\over 2}\widehat{Z}^2(z_{23}))
	-{1\over 2}\wp(z_{23})^2 \widehat{Z}(z_{23})\widehat{Z}'(z_{23})\,,
	\]
the above reduces to (using \eqref{eqnWeierstrassrelations})
	\begin{eqnarray*}
	&&
	\dashint_{E_{3}}R^{2}_{3}
	\left(\left(
	-{1\over 2}\wp(z_{23})^2 \widehat{Z}'(z_{23})+\wp(z_{23})^2\widehat{Z}'(z_{23})+\wp(z_{23})^2(-\widehat{\eta}_1)+{1\over 2}\wp''(z_{23})\wp(z_{23})\right)\widehat{Z}(z_{23})\right)\\
	&=&
	R^{2}_{3}
	\left(\left(
		{5\over 2}\wp(z_{23})^3 -{3\over 2}\widehat{\eta}_1\wp(z_{23})^2-{1\over 4}g_{2}\wp(z_{23})\right)\widehat{Z}(z_{23})\right)\,.
\end{eqnarray*}		
	Using $[f]_{k}$ to denote the degree-$k$ coefficient of a formal Laurent series $f$,  from \eqref{eqnwidehatZexpansion}
	the residue above then gives
		\begin{eqnarray}\label{eqnnonNCDexmethod1}
			&&
	{5\over 2}[\wp^3]_{-6} [\widehat{Z}]_{5}+
	{5\over 2}[\wp^3]_{-4} [\widehat{Z}]_{3}+
	{5\over 2}[\wp^3]_{-2} [\widehat{Z}]_{1}+
	{5\over 2}[\wp^3]_{0} [\widehat{Z}]_{-1}\nonumber\\
	&+&
		-{3\over 2}\widehat{\eta}_1 [\wp^2]_{-4} [\widehat{Z}]_{3}
			-{3\over 2}\widehat{\eta}_1 [\wp^2]_{-2} [\widehat{Z}]_{1}
				-{3\over 2}\widehat{\eta}_1 [\wp^2]_{0} [\widehat{Z}]_{-1}
								\nonumber\\
								&+&  -{1\over 4}g_{2} [\wp]_{-2}[\widehat{Z}]_{1}- {1\over 4}g_{2} [\wp]_{0}[\widehat{Z}]_{-1}
									\nonumber\\
				&=&
					{5\over 2} [\widehat{Z}]_{5}
				+{5\over 2}[\wp^3]_{-2} [\widehat{Z}]_{1}
				+{5\over 2}[\wp^3]_{0} 
				-{3\over 2}\widehat{\eta}_1  [\widehat{Z}]_{3}
						-{3\over 2}\widehat{\eta}_1  [\wp^2]_{0}
							-	{1\over 4}g_{2} [\widehat{Z}]_{1}\nonumber \\
					&=&
				{5\over 2}  (-2G_{6})
				+{5\over 2} 3  (6G_{4}) (-\widehat{\eta}_1  )
				+{5\over 2}3  (10G_{6})
				-{3\over 2}\widehat{\eta}_1   (-2G_{4})
					-{3\over 2}\widehat{\eta}_1   2(6G_{4})
			+{1\over 4}g_{2}\widehat{\eta}_1 \nonumber\\
				&=&
				70G_{6}-60\widehat{\eta}_1 G_{4}
								+{1\over 4}g_{2}\widehat{\eta}_1\,.
\end{eqnarray}	
	
		Let us also evaluate the integral using the splitting \eqref{eqndecomposePhiintolifts} in Construction \ref{consconstructionofE-1lift} above.
	From the results in \eqref{eqnPhidecompositionforexp1p2}, we
	have
	\begin{eqnarray*}
	\phi&=&\wp(z_{12})
	d\left(		 {\wp(z_{21})\wp'(z_{21})\over \wp(z_{32})-\wp(z_{21}) }
  dz_{1}\wedge dz_{2}	\right)\\
	&&+\wp(z_{12})
	\wp'(z_{12})(\widehat{Z}(z_{31})-\widehat{Z}(z_{32})) dz_{3}\wedge dz_{1}\wedge dz_{2}\\
	&&+\wp(z_{12})
	\left(-\wp'(z_{12})\widehat{Z}(z_{21})+2\wp(z_{21})^2-{1\over 4}g_{2}
 	-2\widehat{\eta}_{1}\wp(z_{21})\right)dz_{3}\wedge dz_{1}\wedge dz_{2}\,.
	\end{eqnarray*}
	The first term above is $d$-exact.
	The second term is simplified into 
	\begin{eqnarray*}
	&&-\widehat{Z}(z_{31})d({1\over 2}\wp(z_{21})^2 )\wedge dz_{3}\wedge dz_{1}
	-	\widehat{Z}(z_{32})d({1\over 2}\wp(z_{12})^2 )\wedge dz_{2}\wedge dz_{3}\\
	&=&
	-d(\widehat{Z}(z_{31})\cdot {1\over 2}\wp(z_{21})^2 )\wedge dz_{3}\wedge dz_{1}
	-	d(\widehat{Z}(z_{32})\cdot {1\over 2}\wp(z_{12})^2 )\wedge dz_{2}\wedge dz_{3}\\
	&&+\bar{\partial}(\widehat{Z}(z_{31})\cdot {1\over 2}\wp(z_{21})^2 )\wedge dz_{3}\wedge dz_{1}
	+	\bar{\partial}(\widehat{Z}(z_{32})\cdot {1\over 2}\wp(z_{12})^2 )\wedge dz_{2}\wedge dz_{3}\,.
	\end{eqnarray*}	
	The last term is simplified into
	\begin{eqnarray*}
		&&d({1\over 2}\wp(z_{12})^2 \widehat{Z}(z_{12})dz_2\wedge dz_3)
		+{1\over 2}\wp(z_{12})^2 (\wp(z_{12})+\widehat{\eta}_1)
		dz_1\wedge dz_2\wedge dz_3\\
		&&
		-
		{1\over 2}\pii\wp(z_{12})^2 
	{	d\bar{z}_1-d\bar{z}_{2}\over 
		\bar{\tau}-\tau}\wedge dz_2\wedge dz_3		
		+\left(2
		\wp(z_{12})^3-{1\over 4}g_{2}\wp(z_{12})
		-2\widehat{\eta}_{1}\wp(z_{12})^2\right)	dz_1\wedge dz_2\wedge dz_3\,.
	\end{eqnarray*}
	Using the results in Example \ref{exresidueppowers}, it follows that
	\[
	\phi=
({1\over 4}g_{3}-{1\over 4}g_{2}\widehat{\eta}_{1})dz_1\wedge dz_2\wedge dz_3	+\text{forms of type}~(2,1)
	+d\text{-exact forms}\,.
	\]
	Therefore, we have
	\begin{equation}
	\dashint_{X}\phi\wedge \psi={1\over 4}g_{3}-{1\over 4}g_{2}\widehat{\eta}_{1}\,.
	\end{equation}
	By the relations $g_{2}=60\cdot 2 G_4, g_3=140\cdot 2G_{6}$  (see e.g., \cite{Silverman:2009}), we see that indeed this  
	matches \eqref{eqnnonNCDexmethod1}.

\xxqed

	\end{ex}

	\begin{appendices}

\section{Local integrability of logarithmic forms}
\label{appendixlocalestimates}

In this part we give some estimates that are used in the body of the paper.
We also discuss representatives of the current cohomology class $	[\mathrm{P.V.}_{\omega}]$ fo the current $\mathrm{P.V.}_{\omega}$ induced from \eqref{eqnHLcurrent} (cf. Proposition \ref{propformdefiningcurrentcaseD}, Corollary \ref{coruniquenessofregularization}),
in terms of integral currents arising from differential forms.

\begin{dfn}\label{dfnloglogsheaf}
	Consider  case $\ref{caseC}$ and $(\textbf{NCD})$.
	Fix a Hermitian metric  $|\cdot|^2$ for each of the line bundle $\mathcal{O}_{X}([D_{a}])$.
		For any point $p$ in $D_{I}$, take  $J\subseteq I$ such that $D_{J}=D_{I}$ 
	and $|J|=j=n-\dim\, D_{I}$.
	Pick a polydisk neighborhood with coordinate system $(s,w)=(s_{1},\cdots, s_{j}, w_{j+1},\cdots, w_{n})$ as before, where the $s_a$'s are the local defining equations of irreducible components of $D_{J}$.
	
	Let $\mathcal{A}^{\bullet,\bullet}_{X,\log}(\log D)$ be the sheaf of	
 forms generated 
	by $\prod_{a\in J}(\ln |s_{a}|^2)^{k_{a}} s_{a}^{-\ell_{a}}$  with  $k_{a}\in\mathbb{Z}_{\geq 0},\ell_{a}\in \{0,1\}, k_{a}\ell_{a}=0$,
		  for all choices of $J$, as a $\mathcal{A}^{\bullet,\bullet}_{X}$-module.
	Similarly, let $\mathcal{A}^{\bullet,\bullet}_{X,\log}$ be the sheaf of forms generated 
	by $\prod_{a\in J}(\ln |s_{a}|^2)^{k_{a}}$  with  $k_{a}\in\mathbb{Z}_{\geq 0}$,
	for all choices of $J$, as a $\mathcal{A}^{\bullet,\bullet}_{X}$-module.
\end{dfn}
Since different choices of the Hermitian metric lead to quantities that are differed by smooth functions, the sheaf above is independent of the choice of the metric.

\begin{lem}\label{lemintegrabilityoflogpart}
	Suppose $\alpha\in A^{n,n}_{X,\log}(\log D)$.
	Then $\alpha$ is locally absolutely integrable.
\end{lem}
\begin{proof}
	By a partition of unity argument it suffices to prove the local integrability near a point $p\in D_{J}=\cap_{a\in J}D_{a}$. 
	Then locally $\alpha$ is a linear combination of forms
	of the type
	\[
	{\bigwedge_{a\in J}(\ln |s_{a}|^2)^{k_{a}} {ds_{a}\over s_{a}^{\ell_a}}}\wedge \psi\,,
	\]
	where $\psi$ is a smooth form, $k_{a}\in \mathbb{Z}_{\geq 0},\ell_{a}\in \{0,1\} ,k_{a}\ell_a=0$.
Passing to the polar coordinates, one sees the  local absolute integrability.
\end{proof}

\begin{lem}\label{lemeliminatinglog}
	Suppose $\alpha\in A^{n,q}_{X}(\log D),0\leq q\leq n$.
	Then there exists a
	decomposition
	\begin{equation}\label{eqnlogdecompositionintologlog}
	\alpha=\alpha_0+\partial \beta\,,\quad 
	\alpha_0\in  A^{n,q}_{X,\log}\,,~\beta\in  A^{n-1,q}_{X,\log }(\log D)\,.
	\end{equation}
	Furthermore, $\beta$ defines an integral current.
\end{lem}
\begin{proof}
	Similar to the proof of Proposition \ref{prophomotopy} (i), we only need to work 
	locally near a point $p\in D_{J}$.
	Then $\alpha$ is a linear combination of differential forms of the type
	\[
	\bigwedge_{a\in J}{\partial s_{a}\over s_{a}}\wedge \psi
	=\partial \ln |s_a|^2\wedge 
	\bigwedge_{\substack{b\in J\\
			b\neq a}}{\partial s_{b}\over s_{b}}\wedge \psi\,,
	\]
	where $\psi$ is a smooth form.
	Then one applies integration by parts to obtain the desired decomposition.
	For any test form $\psi\in A^{1,0}_{X}$, one has
	$\beta\wedge \psi \in  A^{n,n}_{X,\log}(\log D)$.
	By Lemma \ref{lemintegrabilityoflogpart}, it is
	locally absolutely integrable. This finishes the proof. 
	
\end{proof}

\begin{prop}\label{propcurrentsofloglogforms}
	Let $	\mathrm{P.V.}$ be the principal value current corresponding to  $A^{\bullet,\bullet}_{X}(\star D)$ via \eqref{eqnHLcurrent}
	that extends the integral current. 	
	Then for any $\gamma\in A^{n-1,n}_{X,\log}(\log D)+A^{n-1,n}_{X}(\star D)$, one has
	$		\mathrm{P.V.}_{\partial \gamma}=\partial \,	\mathrm{P.V.}_{\gamma}$.
\end{prop}
\begin{proof}
	By linearity, it suffices to show the relation for $\gamma\in A^{n-1,n}_{X,\log}(\log D)$ and $\gamma\in A^{n-1,n}_{X}(\star D)$.
	The current $\mathrm{P.V.}_{\gamma},\gamma\in A^{n-1,n}_{X,\log}(\log D) $
	exists
	as an integral current by Lemma \ref{lemeliminatinglog}.
	For $\gamma\in A^{n-1,n}_{X}(\star D)$ and any test form $\psi\in A^{1,0}_{X}$,
	one has $\gamma\wedge \psi\in A^{n,n}_{X}(\star D)$ and thus
	$\mathrm{P.V.}_{\gamma}(\psi)$ also exists by \eqref{eqnHLcurrent}.
	It remains to show that for any $f\in A^{0,0}_{X}$, one has
	\[
	\mathrm{P.V.}_{\partial \gamma}(f)-(\partial \mathrm{P.V.}_{\gamma})(f)
	=\mathrm{P.V.}_{\partial (\gamma f)}(1)
	=0\,.
	\]	
	For  $\gamma\in A^{n-1,n}_{X}(\star D)$, that
	$\mathrm{P.V.}_{\partial (\gamma f)}(1)=0$ follows from \eqref{eqnresiduecurrentvanishing}.
	For $\mathrm{P.V.}_{ \partial (\gamma f)}(1)$ with $\gamma\in A^{n-1,n}_{X,\log}(\log D)$, using the absolute integrability, it suffices to show that 
	\begin{equation}\label{eqnvanishingpartialgamma1}	
		\int_{X}\partial (\gamma f)
		=0\,.
	\end{equation}	
This follows from 
the same arguments proving \eqref{eqnvanishingonpartialexacttermforhypersurfacearrangement} in Lemma 
\ref{lemindependenceondecompositionaseX}.
\end{proof}

For 	
case $\ref{caseC}$ and $(\textbf{NCD})$, one 
can obtain a decomposition
\[
\omega=\omega_0+\partial \gamma\,,\quad 
\omega_0\in  A^{n,n}_{X,\log}\,,~\gamma\in  A^{n-1,n}_{X,\log }(\log D)
+A^{n-1,n}_{X}(\star D)\,,
\] 
by combining  the decomposition 
\eqref{eqnlogdecompositioncasepartial} in  Proposition \ref{prophomotopy} (i)
and the one \eqref{eqnlogdecompositionintologlog} in Lemma \ref{lemeliminatinglog}.
Applying Proposition \ref{propcurrentsofloglogforms}, this gives rise to a differential form that represents the current cohomology class $[T_{\omega}]
=[\mathrm{P.V.}_{ \omega}]$.
However,  unlike Proposition \ref{propconstuctingclassusingDolCech}, this procedure is not constructive enough. Also it
does not give a satisfying representative
in terms of a \emph{smooth} form but one with mild logarithmic singularities.

\end{appendices}

\bibliographystyle{amsalpha}

\begin{thebibliography}{BEK06}

\bibitem[ATY94]{Aizenberg:1994}
L.~Aizenberg, A.~Tsikh, A.~Yuzhakov, \emph{Multidimensional Residues and Applications}. In: Khenkin, G.M., Vitushkin, A.G. (eds) Several Complex Variables II. Encyclopaedia of Mathematical Sciences, vol \textbf{8}, 1994. Springer, Berlin, Heidelberg. 

	\bibitem[BD21a]{Brown:2021a}
	F.~Brown and C.~Dupont, \emph{Single-valued integration and double copy}. J. f\"ur die Reine Angew. Math. (Crelles Journal) 
	\textbf{775}, 145--196 (2021).
	
		\bibitem[BD21b]{Brown:2021b}
	F.~Brown and C.~Dupont, \emph{Single-Valued Integration and Superstring Amplitudes in Genus Zero}. 
	Commun. Math. Phys. \textbf{382}, 815--874 (2021). 

	\bibitem[CH78]{Coleff:1978}
	N.~Coleff and M. ~Herrera, \emph{Les Courants Residuels Associes a Une Forme Meromorphe}. Berlin; New York: Springer-Verlag, 1978.
	
	
		\bibitem[Cle77]{Clemens:1977}
	C.~H. Clemens, \emph{Degeneration of K\"ahler manifolds}, Duke Math. J. {\bf 44} (1977), no.~2, 215--290.
	
		\bibitem[Del71]{Deligne:1971}
	P.~Deligne, \emph{Th\'eorie de Hodge, II.} Publications Math\'ematiques de L’Institut des Hautes Scientifiques, \textbf{40}, 5--57 (1971). 
	
			\bibitem[Del74]{Deligne:1974}
	P.~Deligne, \emph{Th\'eorie de Hodge, III.} Publications Math\'ematiques de L’Institut des Hautes Scientifiques, \textbf{44}, 5--77 (1974). 

		\bibitem[Dol70]{Dolbeault:1970}
	P. Dolbeault, \emph{Courants residus des formes semi-meromorphes}. Seminaire Pierre Lelong (Analyse)
	(annee 1970), Lecture Notes in Math., Vol. \textbf{205}, Springer, Berlin, 1971, pp. 56--70.
	
			\bibitem[Dup15]{Dupont:2015}
	C.~Dupont, \emph{The Orlik-Solomon model for hypersurface arrangements}. Annales de l'Institut Fourier, Volume \textbf{65} (2015) no. 6, pp. 2507--2545. 

		\bibitem[FK17]{Felder:2017}
		G. Felder and D. Kazhdan, \emph{Divergent Integrals, Residues of Dolbeault Forms, and
	Asymptotic Riemann Mappings}.
	Int. Math. Res. Not, Vol. \textbf{2017}, No. 19, pp. 5897--5918.
	
	\bibitem[FK18]{Felder:2018}
	G.~Felder and D.~ Kazhdan, \emph{Regularization of divergent integrals}. Sel. Math. New Ser. \textbf{24}, 157--186 (2018). 

	\bibitem[GL21]{Gui-Li2021}
	Z.~Gui and S.~Li, \emph{{Elliptic Trace Map on Chiral Algebras}}. arXiv:2112.14572 [math.QA].
	
		\bibitem[GS73]{Griffiths:1973}
	P.~A. Griffiths and W. Schmid, \emph{Recent developments in Hodge theory: a discussion of techniques and results}, in {\it Discrete subgroups of Lie groups and applications to moduli (Internat. Colloq., Bombay, 1973)}, pp. 31--127, Tata Inst. Fundam. Res. Stud. Math., No. 7.
	
	
	
	\bibitem[HL71]{Herrera:1971}
	M.~Herrera and D.~ Lieberman, \emph{Residues and principal values on complex spaces}. Math. Ann. \textbf{194}, 259--294 (1971). 

\bibitem[LZ21]{Li:2020regularized}
S.~Li and J.~Zhou, \emph{Regularized Integrals on Riemann Surfaces and Modular Forms}.
Commun. Math. Phys. \textbf{388} (2021), 1403--1474.

\bibitem[LZ23]{Li:2022regularized}
S.~Li and J.~Zhou, \emph{Regularized Integrals on Elliptic Curves and Holomorphic Anomaly Equations}.
Commun. Math. Phys. \textbf{401} (2023), 613--645.

	\bibitem[Pas88]{Passare:1988}
	M.~Passare, \emph{A calculus for meromorphic currents}. Journal für die reine und angewandte Mathematik \textbf{392} (1988): 37--56. 
	
		\bibitem[Pol05]{Polishchuk:2005}	
	A. Polishchuk,
	\emph{Indefinite theta series of signature (1,1) from the point of view of homological mirror symmetry}.
	Advances in Mathematics,
	Volume \textbf{196}, Issue 1,
	2005, 	Pages 1--51.

		\bibitem[Sai80]{Saito:1980}
	K.~Saito, \emph{Theory of logarithmic differential forms and logarithmic vector fields}. J. Fac. Sci. Univ. Tokyo Sec. 1A, \textbf{27} (1980), no.2., 266--291.
	

	
	\bibitem[Siv09]{Silverman:2009}
J.~Silverman, \emph{The Arithmetic of Elliptic Curves}. Graduate Texts in Mathematics, vol \textbf{106}. Springer, New York, NY, 2009.

		\bibitem[Zho23]{Zhou:GW}
	J.~Zhou, \emph{Gromov-Witten Generating Series of Elliptic Curves and Iterated Integrals of Eisenstein-Kronecker Forms}. arXiv:2310.08018 [math.CV].
	
			\bibitem[Zuc79]{Zucker:79}
	S.~M. Zucker, \emph{Hodge theory with degenerating coefficients. $L\sb{2}$\ cohomology in the Poincar\'e{} metric}, Ann. of Math. (2) {\bf 109} (1979), no.~3, 415--476.
	
	
\end{thebibliography}

\providecommand{\bysame}{\leavevmode\hbox to3em{\hrulefill}\thinspace}
\providecommand{\MR}{\relax\ifhmode\unskip\space\fi MR }
\providecommand{\MRhref}[2]{%
	\href{http://www.ams.org/mathscinet-getitem?mr=#1}{#2}
}
\providecommand{\href}[2]{#2}

\bigskip{}

\noindent{\small Yau Mathematical Sciences Center, Tsinghua University, Beijing 100084, P. R. China}


\noindent{\small Email: \tt jzhou2018@mail.tsinghua.edu.cn}

\end{document}